\documentclass[12pt]{amsart}
\usepackage{amsmath,amssymb,latexsym,esint,cite,mathrsfs}
\usepackage{verbatim,wasysym}
\usepackage[left=2.6cm,right=2.6cm,top=3cm,bottom=3cm]{geometry}
\makeatletter \@addtoreset{equation}{section} \makeatother

\usepackage{graphicx}
\usepackage{subfigure}
\usepackage{graphicx,epic,eepic}

\usepackage{color,enumitem,graphicx}
\usepackage[colorlinks=true,urlcolor=blue,
citecolor=red,linkcolor=blue,linktocpage,pdfpagelabels,
bookmarksnumbered,bookmarksopen]{hyperref}

\numberwithin{equation}{section}
\newtheorem{theorem}{Theorem}[section]
\newtheorem{lemma}[theorem]{Lemma}

\newtheorem{proposition}[theorem]{Proposition}
\newtheorem{remark}[theorem]{Remark}
\newtheorem{corollary}[theorem]{Corollary}
\numberwithin{equation}{section}

\allowdisplaybreaks

\begin{document}

\title[Stability of Hardy-Sobolev inequality]
{Stability of Hardy-Sobolev inequality involving $p$-Laplace}

\author[S. Deng]{Shengbing Deng$^{\ast}$}
\address{\noindent Shengbing Deng (Corresponding author) \newline
School of Mathematics and Statistics, Southwest University,
Chongqing 400715, People's Republic of China}\email{shbdeng@swu.edu.cn}

\author[X. Tian]{Xingliang Tian}
\address{\noindent Xingliang Tian  \newline
School of Mathematics and Statistics, Southwest University,
Chongqing 400715, People's Republic of China.}\email{xltian@email.swu.edu.cn}

\thanks{$^{\ast}$ Corresponding author}



\allowdisplaybreaks

\begin{abstract}
This paper is devoted to considering the following Hardy-Sobolev inequality
\[
\int_{\mathbb{R}^N}|\nabla u|^p \mathrm{d}x
\geq \mathcal{S}_\beta\left(\int_{\mathbb{R}^N}\frac{|u|^{p^*_\beta}}{|x|^{\beta}} \mathrm{d}x\right)^\frac{p}{p^*_\beta},\quad \forall u\in  C^\infty_0(\mathbb{R}^N),
\]
for some constant $\mathcal{S}_\beta>0$, where $1<p<N$, $0\leq \beta<p$, $p^*_\beta=\frac{p(N-\beta)}{N-p}$.
Firstly, since this problem involves quasilinear operator, we need to establish a compact embedding theorem for some suitable weighted spaces.  Moreover,  due to the Hardy term $|x|^{-\beta}$, some new estimates are established. Based on those works, we give the classification to the linearized problem related to the extremals which has its own interest such as in blow-up analysis. Then we investigate the gradient stability of above inequality by using spectral estimate combined with a compactness argument, which extends the work of Figalli and Zhang (Duke Math. J., 2022) to a weighted case.
\end{abstract}

\vspace{3mm}

\maketitle

\section{{\bfseries Introduction}}\label{sectir}

\subsection{Motivation}\label{subsectmot}
    Given $N\geq 2$ and $p\in (1,N)$, denote the homogeneous Sobolev space $\mathcal{D}^{1,p}_0(\mathbb{R}^N)$ be the closure of $C^\infty_0(\mathbb{R}^N)$ with respect to the norm
    \[
    \|u\|_{\mathcal{D}^{1,p}_0(\mathbb{R}^N)}
    :=\left(\int_{\mathbb{R}^N}|\nabla u|^p\mathrm{d}x\right)^\frac{1}{p}.
    \]
    The Sobolev inequality states as
    \begin{equation}\label{bzsi}
    \|\nabla u\|^p_{L^p(\mathbb{R}^N)}\geq \mathcal{S}\|u\|^p_{L^{p^*}(\mathbb{R}^N)},\quad \forall u\in \mathcal{D}^{1,p}_0(\mathbb{R}^N).
    \end{equation}
    for some $\mathcal{S}>0$, where $p^*:=\frac{pN}{N-p}$.
    It is well known that Aubin \cite{Au76} and Talenti \cite{Ta76} found the optimal constant and the extremals for inequality \eqref{bzsi}. Indeed, equality is achieved precisely by the functions (up to scalar multiplications)
    \begin{equation}\label{defvlz}
    V_{\lambda,z}(x):=\lambda^{\frac{N-p}{p}}V(\lambda(x-z)),\quad \mbox{for all}\quad \lambda>0,\quad z\in \mathbb{R}^N,
    \end{equation}
    where
    \begin{equation*}
    V(x)=\gamma_{N,p}(1+|x|^{\frac{p}{p-1}})^{-\frac{N-p}{p}},\quad \mbox{for some constant}\quad  \gamma_{N,p}>0, 
    \end{equation*}
    which solve the related Sobolev critical equation
    \begin{equation}\label{Ppwh}
    -\mathrm{div}(|\nabla u|^{p-2}\nabla u)= u^{p^*-1} ,\quad u>0 \quad \mbox{in}\quad \mathbb{R}^N,\quad u\in \mathcal{D}^{1,p}_0(\mathbb{R}^N).
    \end{equation}
    All the solutions to equation \eqref{Ppwh} are indeed the only ones of \eqref{defvlz}. Caffarelli et al. \cite{CGS89} proved the claim when $p=2$ without the restricted condition $u\in\mathcal{D}^{1,2}_0(\mathbb{R}^N)$. The case $p\neq 2$ has been firstly solved by Guedda and V\'{e}ron \cite{GV88} for the radial case, where the authors classified all the positive radial solutions and successively by Damascelli et al. \cite{DMMS14} when $\frac{2N}{N+2}\leq p<2$, by Damascelli and Ramaswamy \cite{DR01} and V\'{e}tois \cite{Ve16} when $1<p<2$ and finally by Sciunzi \cite{Sc16} for the remaining case, namely when $2<p<N$. And recently, Catino et al. \cite{CMR22} and Ou \cite{Ou22} proved the claim when $p\neq 2$ without the restricted condition $u\in\mathcal{D}^{1,p}_0(\mathbb{R}^N)$ for some special cases.

    Pistoia and Vaira \cite{PV21} proved that the solution $V$ of equation \eqref{Ppwh} is non-degenerate in the sense that all solutions of equation
    \begin{align}\label{Ppwhlp}
    & -{\rm div}(|\nabla V|^{p-2}\nabla \varphi)-(p-2){\rm div}(|\nabla V|^{p-4}(\nabla V\cdot\nabla \varphi)\nabla V)
    =\left(p^*-1\right)V^{p^*-2}\varphi,
    \end{align}
    in $\mathbb{R}^N$, $\varphi\in \mathcal{D}^{1,2}_{0,*}(\mathbb{R}^N)$, are linear combination of the functions
    \begin{equation*}
    Z_0(x)=\frac{N-p}{p}V+x\cdot \nabla V,\quad Z_i(x)=\frac{\partial V(x)}{\partial x_i},\quad i=1,\ldots,N.
    \end{equation*}
    Here $\mathcal{D}^{1,2}_{0,*}(\mathbb{R}^N)$ is the weighted Sobolev space which is defined as the completion of $C^\infty_c(\mathbb{R}^N)$ with respect to the norm
    \begin{align}\label{defd120*}
    \|\varphi\|_{\mathcal{D}^{1,2}_{0,*}(\mathbb{R}^N)}:=\left(\int_{\mathbb{R}^N}|\nabla V|^{p-2}|\nabla \varphi|^2\mathrm{d}x\right)^{\frac{1}{2}}.
    \end{align}
    We note that Pistoia and Vaira \cite{PV21} obtained the above conclusion by proving that $\mathcal{D}^{1,2}_{0,*}(\mathbb{R}^N)\hookrightarrow L^{2}_{0,*}(\mathbb{R}^N)$ continuously, where $L^{2}_{0,*}(\mathbb{R}^N)$ is the set of measurable functions $\varphi: \mathbb{R}^N\to \mathbb{R}$ whose norm
    \[
    \|\varphi\|_{L^{2}_{0,*}(\mathbb{R}^N)}:=\left(\int_{\mathbb{R}^N}
    V^{p^*-2}\varphi^2\mathrm{d}x\right)^{\frac{1}{2}}.
    \]
    Furthermore, Figalli and Neumayer \cite{FN19} proved that  $\mathcal{D}^{1,2}_{0,*}(\mathbb{R}^N)\hookrightarrow \hookrightarrow L^{2}_{0,*}(\mathbb{R}^N)$ compactly when $2\leq p<N$ then they showed the solutions of \eqref{Ppwhlp} in $L^{2}_{0,*}(\mathbb{R}^N)$ are linear combination of the functions
    $Z_0$ and $Z_i$ $(i=1,\ldots,N)$. Figalli and Zhang \cite{FZ22} proved that  $\mathcal{D}^{1,2}_{0,*}(\mathbb{R}^N)\hookrightarrow \hookrightarrow L^{2}_{0,*}(\mathbb{R}^N)$ compactly for all $1< p<N$ and the non-degenerate conclusion in \cite{FN19} also holds.

    Now, let us consider the stability of inequality (\ref{bzsi}). For $p=2$, Brezis and Lieb in \cite{BrL85} asked the question whether a remainder term - proportional to the quadratic distance of the function $u$ to be the manifold  $\mathcal{M}_0:=\{cV_{\lambda,z}: c\in\mathbb{R}, \lambda>0, z\in\mathbb{R}^N\}$ - can be added to the right hand side of (\ref{bzsi}). This question was answered affirmatively by Bianchi and Egnell \cite{BE91} by using spectral estimate combined with Lions' concentration and compactness principle (see \cite{Li85-2}), which reads that there exists constant $c_{\mathrm{BE}}>0$ such that
    \begin{equation}\label{defcbe}
    \int_{\mathbb{R}^N}|\nabla u|^2 \mathrm{d}x- \mathcal{S}\left(\int_{\mathbb{R}^N}|u|^{2^*} \mathrm{d}x\right)^{\frac{2}{2^*}}\geq c_{\mathrm{BE}} \inf_{v\in \mathcal{M}_0}\|\nabla u-\nabla v\|^2_{L^2(\mathbb{R}^N)}, \quad \forall u\in \mathcal{D}^{1,2}_0(\mathbb{R}^N).
    \end{equation}
    After that, this result was extended later to the biharmonic case by Lu and Wei \cite{LW00} and the case of arbitrary high order in \cite{BWW03}, and the whole fractional order case was proved in \cite{CFW13}.
    Furthermore, R\u{a}dulescu et. al \cite{RSW02} gave the remainder terms of Hardy-Sobolev inequality for exponent two.
    Wang and Willem \cite{WaWi03} studied Caffarelli-Kohn-Nirenberg inequalities (see \cite{CKN84}) with Lebesgue-type remainder terms, see also \cite{ACP05,DT23,ST18} for remainder terms of weighted Hardy inequalities.
    Recently, Wei and Wu \cite{WW22} established the stability of the profile decompositions to a special case of the Caffarelli-Kohn-Nirenberg inequality and also gave the gradient type remainder term. It is worth to mention some recent works,  Dolbeault et al. \cite{DEFS22} obtained  the lower bound estimate for sharp constant $c_{\mathrm{BE}}$ obtained in \eqref{defcbe}, and K\"{o}nig \cite{Ko22} gave its upper bound that is $c_{\mathrm{BE}}<4/(N+4)$, and  K\"{o}nig \cite{Ko22-2} proved that $c_{\mathrm{BE}}<2-2^{\frac{N-2}{N}}$ and there exist a minimizer for $c_{\mathrm{BE}}$ in $\mathcal{D}^{1,2}_0(\mathbb{R}^N)\setminus \mathcal{M}_0$.

    While for $p\neq 2$, it needs much delicate analysis to deal with the stability of inequality (\ref{bzsi}). Cianchi et al. \cite{CFMP09} first proved a stability version of Lebesgue-type for every $1<p<N$, Figalli and Neumayer \cite{FN19} proved the gradient stability for the Sobolev inequality when $p\geq 2$, Neumayer \cite{Ne20} extended the result in \cite{FN19} to all $1<p<N$. It is worth to mention that very recently, Figalli and Zhang \cite{FZ22} obtained the sharp stability of critical points of the Sobolev inequality (\ref{bzsi}) for all $1<p<N$ which reads
    \[
    \frac{\|\nabla u\|_{L^p(\mathbb{R}^N)}}{\|u\|_{L^{p^*}(\mathbb{R}^N)}}-\mathcal{S}^{\frac{1}{p}}
    \geq c_{\mathrm{FZ}} \inf_{v\in \mathcal{M}_0}\left(\frac{\|\nabla u-\nabla v\|_{L^p(\mathbb{R}^N)}}{\|\nabla u\|_{L^p(\mathbb{R}^N)}}\right)^{\max\{2,p\}},\quad \forall u\in \mathcal{D}^{1,p}_0(\mathbb{R}^N),
    \]
    for some constant $c_{\mathrm{FZ}}>0$,  
     furthermore the exponent $\max\{2,p\}$ is sharp, and this can be understood as a weak form of Bianchi-Egnell type
    \[
    \int_{\mathbb{R}^N}|\nabla u|^p \mathrm{d}x- \mathcal{S}\left(\int_{\mathbb{R}^N}|u|^{p^*} \mathrm{d}x\right)^{\frac{p}{p^*}}\geq c_{\mathrm{FZ}} \inf_{v\in \mathcal{M}_0}\|\nabla u-\nabla v\|_{L^p(\mathbb{R}^N)}^{\max\{2,p\}}.
    \]

\subsection{Problem setup and main results}\label{subsectmr}

    In present paper, we extend the works of Figalli and Zhang \cite{FZ22} to a weighted case, and R\u{a}dulescu et. al \cite{RSW02} to $p$-Laplace case, that is,
    we consider the gradient stability of the following Hardy-Sobolev inequality
\begin{align}\label{defhsi}
\int_{\mathbb{R}^N}|\nabla u|^p \mathrm{d}x
\geq \mathcal{S}_\beta\left(\int_{\mathbb{R}^N}\frac{|u|^{p^*_\beta}}{|x|^{\beta}} \mathrm{d}x\right)^\frac{p}{p^*_\beta},\quad \forall u\in \mathcal{D}^{1,p}_0(\mathbb{R}^N),
\end{align}
for some constant $\mathcal{S}_\beta>0$, where
    \begin{equation*}
    1<p<N,\quad 0<\beta<p,\quad p^*_{\beta}=\frac{p(N-\beta)}{N-p}.
    \end{equation*}
In \cite{GY00}, Ghoussoub and Yuan gave the optimal constant $\mathcal{S}_\beta$ and proved that the equality is achieved precisely by the functions (up to scalar multiplications):
    \begin{equation}\label{defvlzb}
    U_{\lambda}(x):=\frac{C_{N,p,\beta}\lambda^{\frac{N-p}{p}}}
    {(1+\lambda^{\frac{p-\beta}{p-1}}|x|^{\frac{p-\beta}{p-1}})
    ^{\frac{N-p}{p-\beta}}},
    \quad \mbox{where}\quad C_{N,p,\beta}=\left[(N-\beta)\left(\frac{N-p}{p-1}\right)^{p-1}
    \right]^{\frac{N-p}{p(p-\beta)}},
    \end{equation}
    for some $\lambda>0$.
    Moreover the functions $U_\lambda$ are the only radial solutions (up to scalings) of
    \begin{equation*}
    -\mathrm{div}(|\nabla u|^{p-2}\nabla u)=\frac{u^{p^*_{\beta}-1}}{|x|^{\beta}},\quad u>0 \quad \mbox{in}\quad \mathbb{R}^N.
    \end{equation*}
    Inspired by the work of Pistoia and Vaira \cite{PV21}, our first result concerns the non-degeneracy property of the solution $U:=U_1$ in suitable space. Let us
    denote the weighted Sobolev space $L^2_{\beta,*}(\mathbb{R}^N)$ as the set of measurable functions with respect to  the norm
    \begin{equation}\label{defd12*na}
    \|\varphi\|_{L^{2}_{\beta,*}(\mathbb{R}^N)}:=\left(\int_{\mathbb{R}^N}
    |x|^{-\beta} U^{p^*_{\beta}-2}\varphi^2\mathrm{d}x\right)^{\frac{1}{2}}.
    \end{equation}
    Denote also $C^1_{c,0}(\mathbb{R}^N)$  be the space of compactly supported functions of class $C^1$ that are zero in a neighborhood of the origin, then we define the weighted Sobolev space $\mathcal{D}^{1,2}_{0,*}(\mathbb{R}^N)$ as the completion of $C^1_{c,0}(\mathbb{R}^N)$ with respect to the norm
    \begin{equation}\label{defd12*nb}
    \|u\|_{\mathcal{D}^{1,2}_{0,*}(\mathbb{R}^N)}:=\left(\int_{\mathbb{R}^N}|\nabla U|^{p-2} |\nabla u|^2  \mathrm{d}x\right)^{\frac{1}{2}},
    \end{equation}
    and note that this norm is equivalent to \eqref{defd120*} replacing $V$ by $U$.

    \begin{remark}\label{remdefsn2}\rm
    As stated in \cite[Remark 3.1]{FZ22}, it is important for us to consider weighted that are not necessarily integrable at the origin, since $|\nabla U|^{p-2}\sim |x|^{\frac{(p-2)(1-\beta)}{p-1}}\not\in L^1(B_1)$ for $p\leq \frac{N+2(1-\beta)}{N+1-\beta}$. This is why, when defining weighted Sobolev spaces, we consider the space $C^1_{c,0}(\mathbb{R}^N)$, so that gradients vanish near zero. Of course, replacing $C^1(\mathbb{R}^N)$ by $C^1_{c,0}(\mathbb{R}^N)$ plays no role in the case $p> \frac{N+2(1-\beta)}{N+1-\beta}$.
    \end{remark}

    We first give a crucial compact embedding result as follows.
    \begin{proposition}\label{propcet}
    Suppose $1<p<N$ and $0<\beta<p$. The space $\mathcal{D}^{1,2}_{0,*}(\mathbb{R}^N)$ compactly embeds into $L^2_{\beta,*}(\mathbb{R}^N)$.
    \end{proposition}

    Based on above compact embedding theorem, we will prove that $U$ is non-degenerate in $L^2_{\beta,*}(\mathbb{R}^N)$. This leads to study the problem:
    \begin{align}\label{Ppwhla}
    -{\rm div}(|\nabla U|^{p-2}\nabla v)-(p-2){\rm div}(|\nabla U|^{p-4}(\nabla U\cdot\nabla v)\nabla U)
        = (p^*_{\beta}-1)|x|^{-\beta} U^{p^*_{\beta}-2}v,
    \end{align}
    in $\mathbb{R}^N$, $v\in L^2_{\beta,*}(\mathbb{R}^N)$.

    \begin{theorem}\label{coroPpwhlpa}
    Suppose $1<p<N$ and $0<\beta<p$. Then the space of solutions of problem (\ref{Ppwhla}) has dimension $1$ and is spanned by $(\frac{N-p}{p}U+x\cdot \nabla U)$.
    \end{theorem}

    Next is the main result of the present paper, which is the gradient stability of inequality (\ref{defhsi}).
    \begin{theorem}\label{thmprtp}
    Suppose $1<p<N$ and $0<\beta<p$. Then there exists constant $\mathcal{B}=\mathcal{B}(N,p,\beta)>0$ such that for every $u\in \mathcal{D}^{1,p}_0(\mathbb{R}^N)$, it holds that
    \begin{align}\label{hsrt}
    \int_{\mathbb{R}^N}|\nabla u|^p \mathrm{d}x
    -\mathcal{S}_\beta\left(\int_{\mathbb{R}^N}
    \frac{|u|^{p^*_\beta}}{|x|^{\beta}} \mathrm{d}x\right)^\frac{p}{p^*_\beta}
    \geq \mathcal{B}\inf_{v\in \mathcal{M}_\beta}
    \|\nabla u-\nabla v\|_{L^p(\mathbb{R}^N)}^\gamma,
    \end{align}
    where $\gamma={\max\{p,2\}}$, $\mathcal{M}_\beta:=\{cU_{\lambda}: c\in\mathbb{R}, \lambda>0\}$ is the set of extremal functions for Hardy-Sobolev inequality \eqref{defhsi}.
    \end{theorem}

    \begin{remark}\label{remp2cy}\rm
    By using perturbation methods as those in \cite{FZ22}, it is easy to verify that the exponent $\gamma$ in Theorem \ref{thmprtp} is sharp.

    In fact, if $1<p<2$, let us fix $U:=U_1$ and consider first $u_n(x):=\left(\frac{|x|}{|A_nx|}\right)^{\frac{\beta}{p^*_\beta}} U(A_nx)$ for $x\in\mathbb{R}^N\setminus\{0\}$, and $u_n(0):=U(0)$, where $A_n\in\mathbb{R}^N\times\mathbb{R}^N$ denotes the diagonal matrix
    \[
    A_n=\mathrm{diag}\left(1,\ldots,1,1+\frac{1}{n}\right).
    \]
    Then it is not difficult to check that
    \[
    \int_{\mathbb{R}^N}|\nabla u_n|^p \mathrm{d}x
    \geq (1+ n^{-1})\int_{\mathbb{R}^N}|\nabla U|^p \mathrm{d}x
    + C_1(N,p,\beta)n^{-2}
    \]
    and
    \[
    \left(\int_{\mathbb{R}^N}
    \frac{|u_n|^{p^*_\beta}}{|x|^{\beta}} \mathrm{d}x\right)^\frac{p}{p^*_\beta}
    = (1+n^{-1})^\frac{p}{p^*_\beta}
    \left(\int_{\mathbb{R}^N}
    \frac{|U|^{p^*_\beta}}{|x|^{\beta}} \mathrm{d}x\right)^\frac{p}{p^*_\beta}
    \leq (1+n^{-1})
    \left(\int_{\mathbb{R}^N}
    \frac{|U|^{p^*_\beta}}{|x|^{\beta}} \mathrm{d}x\right)^\frac{p}{p^*_\beta},
    \]
    for $n$ sufficiently large, thus the left hand side of \eqref{hsrt} behaves as $n^{-2}$. While the right side of \eqref{hsrt} behaves as $n^{-\gamma}$, that is,
    \begin{align*}
    \inf_{v\in \mathcal{M}_\beta}
    \|\nabla u-\nabla v\|_{L^p(\mathbb{R}^N)}^\gamma
    \sim & \left(\int_{\mathbb{R}^N}|\nabla (U(A_nx)-U(x))|^p \mathrm{d}x\right)^{\frac{\gamma}{p}}
    \\ \sim & \left(\int_{\mathbb{R}^N}\left|\frac{1}{n}\frac{\partial U}{\partial \bar{x}_N}\bar{x}_N\right|^p \mathrm{d}x\right)^{\frac{\gamma}{p}}
    \\ \sim & n^{-\gamma},
    \end{align*}
    hence \eqref{hsrt} cannot hold with $\gamma<2$.

    On the other hand, if $p\geq 2$, let us fix $\varphi\in C^\infty_c(B_1)$ a nontrivial function, here $B_1=B(\mathbf{0},1)$ denotes the unit ball centered at the origin, and consider now
    \[
    \tilde{u}_n(x):=\left(\frac{|x|}{|x_n+x|}\right)^{\frac{\beta}{p^*_\beta}}\left(U(x)
    +
    \varphi(x_n+x)\right),
     \]
    for $x\in\mathbb{R}^N\setminus\{-x_n\}$, and $\tilde{u}_n(-x_n):=U(-x_n)+\varphi(0)$,
     where $x_n\in\mathbb{R}^N$ is a sequence of points satisfying $|x_n|\to \infty$ as $n\to \infty$. One can check that
    \[
    \int_{\mathbb{R}^N}|\nabla \tilde{u}_n|^p \mathrm{d}x
    = \int_{\mathbb{R}^N}|\nabla U|^p \mathrm{d}x
    + \int_{\mathbb{R}^N}|\nabla \varphi|^p \mathrm{d}x
    + r_{n,1}
    \]
    and
    \[
    \int_{\mathbb{R}^N}
    \frac{|\tilde{u}_n|^{p^*_\beta}}{|x|^{\beta}} \mathrm{d}x
    = \int_{\mathbb{R}^N}
    \frac{|U|^{p^*_\beta}}{|x|^{\beta}} \mathrm{d}x
    + \int_{\mathbb{R}^N}
    \frac{|\varphi|^{p^*_\beta}}{|x|^{\beta}} \mathrm{d}x
    + r_{n,2}
    \]
    with $|r_{n,1}|+|r_{n,2}|\leq C(|\nabla U(x_n)|+U(x_n))\leq CU(x_n)\to 0$ as $n\to \infty$.  Hence, choosing a sequence $\varepsilon_n\to 0$ as $n\to \infty$ such that $U(x_n)\ll \varepsilon_n\ll 1$, the functions \[
    \hat{u}_n(x):=\left(\frac{|x|}{|x_n+x|}\right)^{\frac{\beta}{p^*_\beta}}
    \left(U(x)+\varepsilon_n
    \varphi(x_n+x)\right)
    \]
    for $x\in\mathbb{R}^N\setminus\{-x_n\}$, and $\hat{u}_n(-x_n):=U(-x_n)+\varepsilon_n\varphi(0)$
     satisfy
    \[
    \int_{\mathbb{R}^N}|\nabla \hat{u}_n|^p \mathrm{d}x
    = \int_{\mathbb{R}^N}|\nabla U|^p \mathrm{d}x
    + \varepsilon_n^p\int_{\mathbb{R}^N}|\nabla \varphi|^p \mathrm{d}x
    + o(\varepsilon_n^p)
    \]
    and
    \[
    \int_{\mathbb{R}^N}
    \frac{|\hat{u}_n|^{p^*_\beta}}{|x|^{\beta}} \mathrm{d}x
    = \int_{\mathbb{R}^N}
    \frac{|U|^{p^*_\beta}}{|x|^{\beta}} \mathrm{d}x
    + \varepsilon_n^{p^*_\beta}\int_{\mathbb{R}^N}
    \frac{|\varphi|^{p^*_\beta}}{|x|^{\beta}} \mathrm{d}x
    + o(\varepsilon_n^{p^*_\beta}).
    \]
    Thanks to these facts, one easily deduces that the left hand side of \eqref{hsrt} behaves as $\varepsilon_n^p$, while the right hand side of \eqref{hsrt} behaves as $\varepsilon^\gamma_n$. Hence \eqref{hsrt} cannot hold with $\gamma<p$.
    \end{remark}

    \begin{remark}\label{remp2c}\rm
    In this paper, to handle the general case $1<p<N$ with $0<\beta<p$ and obtain the remainder terms, as stated in \cite{FZ22}, we need to consider three cases $1<p\leq\frac{2N}{N+2-\beta}$, $\frac{2N}{N+2-\beta}< p<2$, and $2\leq p<N$ by using different arguments.
    In fact, note that $p^*_{\beta}=\frac{p(N-\beta)}{N-p}\leq 2$ implies $p\leq \frac{2N}{N+2-\beta}$, $p^*_{\beta}>2$ implies $p>\frac{2N}{N+2-\beta}$, respectively. Moreover, $\frac{2N}{N+2-\beta}<2$ is equivalent to $\beta<2$. Therefore, when $1<p<2$, we will split this problem into two cases:
    \begin{itemize}
    \item[$(i)$]
    $1<p\leq \frac{2N}{N+2-\beta}$;
    \item[$(ii)$]
    $\frac{2N}{N+2-\beta}< p<2$.
    \end{itemize}
    The reason why we consider the above two cases is that, it needs some appropriate algebraic inequalities which requires to compare $p$ and $p^*_{\beta}$ with $2$, see Lemmas \ref{lemui1p} and \ref{lemui1p*l}. However,  $2\leq p<N$ implies $p^*_{\beta}> 2$, and the $\mathcal{D}^{1,p}_{0}(\mathbb{R}^N)$ norm is stronger than any weighted $L^2_{\beta,*}(\mathbb{R}^N)$ norm (see \eqref{defd12*na}), so that we can deal with this case directly.
    \end{remark}

\subsection{Structure of the paper}\label{subsect:structrue}

In Section \ref{sectce}, we prove the compactness that is we give the proof of Proposition \ref{propcet}. Section \ref{sectndr} is devoted to proving the non-degeneracy of extremal function $U$. Then in Section \ref{sectspana}, we give the spectral analysis with the help of compactness. In Section \ref{sectpromr}, we first prove a local variant of Theorem \ref{thmprtp} then complete the proof of Theorem \ref{thmprtp} by Lions' concentration and compactness principle.  Finally, we collect some technical estimates in Appendixes \ref{sectpls} and \ref{appsue}.

    \noindent{\bfseries Notations.}
    Throughout this paper, we denote $B_\rho:=B(\mathbf{0},\rho)$ be the ball with radius $\rho$ centered at the origin.
Moreover, $c$, $C$, $C'$ and $C_i$ are indiscriminately used to denote various absolutely positive constants from line to line. $a\sim b$ means that $C'b\leq a\leq Cb$.

\section{Compact embedding theorem}\label{sectce}

    We first establish the following Poincar\'{e} type inequalities which will be useful later.
    \begin{lemma}\label{lemcztj}
    Let $1<p<N$ and $0<\beta<p$. Then there exists $C=C(N,p,\beta)>0$ such that for any $\varphi\in \mathcal{D}^{1,2}_{0,*}(\mathbb{R}^N)\cap L^2_{\beta,*}(\mathbb{R}^N)$, it holds that
    \begin{align}\label{continouseb}
    \int_{\mathbb{R}^N}|\nabla U|^{p-2}|\nabla \varphi|^2\mathrm{d}x
    \geq C \int_{\mathbb{R}^N}
    |x|^{-\beta}|U|^{p^*_\beta-2}\varphi^2\mathrm{d}x.
    \end{align}
    Also, there exists $\vartheta=\vartheta(N,p,\beta)>0$ such that, for any $\rho\in (0,1)$, we have
    \begin{align}\label{continouseb1}
    \int_{\mathbb{R}^N}|\nabla U|^{p-2}|\nabla \varphi|^2\mathrm{d}x
    \geq \frac{C}{\rho^{\vartheta}}\int_{B_\rho}
    |x|^{-\beta}|U|^{p^*_\beta-2}\varphi^2\mathrm{d}x,
    \end{align}
    and
    \begin{align}\label{continouseb2}
    \int_{\mathbb{R}^N}|\nabla U|^{p-2}|\nabla \varphi|^2\mathrm{d}x
    \geq C|\log \rho|^2\int_{\mathbb{R}^N\setminus B_{\frac{1}{\rho}}}
    |x|^{-\beta}|U|^{p^*_\beta-2}\varphi^2\mathrm{d}x.
    \end{align}
    \end{lemma}

    \begin{proof}
    To prove \eqref{continouseb}, we can assume by approximation that $\varphi\in C^1_{c,0}(\mathbb{R}^N)$ as in \cite[Lemma 4.2]{FN19}.  Define
    \[
    \mathfrak{F}(u):=\int_{\mathbb{R}^N}|\nabla u|^p \mathrm{d}x
    -\mathcal{S}_\beta\left(\int_{\mathbb{R}^N}
    \frac{|u|^{p^*_\beta}}{|x|^{\beta}} \mathrm{d}x\right)^\frac{p}{p^*_\beta}.
    \]
    We know $U$ is a local minimum of the functional $\mathfrak{F}$, then
    \begin{align*}
    0\leq & \frac{d^2}{d\epsilon^2}\Big|_{\epsilon=0} \mathfrak{F}(U+\epsilon \varphi)
    \\ = & p \int_{\mathbb{R}^N}|\nabla U|^{p-2}|\nabla \varphi|^2\mathrm{d}x
    + p(p-2) \int_{\mathbb{R}^N}|\nabla U|^{p-4}(\nabla U\cdot\nabla\varphi)^2\mathrm{d}x
    \\ & - \mathcal{S}_\beta
    \Bigg[
    p\left(\frac{p}{p^*_\beta}-1\right)
    \left(\int_{\mathbb{R}^N}|x|^{-\beta}|U|^{p^*_\beta} \mathrm{d}x\right)^{\frac{p}{p^*_\beta}-2}
    \left(\int_{\mathbb{R}^N}
    |x|^{-\beta}|U|^{p^*_\beta-2}U\varphi\mathrm{d}x\right)^2
    \\ & \quad\quad + p(p^*_\beta-1)
    \left(\int_{\mathbb{R}^N}|x|^{-\beta}|U|^{p^*_\beta} \mathrm{d}x\right)^{\frac{p}{p^*_\beta}-1}
    \int_{\mathbb{R}^N}
    |x|^{-\beta}|U|^{p^*_\beta-2}\varphi^2\mathrm{d}x
    \Bigg].
    \end{align*}
    Noting that
    \begin{align*}
    \int_{\mathbb{R}^N}|\nabla U|^{p-4}(\nabla v\cdot\nabla\varphi)^2\mathrm{d}x
    \leq \int_{\mathbb{R}^N}|\nabla U|^{p-2}|\nabla \varphi|^2\mathrm{d}x,
    \end{align*}
    and
    \begin{align*}
    \left(\int_{\mathbb{R}^N}|x|^{-\beta}|U|^{p^*_\beta} \mathrm{d}x\right)^{\frac{p}{p^*_\beta}-2}
    \left(\int_{\mathbb{R}^N}
    |x|^{-\beta}|U|^{p^*_\beta-2}U\varphi\mathrm{d}x\right)^2\geq 0,
    \end{align*}
    these imply that
    \begin{align*}
    0\leq & p(p-1) \int_{\mathbb{R}^N}|\nabla U|^{p-2}|\nabla \varphi|^2\mathrm{d}x
    \\ & - \mathcal{S}_\beta p(p^*_\beta-1)
    \left(\int_{\mathbb{R}^N}|x|^{-\beta}|U|^{p^*_\beta} \mathrm{d}x\right)^{\frac{p}{p^*_\beta}-1}
    \int_{\mathbb{R}^N}
    |x|^{-\beta}|U|^{p^*_\beta-2}\varphi^2\mathrm{d}x
    \\ = & p(p-1) \int_{\mathbb{R}^N}|\nabla U|^{p-2}|\nabla \varphi|^2\mathrm{d}x
    - p(p^*_\beta-1)
    \int_{\mathbb{R}^N}
    |x|^{-\beta}|U|^{p^*_\beta-2}\varphi^2\mathrm{d}x,
    \end{align*}
    due to
    \[
    \int_{\mathbb{R}^N}|x|^{-\beta}|U|^{p^*_\beta} \mathrm{d}x
    =\mathcal{S}_\beta^{\frac{p^*_\beta}{p^*_\beta-p}}.
    \]
    Thus \eqref{continouseb} holds.

    To prove \eqref{continouseb1}, we can also assume by approximation that $\varphi\in C^1_{c,0}(\mathbb{R}^N)$ and we apply the Sobolev inequality with radial weights. More precisely, we have
    \begin{align}\label{wsit}
    \int_{\mathbb{R}^N}|\nabla U|^{p-2}|\nabla \varphi|^2\mathrm{d}x
    \geq & c\left(\int_{B_1}
    |x|^{-\frac{q\beta}{2}}|\varphi|^q\mathrm{d}x\right)^{\frac{2}{q}},
    \end{align}
    where $q=q(N,p,\beta)>2$.
    In fact, when $0<\beta\leq 1$, $|\nabla U|^{p-2}\geq  c$ inside $B_1$, it follows by Hardy-Sobolev inequality (see \cite{GY00}) that
    \[
    \int_{\mathbb{R}^N}|\nabla U|^{p-2}|\nabla \varphi|^2\mathrm{d}x
    \geq  c_1\int_{B_1}|\nabla \varphi|^2\mathrm{d}x
    \geq  c_2\left(\int_{B_1}
    |x|^{-\frac{q\beta}{2}}|\varphi|^q\mathrm{d}x\right)^{\frac{2}{q}},
    \]
    where $q=q(N,p,\beta)=\frac{2N}{N-2+\beta}>2$, thus \eqref{wsit} holds. Otherwise, if $1< \beta<p$, $|\nabla U|^{p-2}\geq  c|x|^{\frac{(2-p)(\beta-1)}{p-1}}$ inside $B_1$, then applying \eqref{continouseb} and the Sobolev inequality with radial weights (see \cite[Section 2.1]{Ma85}) we deduce
    \begin{align*}
    \int_{\mathbb{R}^N}|\nabla U|^{p-2}|\nabla \varphi|^2\mathrm{d}x
    \geq & c_1\int_{\mathbb{R}^N}(|x|^{-\beta}|U|^{p^*_\beta-2}\varphi^2
    +|\nabla U|^{p-2}|\nabla \varphi|^2)\mathrm{d}x
    \\ \geq & c_2\int_{B_1}
    |x|^{-\beta}\left(\varphi^2+|x|^{1+\frac{\beta-1}{p-1}}|\nabla \varphi|^2\right)\mathrm{d}x
    \\ \geq & c_2\int_{B_1}
    |x|^{-\beta}\left(\varphi^2+|x|^{2}|\nabla \varphi|^2\right)\mathrm{d}x
    \\ \geq & c_3\int_{B_1}
    \left(\left(|x|^{-\frac{\beta}{2}}\varphi\right)^2+|x|^{2}\left|\nabla \left(|x|^{-\frac{\beta}{2}}\varphi\right)\right|^2\right)\mathrm{d}x
    \\ \geq & c_4\left(\int_{B_1}
    |x|^{-\frac{q\beta}{2}}|\varphi|^q\mathrm{d}x\right)^{\frac{2}{q}},
    \end{align*}
    where $q=q(N)>2$, then \eqref{wsit} holds.
    Therefore, by H\"{o}lder inequality, for any $\rho\in (0,1)$ we obtain
    \begin{align*}
    \int_{B_\rho}
    |x|^{-\beta}|U|^{p^*_\beta-2}\varphi^2\mathrm{d}x
    \leq & C_1 \int_{B_\rho}
    |x|^{-\beta}\varphi^2\mathrm{d}x
    \\ \leq & C_2 \rho^{N(1-\frac{q}{2})}\left(\int_{B_\rho}
    |x|^{-\frac{q\beta}{2}}|\varphi|^q\mathrm{d}x\right)^{\frac{2}{q}}
    \\ \leq & C_3 \rho^{N(1-\frac{q}{2})}\int_{\mathbb{R}^N}|\nabla U|^{p-2}|\nabla \varphi|^2\mathrm{d}x,
    \end{align*}
    as desired.

    To prove \eqref{continouseb2}, we define
    \begin{eqnarray*}
    \eta_\rho(x):=
    \left\{ \arraycolsep=1.5pt
       \begin{array}{ll}
        0,\ \ &{\rm for}\ \ |x|<\rho^{-\frac{1}{2}},\\[2mm]
        \frac{2\log |x|-|\log \rho|}{|\log \rho|},\ \ &{\rm for}\ \  \rho^{-\frac{1}{2}}\leq |x|\leq \rho^{-1},\\[2mm]
        1,\ \ &{\rm for}\ \ |x|> \rho^{-1},
        \end{array}
    \right.
    \end{eqnarray*}
    and $\phi_\rho:=\eta_\rho\varphi\in C^1_{c,0}(\mathbb{R}^N)$.
    We note that, since
    \begin{equation*}
    U(r)\sim
    (1+r^{\frac{p-\beta}{p-1}})
    ^{-\frac{N-p}{p-\beta}},\quad
    U'(r)\sim
    (1+r^{\frac{p-\beta}{p-1}})
    ^{-\frac{N-\beta}{p-\beta}}r^{\frac{1-\beta}{p-1}},
    \end{equation*}
    then
    \begin{equation*}
    |U|^{p^*_\beta-2}\sim
    (1+r^{\frac{p-\beta}{p-1}})
    ^{-\frac{(N-\beta)(p-2)}{p-\beta}-2},\quad
    |\nabla U|^{p-2}\sim
    (1+r^{\frac{p-\beta}{p-1}})
    ^{-\frac{(N-\beta)(p-2)}{p-\beta}}r^{\frac{(1-\beta)(p-2)}{p-1}}.
    \end{equation*}
    Thanks to Fubini's theorem and using polar coordinates,
    \begin{align*}
    \int_{\mathbb{R}^N}
    |x|^{-\beta}|U|^{p^*_\beta-2}|\phi_\rho|^2\mathrm{d}x
    \leq & C_1\int_{\mathbb{S}^{N-1}}\int^{+\infty}_0    r^{N-1-\beta}(1+r^{\frac{p-\beta}{p-1}})
    ^{-\frac{(N-\beta)(p-2)}{p-\beta}-2}|\phi_\rho(r\theta)|^2
    \mathrm{d}r\mathrm{d}\theta
    \\ \leq & C_2\int_{\mathbb{S}^{N-1}}\int^{+\infty}_0    r^{N-1-\beta}(1+r^{\frac{p-\beta}{p-1}})
    ^{-\frac{(N-\beta)(p-2)}{p-\beta}-2}
    \\ & \quad \times
    \int^{+\infty}_r |\phi_\rho(t\theta)||\nabla\phi_\rho(t\theta)|
    \mathrm{d}t\mathrm{d}r\mathrm{d}\theta
    \\ \leq & C_3\int_{\mathbb{S}^{N-1}}\int^{+\infty}_0 \int^{t}_0    r^{N-1-\beta}(1+r^{\frac{p-\beta}{p-1}})
    ^{-\frac{(N-\beta)(p-2)}{p-\beta}-2}
    \\ & \quad \times
    |\phi_\rho(t\theta)||\nabla\phi_\rho(t\theta)|
    \mathrm{d}r\mathrm{d}t\mathrm{d}\theta
    \\ \leq & C_4\int_{\mathbb{S}^{N-1}}\int^{+\infty}_0     t^{N-\beta}(1+t^{\frac{p-\beta}{p-1}})
    ^{-\frac{(N-\beta)(p-2)}{p-\beta}-2}
    |\phi_\rho(t\theta)||\nabla\phi_\rho(t\theta)|
    \mathrm{d}t\mathrm{d}\theta.
    \end{align*}
    Thus, by Cauchy-Schwarz inequality we obtain
    \begin{align*}
    \int_{\mathbb{R}^N}
    |x|^{-\beta}|U|^{p^*_\beta-2}|\phi_\rho|^2\mathrm{d}x
    \leq & C_5\left(\int_{\mathbb{S}^{N-1}}\int^{+\infty}_0     t^{N+1-\beta}(1+t^{\frac{p-\beta}{p-1}})
    ^{-\frac{(N-\beta)(p-2)}{p-\beta}-2}|\nabla\phi_\rho(t\theta)|^2
    \mathrm{d}t\mathrm{d}\theta\right)^{\frac{1}{2}}
    \\ & \quad \times
    \left(\int_{\mathbb{S}^{N-1}}\int^{+\infty}_0     t^{N-1-\beta}(1+t^{\frac{p-\beta}{p-1}})
    ^{-\frac{(N-\beta)(p-2)}{p-\beta}-2}|\phi_\rho(t\theta)|^2
    \mathrm{d}t\mathrm{d}\theta\right)^{\frac{1}{2}},
    \end{align*}
    and since the last term in the right hand coincides with $\|\phi_\rho\|_{L^2_{\beta,*}(\mathbb{R}^N)}$ (up to a multiplicative constant), we conclude that
    \begin{align}\label{pfleql}
    \int_{\mathbb{R}^N}
    |x|^{-\beta}|U|^{p^*_\beta-2}|\phi_\rho|^2\mathrm{d}x
    \leq & C_6\int_{\mathbb{S}^{N-1}}\int^{+\infty}_0     t^{N+1-\beta}(1+t^{\frac{p-\beta}{p-1}})
    ^{-\frac{(N-\beta)(p-2)}{p-\beta}-2}|\nabla\phi_\rho(t\theta)|^2
    \mathrm{d}t\mathrm{d}\theta
    \nonumber\\ \leq & C_7 \int_{\mathbb{R}^N}|x|^{2-\beta}(1+|x|^{\frac{p-\beta}{p-1}})
    ^{-\frac{(N-\beta)(p-2)}{p-\beta}-2}|\nabla\phi_\rho(x)|^2
    \mathrm{d}x.
    \end{align}
    Then we have
    \begin{align*}
    \int_{\mathbb{R}^N\setminus B_{\frac{1}{\rho}}}
    |x|^{-\beta}|U|^{p^*_\beta-2}\varphi^2\mathrm{d}x
    \leq & C_1\int_{\mathbb{R}^N}
    |x|^{-\beta}|U|^{p^*_\beta-2}|\phi_\rho|^2\mathrm{d}x
    \\ \leq & C_2 \int_{\mathbb{R}^N}|x|^{2-\beta}(1+|x|^{\frac{p-\beta}{p-1}})
    ^{-\frac{(N-\beta)(p-2)}{p-\beta}-2}|\nabla \phi_\rho|^2\mathrm{d}x
    \\ \leq & C_3 \int_{\mathbb{R}^N}|x|^{2-\beta}(1+|x|^{\frac{p-\beta}{p-1}})
    ^{-\frac{(N-\beta)(p-2)}{p-\beta}-2}\eta_\rho^2|\nabla\varphi|^2
    \mathrm{d}x
    \\ & + C_3 \int_{\mathbb{R}^N}|x|^{2-\beta}(1+|x|^{\frac{p-\beta}{p-1}})
    ^{-\frac{(N-\beta)(p-2)}{p-\beta}-2}|\nabla\eta_\rho|^2
    \varphi^2\mathrm{d}x
    \\ \leq & C_3 \int_{\mathbb{R}^N\setminus B_{\frac{1}{\rho^{1/2}}}}|x|^{2-\beta}(1+|x|^{\frac{p-\beta}{p-1}})
    ^{-\frac{(N-\beta)(p-2)}{p-\beta}-2}|\nabla\varphi|^2
    \mathrm{d}x
    \\ & + 4C_3 |\log \rho|^{-2}
    \int_{B_{\frac{1}{\rho}}\setminus B_{\frac{1}{\rho^{1/2}}}}|x|^{-\beta}(1+|x|^{\frac{p-\beta}{p-1}})
    ^{-\frac{(N-\beta)(p-2)}{p-\beta}-2}\varphi^2
    \mathrm{d}x
    \\ \leq & C_4 \rho^{\frac{p-\beta}{2(p-1)}}\int_{\mathbb{R}^N\setminus B_{\frac{1}{\rho^{1/2}}}}|\nabla U|^{p-2}|\nabla\varphi|^2
    \mathrm{d}x
    \\ & + C_4 |\log \rho|^{-2}\int_{B_{\frac{1}{\rho}}\setminus B_{\frac{1}{\rho^{1/2}}}}|x|^{-\beta}
    |U|^{p^*_\beta-2}\varphi^2\mathrm{d}x.
    \end{align*}
    Thus, combining with \eqref{continouseb} we deduce that
    \begin{align*}
    \int_{\mathbb{R}^N\setminus B_{\frac{1}{\rho}}}
    |x|^{-\beta}|U|^{p^*_\beta-2}\varphi^2\mathrm{d}x
    \leq &   C_4 \rho^{\frac{p-\beta}{2(p-1)}}\int_{\mathbb{R}^N}|\nabla U|^{p-2}|\nabla\varphi|^2
    \mathrm{d}x
    \\ & + C_4 |\log \rho|^{-2}\int_{\mathbb{R}^N}|x|^{-\beta}
    |U|^{p^*_\beta-2}\varphi^2\mathrm{d}x
    \\ \leq & C_5 |\log \rho|^{-2}\int_{\mathbb{R}^N}|\nabla U|^{p-2}|\nabla\varphi|^2
    \mathrm{d}x,
    \end{align*}
    due to $\frac{p-\beta}{2(p-1)}>0$ and $\rho\in (0,1)$, \eqref{continouseb2} holds.
    \end{proof}

\subsection{\bf Proof of Proposition \ref{propcet}.}
    Let $\{\varphi_n\}$ be a sequence of functions in $\mathcal{D}^{1,2}_{0,*}(\mathbb{R}^N)$ with uniformly bounded norm. It follows from \eqref{continouseb} that $\|\varphi_n\|_{L^2_{\beta,*}(\mathbb{R}^N)}$ are uniformly bounded as well.

    Since both $|\nabla U|^{p-2}$ and $|x|^{-\beta}|U|^{p^*_\beta-2}$ are locally bounded away from zero and infinity in $\mathbb{R}^N\setminus\{\mathbf{0}\}$, by Rellich-Kondrachov Theorem and a diagonal argument we deduce that, up to a subsequence, there exists $\varphi\in \mathcal{D}^{1,2}_{0,*}(\mathbb{R}^N)\cap L^2_{\beta,*}(\mathbb{R}^N)$ such that $\varphi_n\rightharpoonup \varphi$ in $\mathcal{D}^{1,2}_{0,*}(\mathbb{R}^N)\cap L^2_{\beta,*}(\mathbb{R}^N)$ and $\varphi_n\to \varphi$ locally in $L^2_{\beta,*}(\mathbb{R}^N\setminus\{\mathbf{0}\})$.

    Also, it follows from \eqref{continouseb1} and \eqref{continouseb2} that, for any $\rho\in (0,1)$,
    \begin{align*}
    \int_{B_\rho}
    |x|^{-\beta}|U|^{p^*_\beta-2}\varphi^2\mathrm{d}x
    \leq C\rho^{\vartheta},
    \quad
    \int_{B_{\frac{1}{\rho}}}
    |x|^{-\beta}|U|^{p^*_\beta-2}\varphi^2\mathrm{d}x
    \leq \frac{C}{|\log \rho|^2}.
    \end{align*}
    We conclude the proof by defining the compact set $K_\rho:=\overline{B_{\frac{1}{\rho}}}\setminus B_\rho$ and applying the strong convergence of $\varphi_n$ in $K_\rho$, together with the arbitrariness of $\rho$ (that can be chosen arbitrarily small).
    \qed

    As mentioned in \cite{FZ22}, because of the crucial inequality \ref{uinx2pl}, we shall see that Proposition \ref{propcet} allows us to deal with the case $\frac{2N}{N+2-\beta}< p<2$ when we show the stability of Hardy-Sobolev inequality \eqref{defhsi}. However, when $1<p\leq\frac{2N}{N+2-\beta}$ which implies $p<p^*_{\beta}\leq 2$, we will need a much more delicate compactness result that we now present.

    \begin{lemma}\label{propcetl}
    Let $1<p\leq \frac{2N}{N+2-\beta}$ with $0<\beta<p$. Let $\{v_n\}$ be a sequence of functions in $\mathcal{D}^{1,p}_{0}(\mathbb{R}^N)$ satisfying
    \begin{align}\label{propcetll}
    & \int_{\mathbb{R}^N}
    (|\nabla U|+\varepsilon_n|\nabla v_n|)^{p-2}|\nabla v_n|^2\mathrm{d}x\leq 1,
    \end{align}
    where $\varepsilon_n\in (0,1)$ is a sequence of positive numbers converging to $0$. Then, up to a subsequence, $v_n$ convergence weakly in $\mathcal{D}^{1,p}_{0}(\mathbb{R}^N)$ to some $v\in\mathcal{D}^{1,p}_{0}(\mathbb{R}^N)\cap L^2_{\beta,*}(\mathbb{R}^N)$. Also, given any constant $C_1\geq 0$ it holds
    \begin{align}\label{propcetlc}
    \int_{\mathbb{R}^N}|x|^{-\beta}
    \frac{(U+C_1\varepsilon_nv_n)^{p^*_{\beta}}}
    {U^2+|\varepsilon_nv_n|^2}|v_n|^2\mathrm{d}x
    \to \int_{\mathbb{R}^N}|x|^{-\beta}
    U^{p^*_{\beta}-2}|v|^2\mathrm{d}x.
    \end{align}
    \end{lemma}

    \begin{proof}
    We follow the arguments as those in \cite[Lemma 3.4]{FZ22}.
    Up to replacing $v_n$ by $|v_n|$, we can assume that $v_n\geq 0$. Note that $p<p^*_\beta\leq 2$ under our assumption $1<p\leq \frac{2N}{N+2-\beta}$ with $0<\beta<p$.

    Observe that, under H\"{o}lder inequality,
    \begin{align*}
    \int_{\mathbb{R}^N}|\nabla v_n|^p\mathrm{d}x
    \leq & \left(\int_{\mathbb{R}^N}
    (|\nabla U|+\varepsilon_n|\nabla v_n|)^{p-2}|\nabla v_n|^2\mathrm{d}x\right)^{\frac{p}{2}}
    \left(\int_{\mathbb{R}^N}
    (|\nabla U|+\varepsilon_n|\nabla v_n|)^p\mathrm{d}x\right)^{1-\frac{p}{2}}
    \\ \leq & \left(\int_{\mathbb{R}^N}
    (|\nabla U|+\varepsilon_n|\nabla v_n|)^{p-2}|\nabla v_n|^2\mathrm{d}x\right)^{\frac{p}{2}}
    \left(\|U\|^p_{\mathcal{D}^{1,p}_0(\mathbb{R}^N)}
    +\varepsilon_n^p\int_{\mathbb{R}^N}
    |\nabla v_n|^p\mathrm{d}x\right)^{1-\frac{p}{2}}
    \end{align*}
    that combined with \eqref{propcetll} gives
    \begin{align}\label{propcetllb}
    \left(\int_{\mathbb{R}^N}
    |\nabla v_n|^p\mathrm{d}x\right)^{\frac{2}{p}}
    \leq C \int_{\mathbb{R}^N}
    (|\nabla U|+\varepsilon_n|\nabla v_n|)^{p-2}|\nabla v_n|^2\mathrm{d}x
    \leq C.
    \end{align}
    Thus, up to a subsequence, $v_n$ converges weakly in $\mathcal{D}^{1,p}_0(\mathbb{R}^N)$ and also a.e. to some function $v\in \mathcal{D}^{1,p}_0(\mathbb{R}^N)$. Hence, to conclude the proof, we need to show the validity of \eqref{propcetlc}.

    We first prove it under the assumption that $\varepsilon_n v_n\leq \zeta U$ with some small constant $\zeta=\zeta(N,p,\beta,C_1)\in (0,1)$ be determined. Later, we will remove this assumption.

    $\bullet$ {\em \color{blue} Step 1: prove of \eqref{propcetlc} when $\varepsilon_n v_n\leq \zeta U$}. Since $\varepsilon_n v_n$ is bounded by $\zeta U\leq U$, we have that $1+\frac{\varepsilon_n v_n}{U}\leq 2$, thus
    \begin{align*}
    \int_{\mathbb{R}^N}|x|^{-\beta}
    (U+\varepsilon_nv_n)^{p^*_{\beta}-2}|v_n|^2\mathrm{d}x
    = & \int_{\mathbb{R}^N}|x|^{-\beta}
    U^{p^*_{\beta}-2}\left(1+\frac{\varepsilon_n v_n}{U}\right)^{p^*_{\beta}-2}|v_n|^2\mathrm{d}x
    \\ \leq & 2^{p^*_\beta-p}\int_{\mathbb{R}^N}|x|^{-\beta}
    U^{p^*_{\beta}-2}\left(1+\frac{\varepsilon_n v_n}{U}\right)^{p-2}|v_n|^2\mathrm{d}x.
    \end{align*}
    Recall that
    \begin{equation}\label{UnUe}
    |U|^{p^*_\beta-2}\sim
    (1+r^{\frac{p-\beta}{p-1}})
    ^{-\frac{(N-\beta)(p-2)}{p-\beta}-2},\quad
    |\nabla U|^{p-2}\sim
    (1+r^{\frac{p-\beta}{p-1}})
    ^{-\frac{(N-\beta)(p-2)}{p-\beta}}r^{\frac{(1-\beta)(p-2)}{p-1}}.
    \end{equation}
    Moreover, from the following Hardy-Poincar\'{e} inequality, see Lemma \ref{lem:A-hpiw}, we deduce that:  For any $p>1$ and $\xi\geq 1$, and any compactly supported function $w\in \mathcal{D}^{1,p}_0(\mathbb{R}^N)$, one has
    \begin{align}\label{HPI}
    \int_{\mathbb{R}^N}|w|^{p}
    |x|^{-\beta}\left[\left(1+|x|^{\frac{p-\beta}{p-1}}
    \right)^{p-1}\right]^{\xi-1}
    \mathrm{d}x
    \leq & C\int_{\mathbb{R}^N}|\nabla w|^{p}
    \left[\left(1+|x|^{\frac{p-\beta}{p-1}}\right)^{p-1}\right]^{\xi}
    \mathrm{d}x.
    \end{align}
    By approximation, we can apply this inequality with
    \begin{align*}
    \xi=1+\frac{(2-p^*_\beta)(N-p)}{(p-1)(p-\beta)}
    \quad \mbox{and}\quad w=w_n:=\left(\left(1+\frac{\varepsilon_n v_n}{U}\right)^{p-2}|v_n|^2\right)^{\frac{1}{p}}.
    \end{align*}
    Thus, by $|U|^{p^*_\beta-2}\sim \left[\left(1+r^{\frac{p-\beta}{p-1}}\right)^{p-1}\right]^{\xi-1}$, we get
    \begin{small}
    \begin{align}\label{propcetlllx}
    & \int_{\mathbb{R}^N}|x|^{-\beta}
    (U+\varepsilon_nv_n)^{p^*_{\beta}-2}|v_n|^2\mathrm{d}x
    \nonumber\\ \leq & C\int_{\mathbb{R}^N}|\nabla w_n|^{p}
    \left[\left(1+|x|^{\frac{p-\beta}{p-1}}\right)^{p-1}\right]^{\xi}
    \mathrm{d}x
    \nonumber\\ \leq & C \int_{\mathbb{R}^N}|U|^{p^*_\beta-2}
    \left(1+|x|^{\frac{p-\beta}{p-1}}\right)^{p-1}
    \nonumber\\ & \quad \times \left[
    \left(1+\frac{\varepsilon_n v_n}{U}\right)^{-2}|v_n|^2
    \left(\frac{\varepsilon_n v_n |\nabla U|}{U^2}+\frac{\varepsilon_n |\nabla v_n|}{U}\right)^{p}
    +\left(1+\frac{\varepsilon_n v_n}{U}\right)^{p-2}|v_n|^{2-p}|\nabla v_n|^p
    \right]
    \mathrm{d}x
    \nonumber\\ \leq & C \int_{\mathbb{R}^N}|U|^{p^*_\beta-2}
    \left(1+|x|^{\frac{p-\beta}{p-1}}\right)^{p-1}
    \left[
    |v_n|^2\left(\frac{\zeta|\nabla U|}{U}+\frac{\varepsilon_n |\nabla v_n|}{U}\right)^{p}+|v_n|^{2-p}|\nabla v_n|^p
    \right]
    \mathrm{d}x,
    \end{align}
    \end{small}
    where, in the last inequality, we have used that $0\leq \frac{\varepsilon_n v_n}{U}\leq \zeta<1$.

    We now apply \eqref{nis} to the last integrand in \eqref{propcetlllx} with $\varepsilon=\varepsilon_n$, $r=|x|$, $a=|v_n|$, $b=|\nabla v_n|$. In this way, thanks to \eqref{propcetlllx} and since $U+\varepsilon_n v_n\leq 2U$, we deduce that for any $\varepsilon_0>0$ there exists $\zeta=\zeta(\varepsilon_0)\in (0,1)$ such that
    \begin{align*}
    \int_{\mathbb{R}^N}|x|^{-\beta}
    U^{p^*_{\beta}-2}|v_n|^2\mathrm{d}x
    \leq & 2^{2-p^*_\beta}\int_{\mathbb{R}^N}|x|^{-\beta}
    (U+\varepsilon_n v_n)^{p^*_{\beta}-2}|v_n|^2\mathrm{d}x
    \\ \leq & C\int_{\mathbb{R}^N}|\nabla w_n|^{p}
    \left[\left(1+|x|^{\frac{p-\beta}{p-1}}\right)^{p-1}\right]^{\xi}
    \mathrm{d}x
    \\ \leq & C\varepsilon_0\int_{\mathbb{R}^N}|x|^{-\beta}
    U^{p^*_{\beta}-2}|v_n|^2\mathrm{d}x
    \\ & + C(\varepsilon_0)\int_{\mathbb{R}^N}
    (|\nabla U|+\varepsilon_n|\nabla v_n|)^{p-2}|\nabla v_n|^2\mathrm{d}x.
    \end{align*}
    Thus, fixing $\varepsilon_0$ small enough such that $C\varepsilon_0\leq \frac{1}{2}$, it follows from \eqref{propcetll} and the inequality above that
    \begin{align}\label{propcetlllc}
    & \int_{\mathbb{R}^N}|x|^{-\beta}
    U^{p^*_{\beta}-2}|v_n|^2\mathrm{d}x
    + \int_{\mathbb{R}^N}|\nabla w_n|^{p}
    \left[\left(1+|x|^{\frac{p-\beta}{p-1}}\right)^{p-1}\right]^{\xi}
    \mathrm{d}x
    \nonumber\\ \leq & C \int_{\mathbb{R}^N}
    (|\nabla U|+\varepsilon_n|\nabla v_n|)^{p-2}|\nabla v_n|^2\mathrm{d}x
    \leq C.
    \end{align}
    In particular, the sequence $w_n$ is uniformly and locally bounded in $\mathcal{D}^{1,p}_0(\mathbb{R}^N)\subset L^{p^*}_\beta(\mathbb{R}^N)$. Since $1+\frac{\varepsilon_n v_n}{U}\sim 1$, this implies that $|v_n|^{\frac{2}{p}}$ is locally in $L^{p^*}_\beta(\mathbb{R}^N)$. Combining this higher integrability estimate with the a.e. convergence of $v_n$ to $v$, by dominated convergence we deduce that, for any $R>1$,
    \begin{align}\label{propcetlcB}
    \int_{B_R}|x|^{-\beta}
    \frac{(U+C_1\varepsilon_nv_n)^{p^*_{\beta}}}
    {U^2+|\varepsilon_nv_n|^2}|v_n|^2\mathrm{d}x
    \to \int_{B_R}|x|^{-\beta}
    U^{p^*_{\beta}-2}|v|^2\mathrm{d}x,
    \end{align}
    as $n\to \infty$ (recall that $\varepsilon_n\to 0$). 

    Also, since $1<p\leq \frac{2N}{N+2-\beta}$ and $0<\beta<p$ it follows that $N>\max\{\beta,2-\beta\}$, therefore
    \[
    \frac{(p-N)(p^*_\beta-2-p)}{p-1}-\beta+N=\frac{N-2p+\beta(N-p+1)}{p-1}
    =\frac{N+\beta N+\beta-(\beta+2)p}{p-1}>0.
    \]
    This allows us to apply Hardy-Poincar\'{e} inequality (see \cite[Lemma A.1]{FZ22}) to $v_n$ with
    \[
    \alpha=\frac{(N-p)(p^*_\beta-2-p)}{p-1}+\beta<N,
    \]
    namely,
    \[
    \int_{\mathbb{R}^N\setminus B_R}|v_n|^p|x|^{-\alpha}\mathrm{d}x
    \leq C \int_{\mathbb{R}^N\setminus B_R}|\nabla v_n|^p|x|^{-\alpha+p}\mathrm{d}x,\quad \forall R>1.
    \]
    Then similarly to \eqref{propcetlllx}, for any $R> 1$ we obtain
    \begin{align}\label{propcetlllxb}
    & \int_{\mathbb{R}^N\setminus B_R}|x|^{-\beta}
    \frac{(U+C_1\varepsilon_nv_n)^{p^*_{\beta}}}
    {U^2+|\varepsilon_nv_n|^2}|v_n|^2\mathrm{d}x
    \nonumber\\ \leq & C\int_{\mathbb{R}^N\setminus B_R}|x|^{-\beta}
    U^{p^*_{\beta}-2}\left(1+\frac{\varepsilon_n v_n}{U}\right)^{p-2}|v_n|^2\mathrm{d}x
    \nonumber\\ \leq & C \int_{\mathbb{R}^N\setminus B_R}|x|^{\frac{(p-N)(p^*_\beta-2-p)}{p-1}+p-\beta}
    \nonumber\\ & \quad \times \left[
    \left(1+\frac{\varepsilon_n v_n}{U}\right)^{-2}|v_n|^2
    \left(\frac{\varepsilon_n v_n |\nabla U|}{U^2}+\frac{\varepsilon_n |\nabla v_n|}{U}\right)^{p}
    +\left(1+\frac{\varepsilon_n v_n}{U}\right)^{p-2}|v_n|^{2-p}|\nabla v_n|^p
    \right]
    \mathrm{d}x
    \nonumber\\ \leq & C \int_{\mathbb{R}^N\setminus B_R}|x|^{\frac{(p-N)(p^*_\beta-2-p)}{p-1}+p-\beta}
    \left[
    |v_n|^2\left(\frac{\zeta|\nabla U|}{U}+\frac{\varepsilon_n |\nabla v_n|}{U}\right)^{p}+|v_n|^{2-p}|\nabla v_n|^p
    \right]
    \mathrm{d}x.
    \end{align}
    Then applying \eqref{nif} to the last integrand in \eqref{propcetlllxb} with $\varepsilon=\varepsilon_n$, $r=|x|$, $a=|v_n|$, $b=|\nabla v_n|$, we obtain that for any $\varepsilon'_0>0$ there exists $\zeta=\zeta(\varepsilon'_0)\in (0,1)$ such that
    \begin{small}
    \begin{align*}
    \int_{\mathbb{R}^N\setminus B_R}|x|^{-\beta}
    \frac{(U+C_1\varepsilon_nv_n)^{p^*_{\beta}}}
    {U^2+|\varepsilon_nv_n|^2}|v_n|^2\mathrm{d}x
    \leq & C\varepsilon'_0
    \int_{\mathbb{R}^N\setminus B_R}|x|^{-\beta}
    U^{p^*_{\beta}-2}|v_n|^2\mathrm{d}x
    \\ & + C(\varepsilon'_0)R^{-\frac{p-\beta}{p-1}}
    \int_{\mathbb{R}^N\setminus B_R}(|\nabla U|+\varepsilon_n|\nabla v_n|)^{p-2}|\nabla v_n|^2\mathrm{d}x
    \\ \leq & C\varepsilon'_0
    \int_{\mathbb{R}^N\setminus B_R}|x|^{-\beta}
    \frac{(U+C_1\varepsilon_nv_n)^{p^*_{\beta}}}
    {U^2+|\varepsilon_nv_n|^2}|v_n|^2\mathrm{d}x
    \\ & + C(\varepsilon'_0)R^{-\frac{p-\beta}{p-1}}
    \int_{\mathbb{R}^N\setminus B_R}(|\nabla U|+\varepsilon_n|\nabla v_n|)^{p-2}|\nabla v_n|^2\mathrm{d}x.
    \end{align*}
    \end{small}
    Thus, by fixing $\varepsilon'_0$ small enough such that $C\varepsilon'_0\leq \frac{1}{2}$, it follows from \eqref{propcetll} and the inequality above that
    \begin{align*}
    \int_{\mathbb{R}^N\setminus B_R}|x|^{-\beta}
    \frac{(U+C_1\varepsilon_nv_n)^{p^*_{\beta}}}
    {U^2+|\varepsilon_nv_n|^2}|v_n|^2\mathrm{d}x
    \leq & CR^{-\frac{p-\beta}{p-1}}
    \int_{\mathbb{R}^N\setminus B_R}(|\nabla U|+\varepsilon_n|\nabla v_n|)^{p-2}|\nabla v_n|^2\mathrm{d}x
    \\ \leq & CR^{-\frac{p-\beta}{p-1}}.
    \end{align*}
    Combining this bound with \eqref{propcetlllc} and \eqref{propcetlcB}, by the arbitrariness of $R$ we conclude that $v\in L^2_{\beta,*}(\mathbb{R}^N)$ and that \eqref{propcetlc} holds. This concludes the proof under assumption that $\varepsilon_n v_n\leq \zeta U$ with $\zeta=\zeta(N,p,\beta,C_1)>0$ sufficiently small.

    $\bullet$ {\em \color{blue} Step 2: proof of \eqref{propcetlc} in the general case}. Throughout this part, we assume that $\zeta=\zeta(N,p,\beta,C_1)>0$ is a small constant such that Step 1 applies.

    Same as \cite[Lemma 3.4]{FZ22}, we have
    \begin{align}\label{spsccy}
    C\varepsilon_n^{-2}\int_{\{\varepsilon_n\nabla v_n>\zeta U\}}|\nabla U|^p
    \mathrm{d}x
    \leq & \int_{\{\varepsilon_n\nabla v_n>\zeta U\}}(|\nabla U|+\varepsilon_n|\nabla v_n|)^{p-2}|\nabla v_n|^2\mathrm{d}x.
    \end{align}
    For readers' convenience, we explain the reason.
    Observe that, $\zeta U$ is a supersolution for the operator
    \[
    L_U[\varphi]:=-\mathrm{div}\left((|\nabla U|+|\nabla\varphi|)^{p-2}\nabla \varphi +(p-2)(|\nabla U|+|\nabla\varphi|)^{p-3}|\nabla\varphi|\nabla \varphi\right),
    \]
    namely $L_U[\zeta U]\geq 0$. Therefore, multiplying $L_U[\zeta U]\geq 0$ by $(\varepsilon_n v_n-\zeta U)_+$ and integrating by parts, we get
    \begin{align}\label{spsc}
    & \int_{\mathbb{R}^N}
    (|\nabla U|+\zeta|\nabla\varphi|)^{p-2}\zeta \nabla U\cdot\nabla (\varepsilon_n v_n-\zeta U)_+
    \mathrm{d}x
    \nonumber\\
     & \quad+(p-2)\int_{\mathbb{R}^N}
    (|\nabla U|+\zeta|\nabla\varphi|)^{p-3}\zeta^2 |\nabla U|\nabla U\cdot\nabla (\varepsilon_n v_n-\zeta U)_+
    \mathrm{d}x\geq 0.
    \end{align}
    Also, by the convexity of
    \[
    \mathbb{R}^N\ni z\mapsto F_t(z):=(t+|z|)^{p-2}|z|^2,\quad t\geq 0,
    \]
    we have
    \[
    F_t(z)+\nabla F_t(z)\cdot(z'-z)\leq F_t(z'),\quad \forall z,z'\in\mathbb{R}^N,\quad t\geq 0.
    \]
    Hence, applying this inequality with $t=|\nabla U|$, $z=\zeta\nabla U$, and $z'=\varepsilon_n\nabla v_n$, it follows by \eqref{spsc} that
    \begin{small}
    \begin{align*}
    C\varepsilon_n^{-2}\int_{\{\varepsilon_n\nabla v_n>\zeta U\}}|\nabla U|^p
    \mathrm{d}x
    \leq & \varepsilon_n^{-2}\int_{\{\varepsilon_n\nabla v_n>\zeta U\}}(|\nabla U|+\zeta|\nabla U|)^{p-2}\zeta^2|\nabla U|^2
    \mathrm{d}x
    \nonumber\\
    \leq &\varepsilon_n^{-2}\int_{\{\varepsilon_n\nabla v_n>\zeta U\}}(|\nabla U|+\varepsilon_n|\nabla v_n|)^{p-2}\varepsilon_n^{2}|\nabla v_n|^2\mathrm{d}x
    \nonumber\\
     & +2\varepsilon_n^{-2}\int_{\{\varepsilon_n\nabla v_n>\zeta U\}}(|\nabla U|+\zeta|\nabla U|)^{p-2}\zeta\nabla U\cdot\nabla(\varepsilon_n\nabla v_n-\zeta U)_+
    \mathrm{d}x
    \nonumber\\
     & +(p-2)\varepsilon_n^{-2}\int_{\{\varepsilon_n\nabla v_n>\zeta U\}}
    (|\nabla U|+\zeta|\nabla U|)^{p-3}
    \nonumber\\
     & \quad \times\zeta^2 |\nabla U|\nabla U\cdot\nabla (\varepsilon_n v_n-\zeta U)_+
    \mathrm{d}x
    \nonumber \\
    \leq & \int_{\{\varepsilon_n\nabla v_n>\zeta U\}}(|\nabla U|+\varepsilon_n|\nabla v_n|)^{p-2}|\nabla v_n|^2\mathrm{d}x,
    \end{align*}
    \end{small}
    thus \eqref{spsccy} holds.
    We now write $v_n=v_{n,1}+v_{n,2}$, where
    \begin{align}\label{defvn12}
    v_{n,1}:=\min\left\{v_n, \frac{\zeta U}{\varepsilon_n}\right\}
    ,\quad v_{n,2}:=v_n-v_{n,1}.
    \end{align}
    Note that, as a consequence of \eqref{propcetll} and \eqref{spsccy}, \begin{align}\label{spsccc}
    & \int_{\mathbb{R}^N}(|\nabla U|+\varepsilon_n|\nabla v_{n,1}|)^{p-2}|\nabla v_{n,1}|^2\mathrm{d}x
    + \int_{\mathbb{R}^N}(|\nabla U|+\varepsilon_n|\nabla v_{n,2}|)^{p-2}|\nabla v_{n,2}|^2\mathrm{d}x
    \nonumber \\
    \leq & C \int_{\mathbb{R}^N}(|\nabla U|+\varepsilon_n|\nabla v_{n}|)^{p-2}|\nabla v_{n}|^2\mathrm{d}x
    \leq C.
    \end{align}
    Hence, it follows by the analogue of \eqref{propcetllb} that
    \begin{align}\label{propcetllbg}
    \int_{\mathbb{R}^N}|\nabla v_{n,1}|^p\mathrm{d}x
    + \int_{\mathbb{R}^N}|\nabla v_{n,2}|^p\mathrm{d}x
    \leq C.
    \end{align}
    In particular we deduce that $v_{n,2}\rightharpoonup 0$ in $\mathcal{D}^{1,p}_0(\mathbb{R}^N)$ (since $|\{\varepsilon_n\nabla v_n>\zeta U\}\cap B_R|\to 0$ for any $R>1$) and that, up to a subsequence, both $v_n$ and $v_{n,1}$ converge weakly in $\mathcal{D}^{1,p}_0(\mathbb{R}^N)$ and also a.e. to the same function $v\in \mathcal{D}^{1,p}_0(\mathbb{R}^N)$.

    Let $\eta=\eta(N,p,\beta)>0$ be a small exponent to be fixed. We analyze two cases.

    {\em Case 1}. If
    \begin{align*}
    \int_{\{\varepsilon_n\nabla v_n>\zeta U\}}|x|^{-\beta}|v_{n,1}|^{p^*_\beta}\mathrm{d}x
    > & \varepsilon^{-\eta}_n\int_{\{\varepsilon_n\nabla v_n>\zeta U\}}|x|^{-\beta}\left(v_n-\frac{\zeta U}{\varepsilon_n}\right)_+^{p^*_\beta}\mathrm{d}x
    \nonumber \\
    = & \varepsilon^{-\eta}_n\int_{\{\varepsilon_n\nabla v_n>\zeta U\}}|x|^{-\beta}|v_{n,2}|^{p^*_\beta}\mathrm{d}x,
    \end{align*}
    since $v$ is also the limit of $v_{n,1}$, we can apply Step 1 to $v_{n,1}$ to deduce that $v\in L^2_{\beta,*}(\mathbb{R}^N)$ and
    \begin{align*}
    \int_{\mathbb{R}^N}|x|^{-\beta}
    \frac{(U+C_1\varepsilon_nv_n)^{p^*_{\beta}}}
    {U^2+|\varepsilon_nv_n|^2}|v_n|^2\mathrm{d}x
    = & \left(1+O(\varepsilon_n^\eta)\right)
    \int_{\mathbb{R}^N}|x|^{-\beta}
    \frac{(U+C_1\varepsilon_nv_{n,1})^{p^*_{\beta}}}
    {U^2+|\varepsilon_nv_{n,1}|^2}|v_{n,1}|^2\mathrm{d}x
    \\ \to & \int_{\mathbb{R}^N}|x|^{-\beta} U^{p^*_{\beta}-2}|v|^2\mathrm{d}x,
    \end{align*}
    which proves \eqref{propcetlc}.

    {\em Case 2}. Assume now
    \begin{align}\label{conditionl}
    \int_{\{\varepsilon_n\nabla v_n>\zeta U\}}|x|^{-\beta}|v_{n,1}|^{p^*_\beta}\mathrm{d}x
    \leq \varepsilon^{-\eta}_n\int_{\{\varepsilon_n\nabla v_n>\zeta U\}}|x|^{-\beta}|v_{n,2}|^{p^*_\beta}\mathrm{d}x.
    \end{align}
    We claim that
    \begin{align}\label{claimvn2s}
    \varepsilon^{p^*_\beta-2}_n
    \int_{\mathbb{R}^N}|x|^{-\beta}|v_{n,2}|^{p^*_\beta}\mathrm{d}x
    =O(\varepsilon_n^\eta).
    \end{align}
    To prove this, we denote $D_n:=\{\varepsilon_n\nabla v_n>\zeta U\}$ and also define
    \[
    E_n:=\left\{|\nabla v_{n,2}|\leq \frac{|\nabla U|}{\varepsilon_n}\right\}\cap D_n,
    \quad F_n:=\left\{|\nabla v_{n,2}|>\frac{|\nabla U|}{\varepsilon_n}\right\}\cap D_n.
    \]
    Then, since  $v_{n,2}=v_{n}-v_{n,1}\equiv 0$ inside $\mathbb{R}^N\setminus D_n$, $D_n=E_n\cup F_n$, and $|\nabla U|+\varepsilon_n |\nabla v_{n,2}|\leq 2|\nabla U|$ inside $E_n$, it follows by H\"{o}lder inequality that
    \begin{small}
    \begin{align}\label{vn2pl}
    \int_{\mathbb{R}^N}|\nabla v_{n,2}|^p\mathrm{d}x
    = & \int_{D_n}|\nabla v_{n,2}|^p\mathrm{d}x
    + \int_{\mathbb{R}^N\setminus D_n}|\nabla v_{n,2}|^p\mathrm{d}x
    \nonumber \\
    = & \int_{E_n}|\nabla v_{n,2}|^p\mathrm{d}x
    + \int_{F_n}|\nabla v_{n,2}|^p\mathrm{d}x
    \nonumber \\
    \leq & \left(\int_{E_n}|\nabla U|^{p-2}|\nabla v_{n,2}|^{2}\mathrm{d}x\right)^{\frac{p}{2}}
    \left(\int_{E_n}|\nabla U|^{p}\mathrm{d}x\right)^{1-\frac{p}{2}}
    + \int_{F_n}|\nabla v_{n,2}|^p\mathrm{d}x
    \nonumber \\
    \leq & C\left[\int_{E_n}(|\nabla U|+\varepsilon_n |\nabla v_{n,2}|)^{p-2}|\nabla v_{n,2}|^{2}\mathrm{d}x\right]^{\frac{p}{2}}
    \left(\int_{E_n}|\nabla U|^{p}\mathrm{d}x\right)^{1-\frac{p}{2}}
    + \int_{F_n}|\nabla v_{n,2}|^p\mathrm{d}x.
    \end{align}
    \end{small}
    Also, using \eqref{UnUe} and condition \eqref{conditionl} together with H\"{o}lder inequality,
    taking
    \begin{align}\label{defQn}
    Q:=\frac{pN-\beta}{p(N-1)-(p-\beta)\nu}>1
    \end{align}
    for any $\nu$ satisfying $\frac{(p-1)N}{N-\beta}<\nu<\frac{p(N-1)}{p-\beta}$ (since $\beta<p<N$), 
     we deduce
    \begin{align}\label{vn2pl1}
    \int_{E_n}|\nabla U|^{p}\mathrm{d}x
    \leq & C \int_{E_n}(1+|x|^{\frac{p-\beta}{p-1}})
    ^{-\frac{(N-\beta)p}{p-\beta}}|x|^{\frac{(1-\beta)p}{p-1}}
    \mathrm{d}x
    \nonumber \\
    \leq & C \left[\int_{E_n}\left((1+|x|^{\frac{p-\beta}{p-1}})
    ^{-\frac{(N-\beta)p}{p-\beta}+\nu}|x|^{\frac{(1-\beta)p}{p-1}}\right)
    ^{Q}
    \mathrm{d}x\right]^{\frac{1}{Q}}
    \nonumber \\ & \quad \times\left(\int_{E_n}(1+|x|^{\frac{p-\beta}{p-1}})
    ^{-\frac{\nu Q}{Q-1}}
    \mathrm{d}x\right)^{\frac{Q-1}{Q}}
    \nonumber \\
    \leq & C \left[\int_{D_n}\left(\frac{\varepsilon_n v_n }{\zeta U}\right)^{p^*_\beta}\left((1+|x|^{\frac{p-\beta}{p-1}})
    ^{-\frac{(N-\beta)p}{p-\beta}+\nu}|x|^{\frac{(1-\beta)p}{p-1}}\right)
    ^{Q}
    \mathrm{d}x\right]^{\frac{1}{Q}}
    \nonumber \\
    \leq & C \left(\varepsilon_n^{p^*_\beta}\int_{D_n}|x|^{-\beta}|v_n|^{p^*_\beta}
    \mathrm{d}x\right)^{\frac{1}{Q}}
    \nonumber \\
    \leq & C \left(\varepsilon_n^{p^*_\beta-\eta}\int_{D_n}|x|^{-\beta}|v_{n,2}|^{p^*_\beta}
    \mathrm{d}x\right)^{\frac{1}{Q}},
    \end{align}
    where we used that $\frac{(p-\beta)}{(p-1)}\frac{\nu Q}{Q-1}>N$  and that
    \[
    U^{-p^*_\beta}\left((1+|x|^{\frac{p-\beta}{p-1}})
    ^{-\frac{(N-\beta)p}{p-\beta}+\nu}|x|^{\frac{(1-\beta)p}{p-1}}\right)
    ^{Q}\leq C |x|^{-\beta}.
    \]
    Therefore, introducing the notation
    \[
    N_{n,2}:=\int_{E_n}(|\nabla U|+\varepsilon_n |\nabla v_{n,2}|)^{p-2}|\nabla v_{n,2}|^{2}\mathrm{d}x,
    \]
    by Hardy-Sobolev inequality, \eqref{vn2pl}, and \eqref{vn2pl1}, we deduce that
    \begin{small}
    \begin{align}\label{vn2pl2}
    \varepsilon^{p^*_\beta-2}_n
    \int_{\mathbb{R}^N}|x|^{-\beta}|v_{n,2}|^{p^*_\beta}\mathrm{d}x
    \leq & C \varepsilon^{p^*_\beta-2}_n
    \left(\int_{\mathbb{R}^N}|\nabla v_{n,2}|^p
    \mathrm{d}x\right)^{\frac{p^*_\beta}{p}}
    \nonumber \\
    \leq & C \varepsilon^{p^*_\beta-2}_n
    \left[
    N_{n,2}^{\frac{p^*_\beta}{2}}\left(\int_{E_n}|\nabla U|^p
    \mathrm{d}x\right)^{\frac{(2-p)p^*_\beta}{2p}}
    + \left(\int_{F_n}|\nabla v_{n,2}|^p
    \mathrm{d}x\right)^{\frac{p^*_\beta}{p}}
    \right]
    \nonumber \\
    \leq & C \varepsilon^{p^*_\beta-2}_n
    \Bigg[
    N_{n,2}^{\frac{p^*_\beta}{2}}
    \left(\varepsilon^{p^*_\beta-\eta}_n
    \int_{D_n}|x|^{-\beta}|v_{n,2}|^{p^*_\beta}
    \mathrm{d}x\right)^{\frac{(2-p)p^*_\beta}{2pQ}}
    \nonumber \\ & \quad + \int_{F_n}|\nabla v_{n,2}|^p
    \mathrm{d}x
    \Bigg],
    \end{align}
    \end{small}
    where in the last inequality we used \eqref{propcetllbg} and the fact $\frac{p^*_\beta}{p}\geq 1$.

    Suppose first that
    \[
    \int_{F_n}|\nabla v_{n,2}|^p
    \mathrm{d}x
    \geq
    N_{n,2}^{\frac{p^*_\beta}{2}}
    \left(\varepsilon^{p^*_\beta-\eta}_n
    \int_{D_n}|x|^{-\beta}|v_{n,2}|^{p^*_\beta}
    \mathrm{d}x\right)^{\frac{(2-p)p^*_\beta}{2pQ}}.
    \]
    Then, since $|\nabla U|\leq \varepsilon_n|\nabla v_{n,2}|\sim \varepsilon_n|\nabla v_n|$ inside $F_n$ (recall that $\zeta<1$ small), \eqref{propcetll}, and \eqref{vn2pl2} yield
    \begin{align}\label{vn2pl3}
    \varepsilon^{p^*_\beta-2}_n
    \int_{\mathbb{R}^N}|x|^{-\beta}|v_{n,2}|^{p^*_\beta}\mathrm{d}x
    \leq & C \varepsilon^{p^*_\beta-2}_n
    \int_{F_n}|\nabla v_{n,2}|^p
    \mathrm{d}x
    \nonumber \\
    = & \varepsilon^{p^*_\beta-p}_n
    \int_{F_n}(\varepsilon_n|\nabla v_{n,2}|)^{p-2}|\nabla v_{n,2}|^2
    \mathrm{d}x
    \nonumber \\
    \leq & C \varepsilon^{p^*_\beta-p}_n
    \int_{F_n}(|\nabla U|+\varepsilon_n|\nabla v_{n,2}|)^{p-2}|\nabla v_{n,2}|^2
    \mathrm{d}x,
    \end{align}
    which proves \eqref{claimvn2s} choosing $\eta\leq p^*_\beta-p$ (recall \eqref{spsccc}).

    Then consider instead the case
    \[
    \int_{F_n}|\nabla v_{n,2}|^p
    \mathrm{d}x
    <
    N_{n,2}^{\frac{p^*_\beta}{2}}
    \left(\varepsilon^{p^*_\beta-\eta}_n
    \int_{D_n}|x|^{-\beta}|v_{n,2}|^{p^*_\beta}
    \mathrm{d}x\right)^{\frac{(2-p)p^*_\beta}{2pQ}}.
    \]
    Set $\theta:=\frac{(2-p)p^*_\beta}{2pQ}$ where $Q$ depends on $\nu$ given in \eqref{defQn}, so that \eqref{vn2pl2} yields
    \begin{align*}
    \varepsilon^{p^*_\beta-2}_n
    \int_{\mathbb{R}^N}|x|^{-\beta}|v_{n,2}|^{p^*_\beta}\mathrm{d}x
    \leq & C \varepsilon^{p^*_\beta-2}_n N_{n,2}^{\frac{p^*_\beta}{2}}
    \left(\varepsilon^{p^*_\beta-\eta}_n
    \int_{D_n}|x|^{-\beta}|v_{n,2}|^{p^*_\beta}
    \mathrm{d}x\right)^\theta
    \nonumber \\
    = & C \varepsilon^{p^*_\beta-2+(2-\eta)\theta}_n N_{n,2}^{\frac{p^*_\beta}{2}}
    \left(\varepsilon^{p^*_\beta-2}_n
    \int_{D_n}|x|^{-\beta}|v_{n,2}|^{p^*_\beta}
    \mathrm{d}x\right)^\theta.
    \end{align*}
    We need $1-\frac{p^*_\beta}{2}<\theta<1$ which is equivalent to
    \begin{align*}
    &\nu_1:=\frac{p(N-1)(2-p)(N-\beta)-2(N-p)(Np-\beta)}{(p-\beta)(2-p)(N-\beta)}
    <\nu<
    \\ & \quad
    \frac{p(N-1)(2-p)(N-\beta)-(Np-\beta)[2(N-p)-p(N-\beta)]}
    {(p-\beta)(2-p)(N-\beta)}=:\nu_2,
    \end{align*}
    and since $1<p\leq \frac{2N}{N+2-\beta}$ and $0<\beta<p<2$, it is easy to verify that
    \[
    \nu_1<
    \frac{p(N-1)}{p-\beta}\quad \mbox{and}\quad
    \nu_2>
    \frac{(p-1)N}{N-\beta},
    \]
    we can always choose suitable $\nu$ satisfying $\frac{(p-1)N}{N-\beta}<\nu<\frac{p(N-1)}{p-\beta}$ such that $1-\frac{p^*_\beta}{2}<\theta<1$. In fact, taking
    \[
    \nu=\frac{1}{4}\left(\frac{(p-1)N}{N-\beta}+\frac{p(N-1)}{p-\beta}
    +\nu_1+\nu_2\right),
     \]
    this condition holds.
     Then recalling the definition $N_{n,2}$ and \eqref{spsccc}, this gives
    \begin{align}\label{vn2pl4}
    \varepsilon^{p^*_\beta-2}_n
    \int_{\mathbb{R}^N}|x|^{-\beta}|v_{n,2}|^{p^*_\beta}\mathrm{d}x
    \leq & C \varepsilon_n^{\frac{p^*_\beta-2+(2-\eta)\theta}{1-\theta}} \left(\int_{E_n}(|\nabla U|+\varepsilon_n|\nabla v_{n,2}|)^{p-2}|\nabla v_{n,2}|^2
    \mathrm{d}x\right)^{\frac{p^*_\beta}{2(1-\theta)}}
    \nonumber \\
    \leq & C \varepsilon_n^{\eta}
    \int_{E_n}(|\nabla U|+\varepsilon_n|\nabla v_{n,2}|)^{p-2}|\nabla v_{n,2}|^2
    \mathrm{d}x,
    \end{align}
    where the last inequality follows by choosing $\eta>0$ sufficiently small (notice that $p^*_\beta-2+2\theta>0$ and $\frac{p^*_\beta}{2(1-\theta)}>1$). This proves \eqref{claimvn2s} also in this case.

    Now, combining \eqref{conditionl} and \eqref{claimvn2s}, we finally obtain
    \begin{align*}
    & \left|\int_{\mathbb{R}^N}|x|^{-\beta}
    \frac{(U+C_1\varepsilon_nv_n)^{p^*_{\beta}}}
    {U^2+|\varepsilon_nv_n|^2}|v_n|^2\mathrm{d}x
    - \int_{\mathbb{R}^N}|x|^{-\beta}
    \frac{(U+C_1\varepsilon_nv_{n,1})^{p^*_{\beta}}}
    {U^2+|\varepsilon_nv_{n,1}|^2}|v_{n,1}|^2\mathrm{d}x
    \right|
    \\
    \leq  & C\left(
    \varepsilon^{p^*_\beta-2}_n
    \int_{D_n}|x|^{-\beta}|v_{n,2}|^{p^*_\beta}\mathrm{d}x
    + \varepsilon^2_n \int_{D_n}|x|^{-\beta}
    \frac{(U+C_1\zeta U)^{p^*_{\beta}}}
    {U^2+|\zeta U|^2}|\zeta U|^2\mathrm{d}x
    \right)
    \\ = & O(\varepsilon_n^\eta)+O(\varepsilon_n^2)
    =o(1).
    \end{align*}
    Thanks to this estimate, and since $v$ is also the limit of $v_{n,1}$, applying Step 1 to $v_{n,1}$ we conclude the proof of this lemma.
    \end{proof}

    An important consequence of Lemma \ref{propcetl} is the following weighted Orlicz-type Poincar\'{e} inequality:
    \begin{corollary}\label{propcetlpi}
    Let $1<p\leq \frac{2N}{N+2-\beta}$ with $0<\beta<p$. There exists $\varepsilon_0=\varepsilon_0(N,p,\beta)>0$ small such that the following holds:
    For any $\varepsilon \in (0,\varepsilon_0)$ and any radial function $v\in \mathcal{D}^{1,p}_{0}(\mathbb{R}^N)\cap \mathcal{D}^{1,2}_{0,*}(\mathbb{R}^N)$ satisfying
    \begin{align*}
    & \int_{\mathbb{R}^N}(|\nabla U|+\varepsilon|\nabla v|)^{p-2}|\nabla v|^2\mathrm{d}x\leq 1,
    \end{align*}
    we have
    \begin{align}\label{propcetlcpi}
    \int_{\mathbb{R}^N}|x|^{-\beta}
    (U+ \varepsilon v )^{p^*_{\beta}-2}|v |^2\mathrm{d}x
    \leq C(N,p,\beta)
    \int_{\mathbb{R}^N}(|\nabla U|+\varepsilon|\nabla v|)^{p-2}|\nabla v|^2\mathrm{d}x.
    \end{align}
    \end{corollary}

    \begin{proof}
    Based on Lemma \ref{propcetl}, this proof can be deduced directly from the proof of \cite[Corollary 3.5]{FZ22} with minor changes, so we omit it.
    \end{proof}

\section{{\bfseries Non-degenerate result}}\label{sectndr}

    First of all, let us rewrite the linear equation (\ref{Ppwhla}) as
    \begin{align}\label{rwlp}
    & -|x|^2\Delta v -(p-2) \sum^{N}_{i,j=1}\frac{\partial^2 v}{\partial x_i\partial x_j}x_i x_j
    -\frac{(p-2)(N-\beta)}{1+|x|^{\frac{p-\beta}{p-1}}} (x\cdot\nabla v) \nonumber\\
    = & (p^*_{\beta}-1)C_{N,p,\beta}^{p^*_{\beta}-p}
    \left(\frac{N-p}{p-1}\right)^{2-p}
    \frac{|x|^{\frac{p-\beta}{p-1}}}{(1+|x|^{\frac{p-\beta}{p-1}})^2}v
    \quad \mbox{in}\quad \mathbb{R}^N,\quad v\in L^2_{\beta,*}(\mathbb{R}^N).
    \end{align}
    Indeed a straightforward computation shows that
    \begin{align*}
    & \mathrm{div}(|\nabla U|^{p-2}\nabla v)+(p-2)\mathrm{div}(|\nabla U|^{p-4}(\nabla U\cdot\nabla v)\nabla U) \nonumber\\
    = & |\nabla U|^{p-2}\Delta v + \nabla(|\nabla U|^{p-2})\cdot\nabla v
     + (p-2)|\nabla U|^{p-4}(\nabla U\cdot\nabla v)\Delta U \nonumber\\
     & + (p-2)(\nabla U\cdot\nabla v) (\nabla (|\nabla U|)\cdot\nabla U)
     + (p-2)|\nabla U| (\nabla (\nabla U \cdot \nabla v)\cdot\nabla U) \nonumber\\
    = & |\nabla U|^{p-2}\Delta v
        +(p-2)(p-4)|\nabla U|^{p-6}(\nabla U\cdot\nabla v) (\nabla U \nabla(\nabla U) \cdot\nabla U)
        \nonumber\\
      &
      + (p-2)|\nabla U|^{p-4}\left[ (\nabla U\cdot\nabla v)\Delta U +2(\nabla U\nabla (\nabla U)\cdot\nabla v)
       + (\nabla U\nabla(\nabla v) \cdot \nabla U)\right],
    \end{align*}
    and
    \begin{align*}
    \nabla U
    = & -\frac{c_{N,p}|x|^{\frac{2-p-\beta}{p-1}}x}
    {(1+|x|^{\frac{p-\beta}{p-1}})^{\frac{N-\beta}{p-\beta}}},  \quad (x\cdot\nabla U)
    = -\frac{c_{N,p}|x|^{\frac{p-\beta}{p-1}}}
    {(1+|x|^{\frac{p-\beta}{p-1}})^{\frac{N-\beta}{p-\beta}}}, \nonumber\\
    (\nabla U\cdot\nabla v)
    = & -\frac{c_{N,p}|x|^{\frac{2-p-\beta}{p-1}}(x\cdot\nabla v)}
    {(1+|x|^{\frac{p-\beta}{p-1}})^{\frac{N-\beta}{p-\beta}}}.
    \end{align*}
    Moreover,
    \begin{align*}
    \Delta U
    = & \frac{-c_{N,p}}
    {(1+|x|^{\frac{p-\beta}{p-1}})^{\frac{N-\beta}{p-\beta}}}
    \left\{\left(\frac{2-p-\beta}{p-1}+N\right)|x|^{\frac{2-p-\beta}{p-1}}
    -\frac{\frac{N-\beta}{p-1}|x|^{\frac{2-2\beta}{p-1}}}
    {1+|x|^{\frac{p-\beta}{p-1}}}
    \right\}, \nonumber\\
    \sum^N_{j=1}\frac{\partial U}{\partial x_j}\frac{\partial^2 U}{\partial x_i\partial x_j}
    = & \frac{c^2_{N,p}}
    {(1+|x|^{\frac{p-\beta}{p-1}})^{\frac{2(N-\beta)}{p-\beta}}}
    \left\{\frac{1-\beta}{p-1}|x|^{\frac{2(2-p-\beta)}{p-1}}
    -\frac{\frac{N-\beta}{p-1}|x|^{\frac{4-p-3\beta}{p-1}}}
    {1+|x|^{\frac{p-\beta}{p-1}}}
    \right\}x_i,
    \end{align*}
    where $c_{N,p}:=C_{N,p,\beta}\frac{N-p}{p-1}$. Here,
    \begin{align*}
    (\nabla U\nabla (\nabla v)\cdot\nabla U)
    := & \sum^N_{i,j}\frac{\partial U}{\partial x_j}\frac{\partial^2 v}{\partial x_i\partial x_j}\frac{\partial U}{\partial x_i}
    = \frac{c^2_{N,p}|x|^{\frac{2(2-p-\beta)}{p-1}}}
    {(1+|x|^{\frac{p-\beta}{p-1}})^{\frac{2(N-\beta)}{p-\beta}}}
    \sum^N_{i,j}\frac{\partial^2 v}{\partial x_i\partial x_j}x_ix_j,  \\
    (\nabla U\nabla (\nabla U)\cdot\nabla v)
    := & \sum^N_{i,j}\frac{\partial U}{\partial x_j}\frac{\partial^2 U}{\partial x_i\partial x_j}\frac{\partial v}{\partial x_i} \\
    = & \frac{c^2_{N,p}(x\cdot\nabla v)}
    {(1+|x|^{\frac{p-\beta}{p-1}})^{\frac{2(N-\beta)}{p-\beta}}}
    \left\{\frac{1-\beta}{p-1}|x|^{\frac{2(2-p-\beta)}{p-1}}
    -\frac{\frac{N-\beta}{p-1}|x|^{\frac{4-p-3\beta}{p-1}}}
    {1+|x|^{\frac{p-\beta}{p-1}}}
    \right\}, \\
    (\nabla U\nabla(\nabla U) \cdot\nabla U)
    := & \sum^N_{i,j}\frac{\partial U}{\partial x_j}\frac{\partial^2 U}{\partial x_i\partial x_j}\frac{\partial U}{\partial x_i} \\
    = & \frac{-c^3_{N,p}|x|^{\frac{p-\beta}{p-1}}}
    {(1+|x|^{\frac{p-\beta}{p-1}})^{\frac{3(N-\beta)}{p-\beta}}}
    \left\{\frac{1-\beta}{p-1}|x|^{\frac{2(2-p-\beta)}{p-1}}
    -\frac{\frac{N-\beta}{p-1}|x|^{\frac{4-p-3\beta}{p-1}}}
    {1+|x|^{\frac{p-\beta}{p-1}}}
    \right\}.
    \end{align*}
    Then by using the standard spherical decomposition and making the change of variable $r\mapsto r^{\frac{p}{p-\beta}}$, we can characterize all solutions to the linearized problem (\ref{rwlp}).

    \subsection{Proof of Theorem \ref{coroPpwhlpa}.}

    We make the standard partial wave decomposition of (\ref{rwlp}), namely
    \begin{equation}\label{Ppwhl2defvdp}
    v=v(r,\theta)=\sum^{\infty}_{k=0}\varphi_k(r)\Psi_k(\theta),
    \end{equation}
    where $r=|x|$, $\theta=\frac{x}{|x|}\in \mathbb{S}^{N-1}$, and
    \begin{equation*}
    \varphi_k(r)=\int_{\mathbb{S}^{N-1}}v(r,\theta)\Psi_k(\theta)
    \mathrm{d}\theta.
    \end{equation*}
    Here $\Psi_k(\theta)$ denotes the $k$-th spherical harmonic, i.e., it satisfies
    \begin{equation}\label{deflk}
    -\Delta_{\mathbb{S}^{N-1}}\Psi_k=\lambda_k \Psi_k,
    \end{equation}
    where $\Delta_{\mathbb{S}^{N-1}}$ is the Laplace-Beltrami operator on $\mathbb{S}^{N-1}$ with the standard metric and  $\lambda_k$ is the $k$-th eigenvalue of $-\Delta_{\mathbb{S}^{N-1}}$. It is well known that \begin{equation}\label{deflklk}
    \lambda_k=k(N-2+k),\quad k=0,1,2,\ldots,
    \end{equation}
    whose multiplicity is $\frac{(N+2k-2)(N+k-3)!}{(N-2)!k!}$ and that \[
    \mathrm{Ker}(\Delta_{\mathbb{S}^{N-1}}+\lambda_k)
    =\mathbb{Y}_k(\mathbb{R}^N)|_{\mathbb{S}^{N-1}},
    \]
    where $\mathbb{Y}_k(\mathbb{R}^N)$ is the space of all homogeneous harmonic polynomials of degree $k$ in $\mathbb{R}^N$. It is standard that $\lambda_0=0$ and the corresponding eigenfunction of (\ref{deflk}) is the constant function. The second eigenvalue $\lambda_1=N-1$ and the corresponding eigenfunctions of (\ref{deflk}) are $x_i/|x|$, $i=1,\ldots,N$.

    The following results can be obtained by direct calculation,
    \begin{align}\label{Ppwhl2deflklwp}
    \Delta (\varphi_k(r)\Psi_k(\theta))
    = & \Psi_k\left(\varphi''_k+\frac{N-1}{r}\varphi'_k\right)+\frac{\varphi_k}{r^2}\Delta_{\mathbb{S}^{N-1}}\Psi_k \nonumber\\
    = & \Psi_k\left(\varphi''_k+\frac{N-1}{r}\varphi'_k-\frac{\lambda_k}{r^2}\varphi_k\right).
    \end{align}
    It is easy to verify that
    \begin{equation*}
    \frac{\partial (\varphi_k(r)\Psi_k(\theta))}{\partial x_i}=\varphi'_k\frac{x_i}{r}\Psi_k+\varphi_k\frac{\partial\Psi_k}{\partial \theta_l}\frac{\partial\theta_l}{\partial x_i},
    \end{equation*}
    hence
    \begin{equation}\label{Ppwhl2deflklnp}
    \begin{split}
    x\cdot\nabla (\varphi_k(r)\Psi_k(\theta))=\sum^{N}_{i=1}x_i\frac{\partial (\varphi_k(r)\Psi_k(\theta))}{\partial x_i}=\varphi'_kr\Psi_k+\varphi_k\frac{\partial\Psi_k}{\partial \theta_l}\sum^{N}_{i=i}\frac{\partial\theta_l}{\partial x_i}x_i=\varphi'_kr\Psi_k,
    \end{split}
    \end{equation}
    and
    \begin{align}\label{npp2}
    \sum^N_{i,j=1}\frac{\partial^2 (\varphi_k(r)\Psi_k(\theta))}{\partial x_i\partial x_j}x_ix_j
    = & 2 \varphi'_kr\frac{\partial\Psi_k}{\partial \theta_l}\sum^N_{i=1}\frac{\partial\theta_l}{\partial x_i}x_i
    + \varphi_k\frac{\partial^2\Psi_k}{\partial \theta_l\partial \theta_m}\sum^N_{i,j=1}\frac{\partial\theta_l}{\partial x_i}x_i\frac{\partial\theta_m}{\partial x_j}x_j\nonumber \\
    & + \frac{\partial\Psi_k}{\partial \theta_l}\varphi_k\sum^N_{i,j=1}\frac{\partial^2\theta_l}{\partial x_i\partial x_j}x_ix_j
    +\varphi''_kr^2\Psi_k
    = \varphi''_kr^2\Psi_k,
    \end{align}
    due to
    \begin{equation*}
    \begin{split}
    \sum^N_{i=1}\frac{\partial\theta_l}{\partial x_i}x_i=0\quad \mbox{and}\quad \sum^N_{i,j=1}\frac{\partial^2\theta_l}{\partial x_i\partial x_j}x_ix_j=0,\quad l=1,\ldots,N-1.
    \end{split}
    \end{equation*}
    Then putting together (\ref{Ppwhl2defvdp}), (\ref{Ppwhl2deflklwp}), (\ref{Ppwhl2deflklnp}) and (\ref{npp2}) into (\ref{rwlp}), the function $v$ is a solution of (\ref{rwlp}) if and only if $\varphi_k\in\mathcal{W}$ is a classical solution of the system
    \begin{eqnarray}\label{Ppwhl2p2tpy}
    \left\{ \arraycolsep=1.5pt
       \begin{array}{ll}
        (p-1)\varphi''_k+\frac{\varphi'_k}{r}\left[(N-1)
        +\frac{(p-2)(N-\beta)}{1+r^{\frac{p-\beta}{p-1}}}\right]
    -\frac{\lambda_k}{r^2}\varphi_k \\[4mm]
    + (p^*_{\beta}-1)C_{N,p,\beta}^{p^*_{\beta}-p}
    \left(\frac{N-p}{p-1}\right)^{2-p}
    \frac{r^{\frac{p-\beta}{p-1}-2}}
    {\left(1+r^{\frac{p-\beta}{p-1}}\right)^2}\varphi_k=0 \quad \mbox{in}\quad r\in(0,\infty),\\[4mm]
        \varphi'_k(0)=0 \quad\mbox{if}\quad k=0,\quad \mbox{and}\quad \varphi_k(0)=0 \quad\mbox{if}\quad k\geq 1,
        \end{array}
    \right.
    \end{eqnarray}
    where $\mathcal{W}:=\{\omega\in C([0,\infty))| \int^\infty_0 r^{-\beta} U^{p^*_{\beta}-2}|\omega(r)|^2 r^{N-1} \mathrm{d}r<\infty\}$.
    We use the change of variable: $r=s^q$ with $q=p/(p-\beta)$,  and let
    \begin{equation*}
    \eta_k(s)=\varphi_k(r),
    \end{equation*}
    that transforms (\ref{Ppwhl2p2tpy}) into the following equations for any $\eta_k\in \widetilde{\mathcal{W}}$, $k=0,1,2,\ldots,$
    \begin{equation}\label{Ppwhl2p2tp}
    \eta''_k+\frac{\eta'_k}{s}\left(\frac{K-1}{p-1}
    +\frac{(p-2)K}{(p-1)(1+s^{\frac{p}{p-1}})}\right)
    -\frac{q^2\lambda_k}{(p-1)s^2}\eta_k
    +\frac{K(Kp-K+p)}{(p-1)^2}
    \frac{s^{\frac{p}{p-1}-2}}{(1+s^{\frac{p}{p-1}})^2}\eta_k=0.
    \end{equation}
    where $\widetilde{\mathcal{W}}:=\{\omega\in C([0,\infty))| \int^\infty_0 W^{\frac{Kp}{K-p}-2}|\omega(s)|^2 s^{K-1} ds<\infty\}$, $W(s)=U(r)$ and
    \begin{equation*}
    K=\frac{p(N-\beta)}{p-\beta}>p.
    \end{equation*}
    Here we have used the fact  
    \begin{equation*}
    q^2(p^*_{\beta}-1)C_{N,p,\beta}^{p^*_{\beta}-p}
    \left(\frac{N-p}{p-1}\right)^{2-p}
    =\frac{K(Kp-K+p)}{p-1},
    \end{equation*}
    where $C_{N,p,\beta}$ is given in \eqref{defvlzb}.

    Fix $K$ let us now consider the following eigenvalue problem
    \begin{equation}\label{Ppwhl2p2tep}
    \eta''+\frac{\eta'}{s}\left(\frac{K-1}{p-1}+\frac{(p-2)K}{(p-1)(1+s^{\frac{p}{p-1}})}\right)-\frac{\mu}{(p-1)s^2}\eta
        +\frac{K(Kp-K+p)}{(p-1)^2}\frac{s^{\frac{p}{p-1}-2}}{(1+s^{\frac{p}{p-1}})^2}\eta=0.
    \end{equation}
    When $K$ is an integer we can study (\ref{Ppwhl2p2tep}) as the linearized operator of the equation
    \begin{equation*}
    -\mathrm{div}(|\nabla W|^{p-2}\nabla W)=K\left(\frac{K-p}{p-1}\right)^{p-1}W^{\frac{Kp}{K-p}-1} ,\quad W>0 \quad \mbox{in}\quad \mathbb{R}^N, \quad u\in \mathcal{D}^{1,p}_0(\mathbb{R}^K).
    \end{equation*}
    around the standard solution $W(x)=(1+|x|^{p/(p-1)})^{-(K-p)/p}$ (note that we always have $K>p$). In this case, as in \cite[Proposition 3.1]{FN19} (see also \cite{PV21}), we have that
    \begin{equation}\label{Ppwhl2ptevp}
    \mu_0=0; \quad \mu_1=K-1\quad \mbox{and}\quad \eta_0(s)=\frac{(p-1)-s^{\frac{p}{p-1}}}{(1+s^{\frac{p}{p-1}})^\frac{K}{p}}; \quad \eta_1(s)=\frac{s^{\frac{1}{p-1}}}{(1+s^{\frac{p}{p-1}})^\frac{K}{p}}.
    \end{equation}
    Moreover, even $K$ is not an integer we readily see that (\ref{Ppwhl2ptevp}) remains true. Therefore, we conclude that (\ref{Ppwhl2p2tp}) has nontrivial solutions if and only if
    \begin{equation*}
    q^2\lambda_k\in \{0,K-1\},
    \end{equation*}
    where $\lambda_k$ is given in \eqref{deflklk}. If $q^2\lambda_k=0$ then $k=0$. Moreover, if $q^2\lambda_k=K-1$, that is
    \begin{equation}\label{ppk}
    \left(\frac{p-\beta}{p}\right)^2
    \left[\frac{p(N-\beta)}{p-\beta}-1\right]=k(N-2+k)\quad \mbox{for some}\quad k\in\mathbb{N}.
    \end{equation}
    However, since $0<\beta<p$, it is easy to verify that
    \begin{equation*}
    0<\left(\frac{p-\beta}{p}\right)^2
    \left[\frac{p(N-\beta)}{p-\beta}-1\right]
    < N-1,
    \end{equation*}
    which leads to \eqref{ppk} cannot happen. Therefore, we deduce that (\ref{Ppwhl2p2tpy}) only admits one solution
    \begin{equation}\label{Ppwhl2pyfp}
    \varphi_0(r)=\frac{(p-1)-r^{\frac{p-\beta}{p-1}}}
    {(1+r^{\frac{p-\beta}{p-1}})^\frac{N-\beta}{p-\beta}}.
    \end{equation}
    That is, the space of solutions of (\ref{rwlp}) has dimension $1$ and is spanned by
    \begin{equation*}
    v_0(x)=\frac{(p-1)-|x|^{\frac{p-\beta}{p-1}}}
    {(1+|x|^{\frac{p-\beta}{p-1}})^\frac{N-\beta}{p-\beta}}.
    \end{equation*}
    Since $v_0\sim \frac{\partial U_\lambda}{\partial \lambda}|_{\lambda=1}=\frac{N-p}{p}U+x\cdot \nabla U$, the proof of Theorem \ref{coroPpwhlpa} is complete.
    \qed

    \begin{remark}\label{remndr}\rm
    Theorem \ref{coroPpwhlpa} indicates $U$ is non-degenerate, and this conclusion has its own interest, particularly in the fields blow-up analysis and asymptotic analysis by using finite-dimensional reduction, see the statements in \cite{PV21}.
    \end{remark}

    Based on Proposition \ref{propcet} and Theorem \ref{coroPpwhlpa}, we are going to show the stability of Hardy-Sobolev inequality \eqref{defhsi} and give the proof of Theorem \ref{thmprtp}.

\section{{\bfseries Spectral analysis}}\label{sectspana}

    Let us consider the following eigenvalue problem
    \begin{equation}\label{pevp}
    \begin{split}
    & \mathcal{L}_{U} [v]=\mu|x|^{-\beta} U^{p^*_{\beta}-2}v \quad \mbox{in}\quad \mathbb{R}^N,\quad v\in L^2_{\beta,*}(\mathbb{R}^N),
    \end{split}
    \end{equation}
    where
    \begin{equation*}
    \begin{split}
    & \mathcal{L}_{U} [v]:=-\mathrm{div}(|\nabla U|^{p-2}\nabla v)-(p-2)\mathrm{div}(|\nabla U|^{p-4}(\nabla U\cdot\nabla v)\nabla U).
    \end{split}
    \end{equation*}
    When $p=2$, the eigenvalue problem \eqref{pevp} was classified by Smets and Willem \cite{SmWi03}. Then combining Theorem \ref{coroPpwhlpa} with \cite[Proposition 3.1]{FN19}, thanks to Proposition \ref{propcet}, making the change $r\mapsto r^{\frac{p}{p-\beta}}$ we deduce directly the following conclusion.

    \begin{proposition}\label{propev}
    Suppose $1<p<N$ and $0<\beta<p$. Let $\mu_i$, $i=1,2,\ldots,$ denote the eigenvalues of (\ref{pevp}) in increasing order. Then $\mu_1=(p-1)$ is simple and the corresponding eigenfunction is $\zeta U$ with $\zeta\in\mathbb{R}\setminus\{0\}$, $\mu_2=p^*_{\beta}-1$ and the corresponding eigenfunction is $\zeta (\frac{N-p}{p}U+x\cdot \nabla U)$ with $\zeta\in\mathbb{R}\setminus\{0\}$. Furthermore, the Rayleigh quotient characterization of eigenvalues implies
    \begin{equation*}
    \mu_{3}=\inf\left\{
    \frac{\int_{\mathbb{R}^N}\mathcal{L}_{U} [v] v \mathrm{d}x}{\int_{\mathbb{R}^N}|x|^{-\beta} U^{p^*_\beta-2}v^2 \mathrm{d}x}:  v\perp \mathrm{Span}\left\{U, \ \frac{N-p}{p}U+x\cdot \nabla U\right\}\right\}>p^*_{\beta}-1.
    \end{equation*}
    \end{proposition}

    In particular, Proposition \ref{propev} implies
    \begin{align}\label{czkj}
    T_{U}\mathcal{M}_\beta=\mathrm{Span}\left\{U,\ \frac{N-p}{p}U+x\cdot \nabla U\right\},
    \end{align}
    where $\mathcal{M}_\beta:=\{cU_{\lambda}: c\in\mathbb{R}, \lambda>0\}$ is set of extremal functions for Hardy-Sobolev inequality \eqref{defhsi}.
    From Proposition \ref{propev}, we directly obtain
    \begin{proposition}\label{propevl}
    Suppose $1<p<N$ and $0<\beta<p$. Then there exists a constant $\tau=\tau(N,p,\beta)>0$ such that for any function $v\in L^2_{\beta,*}(\mathbb{R}^N)$ orthogonal to $T_{U} \mathcal{M}_\beta$, it holds that
    \begin{equation*}
    \begin{split}
    & \int_{\mathbb{R}^N}\left[|\nabla U|^{p-2}|\nabla v|^2+(p-2)|\nabla U|^{p-4}(\nabla U\cdot\nabla v)^2\right]\mathrm{d}x
    \geq \left[(p^*_{\beta}-1)+2\tau\right]
    \int_{\mathbb{R}^N}|x|^{-\beta}U^{p^*_\beta-2}v^2\mathrm{d}x.
    \end{split}
    \end{equation*}
    \end{proposition}

    Enlightened by \cite{FZ22}, we give the following remark which will  be important to give a meaning to the notion of  ``orthogonal to $T_{U} \mathcal{M}_\beta$" for functions which are not necessarily in $L^2_{\beta,*}(\mathbb{R}^N)$.

    \begin{remark}\rm
    For any $\xi\in T_{U} \mathcal{M}_\beta$ it holds $U^{p^*_{\beta}-2}\xi \in L_\beta^{\frac{p^*_{\beta}}{p^*_{\beta}-1}}(\mathbb{R}^N)
    =\left(L_\beta^{p^*_{\beta}}(\mathbb{R}^N)\right)'$, here $L_\beta^{q}(\mathbb{R}^N)$ is the set of measurable functions with the norm $\|\varphi\|_{L_\beta^{q}(\mathbb{R}^N)}:
    =\left(\int_{\mathbb{R}^N}|x|^{-\beta}|\varphi|^{q} \mathrm{d}x\right)^{\frac{1}{q}}$.
    Hence, by abuse of notation, for any function $v\in L_\beta^{p^*_{\beta}}(\mathbb{R}^N)$ we say that $v$ is orthogonal to $T_{U} \mathcal{M}_\beta$ in $L^2_{\beta,*}(\mathbb{R}^N)$ if \[
    \int_{\mathbb{R}^N}|x|^{-\beta}U^{p^*_{\beta}-2} v\xi\mathrm{d}x=0,\quad \forall \xi \in T_{U} \mathcal{M}_\beta.
    \]
    \end{remark}

    Note that, by H\"{o}lder inequality, $L_\beta^{p^*_{\beta}}(\mathbb{R}^N)\subset L^2_{\beta,*}(\mathbb{R}^N)$ if $p^*_{\beta}\geq 2$.
    Hence, the notion of orthogonality introduced above is particularly relevant when $p^*_{\beta}<2$ (equivalently, $p<\frac{2N}{N+2-\beta}$). We also observe that, by Sobolev embedding theorem, the previous remark gives a meaning to the orthogonality to $T_{U} \mathcal{M}_\beta$ for functions in $\mathcal{D}^{1,p}_{0}(\mathbb{R}^N)$ since the Hardy-Sobolev inequality \eqref{defhsi} implies  $\mathcal{D}^{1,p}_{0}(\mathbb{R}^N)\hookrightarrow L_\beta^{p^*_{\beta}}(\mathbb{R}^N)$ continuously.

\subsection{{\bfseries Improvements of spectral gap}}\label{subsectsge}

    Same as \cite[Proposition 3.8]{FZ22}, the goal of this subsection is improving the spectral gap obtained in  Proposition \ref{propevl}, that is, we will give the following spectral gap-type estimates for $p\geq 2$ and $1<p<2$, respectively.

    \begin{lemma}\label{lemsgap}
    Let $2\leq p<N$ and $0<\beta<p$. Given any $\gamma_0>0$, there exists $\overline{\delta}=\overline{\delta}(N,p,\beta,\gamma_0)>0$ such that for any function $v\in \mathcal{D}^{1,p}_{0}(\mathbb{R}^N)$ orthogonal to $T_{U} \mathcal{M}_\beta$ in $L^2_{\beta,*}(\mathbb{R}^N)$ satisfying $\|v\|\leq \overline{\delta}$, we have
    \begin{equation*}
    \begin{split}
    & \int_{\mathbb{R}^N}\left[|\nabla U|^{p-2}|\nabla v|^2+(p-2)|\omega|^{p-2}(|\nabla (U+v)|-|\nabla U|)^2\right]\mathrm{d}x \\
    \geq & \left[(p^*_{\beta}-1)+\tau\right]
    \int_{\mathbb{R}^N}|x|^{-\beta}U^{p^*_{\beta}-2}|v|^2\mathrm{d}x,
    \end{split}
    \end{equation*}
    where $\tau>0$ is given in Proposition \ref{propevl}, and $\omega: \mathbb{R}^{2N}\to \mathbb{R}^N$ is defined in analogy to Lemma \ref{lemui1p}:
    \begin{eqnarray*}
    \omega=\omega(\nabla U,\nabla (U+v))=
    \left\{ \arraycolsep=1.5pt
       \begin{array}{ll}
        \nabla U,\ \ &{\rm if}\ \ |\nabla U|<|\nabla (U+v)|,\\[3mm]
        \left(\frac{|\nabla (U+v)|}{|\nabla U|}\right)^{\frac{1}{p-2}}\nabla (U+v),\ \ &{\rm if}\ \  |\nabla (U+v)|\leq |\nabla U|.
        \end{array}
    \right.
    \end{eqnarray*}
    \end{lemma}

    \begin{proof}
    We argue by contradiction. If the statement of this lemma fails, then there exists a sequence $0\not\equiv v_n\to 0$ in $\mathcal{D}^{1,p}_{0}(\mathbb{R}^N)$, with $v_n$ orthogonal to $T_{U} \mathcal{M}_\beta$, such that
    \begin{align}\label{evbc}
    & \int_{\mathbb{R}^N}\left[|\nabla U|^{p-2}|\nabla v_n|^2+(p-2)|\omega_n|^{p-2}(|\nabla (U+v_n)|-|\nabla U|)^2\right]\mathrm{d}x \nonumber\\
    < & \left[(p^*_{\beta}-1)+\tau\right]
    \int_{\mathbb{R}^N}|x|^{-\beta}U^{p^*_{\beta}-2}|v_n|^2\mathrm{d}x,
    \end{align}
    where $\omega_n$ corresponds to $v_n$ as in the statement.
    Let
    \[
    \varepsilon_n:=\|v_n\|_{\mathcal{D}^{1,2}_{0,*}(\mathbb{R}^N)}
    =\left(\int_{\mathbb{R}^N}|\nabla U|^{p-2} |\nabla v_n|^2  \mathrm{d}x\right)^{\frac{1}{2}},\quad \widehat{v}_n=\frac{v_n}{\varepsilon_n}.
    \]
    Note that, since $p\geq 2$, it follows by H\"{o}lder inequality that
    \[
    \int_{\mathbb{R}^N}|\nabla U|^{p-2}|\nabla v_n|^2\mathrm{d}x
    \leq \left(\int_{\mathbb{R}^N}|\nabla U|^{p}\mathrm{d}x\right)^{1-\frac{p}{2}}
    \left(\int_{\mathbb{R}^N}|\nabla v_n|^{p}\mathrm{d}x\right)^{\frac{p}{2}}
    \to 0,
    \]
    hence $\varepsilon_n\to 0$, as $n\to \infty$.
    Since $\|\widehat{v}_n\|_{\mathcal{D}^{1,2}_{0,*}(\mathbb{R}^N)}=1$, Proposition \ref{propcet} implies that, up to a subsequence, $\widehat{v}_n\rightharpoonup \widehat{v}$ in $\mathcal{D}^{1,2}_{0,*}(\mathbb{R}^N)$ and $\widehat{v}_n \rightarrow\widehat{v}$ in $L^2_{\beta,*}(\mathbb{R}^N)$ for some $\widehat{v}\in \mathcal{D}^{1,2}_{0,*}(\mathbb{R}^N)$. Also, since $p\geq 2$, it follows from \eqref{evbc} that
    \[
    1=\int_{\mathbb{R}^N}|\nabla U|^{p-2}|\nabla \widehat{v}_n|^2 \mathrm{d}x \leq  \left[(p^*_{\beta}-1)+\tau\right]
    \int_{\mathbb{R}^N}|x|^{-\beta}U^{p^*_{\beta}-2}
    |\widehat{v}_n|^2\mathrm{d}x,
    \]
    then we deduce that
    \[
    \|\widehat{v}_n\|_{L^2_{\beta,*}(\mathbb{R}^N)}
    =\int_{\mathbb{R}^N}|x|^{-\beta}U^{p^*_{\beta}-2}
    |\widehat{v}_n|^2\mathrm{d}x\geq c
    \]
    for some $c>0$.

    Fix $R>1$ which can be chosen arbitrarily large, set
    \begin{equation*}
    \begin{split}
    \mathcal{R}_n:=\{2|\nabla U|\geq |\nabla v_n|\}&,\quad \mathcal{S}_n:=\{2|\nabla U|< |\nabla v_n|\},  \\
    \mathcal{R}_{n,R}:=\left(B_R\backslash B_{\frac{1}{R}}\right)\cap \mathcal{R}_n&,\quad
    \mathcal{S}_{n,R}:=\left(B_R\backslash B_{\frac{1}{R}}\right)\cap \mathcal{S}_n,
    \end{split}
    \end{equation*}
    thus $B_R\backslash B_{\frac{1}{R}}=\mathcal{R}_{n,R} \cup \mathcal{S}_{n,R}$.
    Since the integrand in the left hand side of \eqref{evbc} is nonnegative, we have
    \begin{align}\label{evbcb}
    & \int_{B_R\backslash B_{\frac{1}{R}}}\left[|\nabla U|^{p-2}|\nabla \widehat{v}_n|^2+(p-2)|\omega_n|^{p-2}\left(\frac{|\nabla (U+v_n)|-|\nabla U|}{\varepsilon_n}\right)^2\right]\mathrm{d}x \nonumber\\
    < & \left[(p^*_{\beta}-1)+\tau\right]
    \int_{\mathbb{R}^N}|x|^{-\beta}U^{p^*_{\beta}-2}|\widehat{v}_n|^2
    \mathrm{d}x.
    \end{align}
    From Proposition \ref{propcet}, we have
    \begin{align*}
    c\leq \int_{\mathbb{R}^N}|x|^{-\beta}
    U^{p^*_{\beta}-2}|\widehat{v}_n|^2\mathrm{d}x
    \leq C_1 \int_{\mathbb{R}^N}|\nabla U|^{p-2} |\nabla \widehat{v}_n|^2  \mathrm{d}x=C_1,
    \end{align*}
    thus
    \begin{align*}
    & \int_{\mathcal{R}_{n,R}}\left| \nabla  U \right|^{p-2}|\nabla \widehat{v}_n|^2 \mathrm{d}x
    +\varepsilon^{p-2}_n\int_{\mathcal{S}_{n,R}}|\nabla \widehat{v}_n|^p \mathrm{d}x \leq C_2.
    \end{align*}
    Then we obtain
    \[
    \varepsilon_n^{-2}\int_{\mathcal{S}_{n,R}}|\nabla U|^{p}\mathrm{d}x
    \leq \frac{\varepsilon_n^{p-2}}{2^p}\int_{\mathcal{S}_{n,R}} |\nabla \widehat{v}_n|^{p}\mathrm{d}x
    \leq C_3,
    \]
    and since
    \[
    0<c(R)\leq |\nabla U|\leq C(R)\quad \mbox{inside}\quad B_R\backslash B_{\frac{1}{R}},\quad \forall R>1,
    \]
    for some constants $c(R)\leq C(R)$ depending only on $R$, we conclude that
    \begin{align}\label{pb2csirt0}
    |\mathcal{S}_{n,R}|\to 0\quad \mbox{as}\quad n\to \infty,\quad \forall R>1.
    \end{align}
    Here $|\mathcal{S}_{n,R}|$ denote the Lebesgue measure of $\mathcal{S}_{n,R}$. Now, writing
    \[
    \widehat{v}_n=\widehat{v}+\varphi_n,\quad\mbox{with}\quad \varphi_n:= \widehat{v}_n-\widehat{v},
    \]
    since $R>1$ is arbitrary, we have
    \[
    \varphi_n \rightharpoonup 0 \quad \mbox{locally in}\quad   \mathcal{D}^{1,2}_{0,*}(\mathbb{R}^N\backslash\{\mathbf{0}\}).\]
    Moreover, we have $|\omega_n|\to |\nabla U|$ a.e. in $\mathbb{R}^N$. Then, let us rewrite
    \begin{equation*}
    \begin{split}
    \left(\frac{|\nabla (U+v_n)|-|\nabla U|}{\varepsilon_n}\right)^2
    = & \left(\left[\int^1_0\frac{\nabla U+ t\nabla v_n}{|\nabla U+ t\nabla v_n|}\mathrm{d}t\right]\cdot \nabla \widehat{v}_n\right)^2 \\
    = & \left(\left[\int^1_0\frac{\nabla U+ t\nabla v_n}{|\nabla U+ t\nabla v_n|}\mathrm{d}t\right]\cdot \nabla (\widehat{v}+\varphi_n)\right)^2.
    \end{split}
    \end{equation*}
    Hence, if we set
    \[
    f_{n,1}=\left[\int^1_0\frac{\nabla U+ t\nabla v_n}{|\nabla U+ t\nabla v_n|}\mathrm{d}t\right]\cdot \nabla \widehat{v},\quad
    f_{n,2}=\left[\int^1_0\frac{\nabla U+ t\nabla v_n}{|\nabla U+ t\nabla v_n|}\mathrm{d}t\right]\cdot \nabla \varphi_n,
    \]
    since $\frac{\nabla U+ t\nabla v_n}{|\nabla U+ t\nabla v_n|}\to \frac{\nabla U}{|\nabla U|}$ a.e., it follows from Lebesgue's dominated convergence theorem that
    \[
    f_{n,1}\to \frac{\nabla U}{|\nabla U|}\cdot\nabla \widehat{v}\quad \mbox{locally in}\quad L^2(\mathbb{R}^N\backslash\{\mathbf{0}\}),\quad f_{n,2}\chi_{\mathcal{R}_n}\rightharpoonup 0\quad \mbox{locally in}\quad L^2(\mathbb{R}^N\backslash\{\mathbf{0}\}).
    \]
    Thus, the left hand side of \eqref{evbcb} from below as follows:
    \begin{align}\label{evbcbb}
    & \int_{\mathcal{R}_{n,R}}\left[|\nabla U|^{p-2}|\nabla \widehat{v}_n|^2+(p-2)|\omega_n|^{p-2}\left(\frac{|\nabla (U+v_n)|-|\nabla U|}{\varepsilon_n}\right)^2\right]\mathrm{d}x \nonumber\\
    = & \int_{\mathcal{R}_{n,R}}\left[|\nabla U|^{p-2}\left(|\nabla \widehat{v}|^2+2 \nabla \varphi_n\cdot \nabla \widehat{v}\right)+(p-2)|\omega_n|^{p-2}\left(f_{n,1}^2
    +2f_{n,1}f_{n,2}\right)\right]\mathrm{d}x \nonumber\\
    & + \int_{\mathcal{R}_{n,R}}\left[|\nabla U|^{p-2} |\nabla \varphi|^2+(p-2)|\omega_n|^{p-2}f_{n,2}^2\right]\mathrm{d}x \nonumber\\
    \geq & \int_{\mathcal{R}_{n,R}}\left[|\nabla U|^{p-2}\left(|\nabla \widehat{v}|^2+2 \nabla \varphi_n\cdot \nabla \widehat{v}\right)+(p-2)|\omega_n|^{p-2}
    \left(f_{n,1}^2+2f_{n,1}f_{n,2}\right)\right]\mathrm{d}x.
    \end{align}
    Then, combining the convergences
    \begin{equation*}
    \begin{split}
    & \nabla \varphi_n \chi_{\mathcal{R}_n}\rightharpoonup 0,
    \quad f_{n,1}\to \frac{\nabla U}{|\nabla U|}\cdot\nabla \widehat{v},
    \quad f_{n,2}\chi_{\mathcal{R}_n}\rightharpoonup 0,
    \quad \mbox{locally in}\quad L^2(\mathbb{R}^N\backslash\{\mathbf{0}\}),\\
    & |\omega_n|\to |\nabla U|\quad \mbox{a.e.}, \quad |(B_R\backslash B_{\frac{1}{R}})\backslash\mathcal{R}_{n,R}|=|\mathcal{S}_{n,R}|\to 0,
    \end{split}
    \end{equation*}
    with the fact that
    \[
    |\omega_n|^{p-2}\leq C(p)|\nabla U|^{p-2},
    \]
    by Lebesgue's dominated convergence theorem, we deduce that
    \begin{small}
    \begin{align*}
    & \liminf_{i\to \infty}\int_{\mathcal{R}_{n,R}}\left[|\nabla U|^{p-2}\left(|\nabla \widehat{v}|^2+2 \nabla \varphi_n\cdot \nabla \widehat{v}\right)+(p-2)|\omega_n|^{p-2}
    \left(f_{n,1}^2+2f_{n,1}f_{n,2}\right)\right]\mathrm{d}x
    \\
    \to & \int_{B_R\backslash B_{\frac{1}{R}}}\left[|\nabla U|^{p-2}|\nabla \widehat{v} |^2+(p-2)|\nabla U|^{p-2}\left(\frac{\nabla U\cdot\nabla \widehat{v} }{| \nabla U|}\right)^2\right]\mathrm{d}x,
    \end{align*}
    \end{small}
    then combining \eqref{evbcb} with \eqref{evbcbb} we have
    \begin{small}
    \begin{align}\label{evbcbbcf}
    & \liminf_{n\to \infty}\int_{B_R\backslash B_{\frac{1}{R}}}\left[|\nabla U|^{p-2}|\nabla \widehat{v}_n|^2+(p-2)|\omega_n|^{p-2}\left(\frac{|\nabla (U+v_n)|-|\nabla U|}{\varepsilon_n}\right)^2\right]\mathrm{d}x
    \nonumber\\
    \geq & \int_{B_R\backslash B_{\frac{1}{R}}}\left[|\nabla U|^{p-2}|\nabla \widehat{v} |^2+(p-2)|\nabla U|^{p-2}\left(\frac{\nabla U\cdot\nabla \widehat{v} }{| \nabla U|}\right)^2\right]\mathrm{d}x.
    \end{align}
    \end{small}
    Recalling \eqref{evbcb} and since $R>1$ is arbitrary, \eqref{evbcbbcf} proves that
    \begin{align*}
    & \int_{\mathbb{R}^N}\left[|\nabla U|^{p-2}|\nabla \widehat{v}|^2+(p-2)|\nabla U|^{p-2}\left(\frac{\nabla U\cdot\nabla \widehat{v}}{|\nabla U|}\right)^2\right]\mathrm{d}x \\
    \leq & \left[(p^*_{\beta}-1)+\tau\right]
    \int_{\mathbb{R}^N}|x|^{\beta}U^{p^*_{\beta}-2}|\widehat{v}|^2
    \mathrm{d}x,
    \end{align*}
    which contradicts Proposition \ref{propevl} due to the orthogonality of $\widehat{v}$ to $T_{U} \mathcal{M}_\beta$ (being the strong $L^2_{\beta,*}(\mathbb{R}^N)$ limit of $\widehat{v}_n$).
    \end{proof}

    \begin{lemma}\label{lemsgap2}
    Let $1<p<2\leq N$ and $0<\beta<p$. Given any $\gamma_0>0$, $C_1>0$ there exists $\overline{\delta}=\overline{\delta}(N,p,\beta,\gamma_0,C_1)>0$ such that for any function $v\in \mathcal{D}^{1,p}_{0}(\mathbb{R}^N)$ orthogonal to $T_{U} \mathcal{M}_\beta$ in $L^2_{\beta,*}(\mathbb{R}^N)$ satisfying $\|v\|\leq \overline{\delta}$, the following holds:
    \begin{itemize}
    \item[$(i)$]
    when $1<p\leq\frac{2N}{N+2-\beta}$, we have
    \begin{small}
    \begin{align*}
    & \int_{\mathbb{R}^N}\left[
    |\nabla U|^{p-2}|\nabla v|^2
    +(p-2)|\omega|^{p-2}(|\nabla (U+v)|-|\nabla U|)^2
    +\gamma_0 \min\{|\nabla v|^p,|\nabla U|^{p-2}|\nabla v|^2\}
    \right]\mathrm{d}x \\
    \geq & \left[(p^*_{\beta}-1)+\tau\right]
    \int_{\mathbb{R}^N}|x|^{-\beta}
    \frac{(U+C_1|v|)^{p^*_{\beta}}}{U^2+|v|^2}|v|^2\mathrm{d}x;
    \end{align*}
    \end{small}
    \item[$(ii)$]
    when $\frac{2N}{N+2-\beta}< p<2$, we have
    \begin{small}
    \begin{align*}
    & \int_{\mathbb{R}^N}\left[
    |\nabla U|^{p-2}|\nabla v|^2
    +(p-2)|\omega|^{p-2}(|\nabla (U+v)|-|\nabla U|)^2
    +\gamma_0 \min\{|\nabla v|^p,|\nabla U|^{p-2}|\nabla v|^2\}
    \right]\mathrm{d}x \\
    \geq & \left[(p^*_{\beta}-1)+\tau\right]
    \int_{\mathbb{R}^N}|x|^{-\beta}U^{p^*_{\beta}-2}|v|^2\mathrm{d}x,
    \end{align*}
    \end{small}
    \end{itemize}
    where $\tau>0$ is given in Proposition \ref{propevl}, and $\omega: \mathbb{R}^{2N}\to \mathbb{R}^N$ is defined in analogy to Lemma \ref{lemui1p}:
    \begin{eqnarray*}
    \omega=\omega(\nabla U,\nabla (U+v))=
    \left\{ \arraycolsep=1.5pt
       \begin{array}{ll}
        \left(\frac{|\nabla (U+v)|}{(2-p)|\nabla (U+v)|+(p-1)|\nabla U|}\right)^{\frac{1}{p-2}}\nabla U,\ \ &{\rm if}\ \  |\nabla U|<|\nabla (U+v)|\\[3mm]
        \nabla U,\ \ &{\rm if}\ \ |\nabla (U+v)|\leq |\nabla U|
        \end{array}.
    \right.
    \end{eqnarray*}
    \end{lemma}

    \begin{proof}
    The proof is similar to the proof of Lemma \ref{lemsgap}, but it is more complicated.
    We argue by contradiction in these two cases.

    $\bullet$ {\em The case $1<p\leq\frac{2N}{N+2-\beta}$} which implies $p^*_{\beta}\leq2$.
    Suppose the inequality does not hold, then there exists a sequence $0\not\equiv v_n\to 0$ in $\mathcal{D}^{1,p}_{0}(\mathbb{R}^N)$, with $v_n$ orthogonal to $T_{U} \mathcal{M}_\beta$, such that
    \begin{small}
    \begin{align}\label{evbc2}
    & \int_{\mathbb{R}^N}\Big[
    |\nabla U|^{p-2}|\nabla v_n|^2
    +(p-2)|\omega_n|^{p-2}(|\nabla (U+v_n)|-|\nabla U|)^2
    \nonumber\\& \quad \quad +\gamma_0 \min\{|\nabla v_n|^p,|\nabla U|^{p-2}|\nabla v_n|^2\}
    \Big]\mathrm{d}x \nonumber\\
     <  & \left[(p^*_{\beta}-1)+\tau\right]
    \int_{\mathbb{R}^N}|x|^{-\beta}
    \frac{(U+C_1|v_n|)^{p^*_{\beta}}}{U^2+|v_n|^2}|v_n|^2\mathrm{d}x,
    \end{align}
    \end{small}
    where $\omega_n$ corresponds to $v_n$ as in the statement.
    Let
    \begin{align}\label{defevh}
    \varepsilon_n:=\left(\int_{\mathbb{R}^N}(|\nabla U|+|\nabla v_n|)^{p-2}|\nabla v_n|^2 \mathrm{d}x\right)^{\frac{1}{2}},\quad \widehat{v}_n=\frac{v_n}{\varepsilon_n}.
    \end{align}
    Note that, since $1<p<2$, it follows by H\"{o}lder inequality that
    \begin{align*}
    \int_{\mathbb{R}^N}(|\nabla U|+|\nabla v_n|)^{p-2}|\nabla v_n|^2 dx\leq & \int_{\mathbb{R}^N}|\nabla v_n|^{p-2}|\nabla v_n|^2\mathrm{d}x
    = \int_{\mathbb{R}^N}|\nabla v_n|^{p}\mathrm{d}x
    \to 0,
    \end{align*}
    hence $\varepsilon_n\to 0$, as $n\to \infty$.

    Since the integrand in the left hand side of \eqref{evbc2} is nonnegative, for any $R>1$ we have
    \begin{align}\label{evbcb2}
    & \int_{B_R\backslash B_{\frac{1}{R}}}\bigg[
    |\nabla U|^{p-2}|\nabla \widehat{v}_n|^2
    +(p-2)|\omega_n|^{p-2}\left(\frac{|\nabla (U+v_n)|-|\nabla U|}{\varepsilon_n}\right)^2 \nonumber\\
    & \quad\quad + \gamma_0 \min\{\varepsilon_n^{p-2}|\nabla \widehat{v}_n|^p,|\nabla U|^{p-2}|\nabla \widehat{v}_n|^2\}
    \bigg]\mathrm{d}x \nonumber\\
    < & \left[(p^*_{\beta}-1)+\tau\right]
    \int_{\mathbb{R}^N}|x|^{-\beta}\frac{(U+C_1|v_n|)^{p^*_{\beta}}}
    {U^2+|v_n|^2}|\widehat{v}_n|^2\mathrm{d}x.
    \end{align}
    Now, same as the proof of Lemma \ref{lemsgap}, let us
    fix $R>1$ which can be chosen arbitrarily large, and set
    \begin{align}\label{defrsi}
    \mathcal{R}_n:=\{2|\nabla U|\geq |\nabla v_n|\}&,\quad \mathcal{S}_n:=\{2|\nabla U|< |\nabla v_n|\},  \nonumber\\
    \mathcal{R}_{n,R}:=\left(B_R\backslash B_{\frac{1}{R}}\right)\cap \mathcal{R}_n&,\quad
    \mathcal{S}_{n,R}:=\left(B_R\backslash B_{\frac{1}{R}}\right)\cap \mathcal{S}_n,
    \end{align}
    thus $B_R\backslash B_{\frac{1}{R}}=\mathcal{R}_{n,R} \cup \mathcal{S}_{n,R}$.
    From \cite[(2.2)]{FZ22}, that is, for $1<p<2$, there exists $c(p)>0$ such that
    \begin{equation}\label{evbcb2i}
    p|x|^{p-2}|y|^2+p(p-2)|\omega|^{p-2}(|x|-|x+y|)^2\geq c(p)\frac{|x|}{|x|+|y|}|x|^{p-2}|y|^2,\quad \forall x\neq 0,\forall y\in\mathbb{R}^N,
    \end{equation}
    we have
    \begin{align*}
    & \left|\frac{\nabla U}{\varepsilon_n}\right|^{p-2}|\nabla \widehat{v}_n|^2
    +(p-2)\left|\frac{\omega_n}{\varepsilon_n}\right|^{p-2}
    \left(\left|\frac{\nabla U }{\varepsilon_n}+\nabla \widehat{v}_n\right|-\left|\frac{\nabla  U}{\varepsilon_n}\right|\right)^2 \\
    \geq & c(p)\frac{|\nabla U|/\varepsilon_n}{|\nabla U|/\varepsilon_n+|\nabla \widehat{v}_n|}\left|\frac{\nabla  U}{\varepsilon_n}\right|^{p-2}|\nabla \widehat{v}_n|^2,
    \end{align*}
    then,
    \[
    \left| \nabla U \right|^{p-2}|\nabla \widehat{v}_n|^2
    +(p-2)\left| \omega_n \right|^{p-2}\left(\frac{|\nabla (U+v_n)|-|\nabla  U| }{\varepsilon_n}\right)^2
     \geq c(p) \left| \nabla  U \right|^{p-2}|\nabla \widehat{v}_n|^2,\quad \mbox{in}\quad \mathcal{R}_{n,R}.
    \]
    Therefore, combining this bound with \eqref{evbcb2}, we obtain
    \begin{align}\label{evbcb2l}
    & c(p)\int_{\mathcal{R}_{n,R}}\left| \nabla  U \right|^{p-2}|\nabla \widehat{v}_n|^2 \mathrm{d}x
    +\gamma_0\varepsilon^{p-2}_n\int_{\mathcal{S}_{n,R}}|x|^\alpha|\nabla \widehat{v}_n|^p \mathrm{d}x \nonumber\\
    \leq & \int_{B_R\backslash B_{\frac{1}{R}}}\bigg[
    |\nabla U|^{p-2}|\nabla \widehat{v}_n|^2
    +(p-2)|\omega_n|^{p-2}\left(\frac{|\nabla (U+v_n)|-|\nabla U|}{\varepsilon_n}\right)^2 \nonumber\\
    & \quad\quad + \gamma_0 \min\{\varepsilon_n^{p-2}|\nabla \widehat{v}_n|^p,|\nabla U|^{p-2}|\nabla \widehat{v}_n|^2\}
    \bigg]\mathrm{d}x \nonumber\\
    < & \left[(p^*_{\beta}-1)+\tau\right]
    \int_{\mathbb{R}^N}|x|^{-\beta}
    \frac{(U+C_1|v_n|)^{p^*_{\beta}}}{U^2+|v_n|^2}|\widehat{v}_n|^2
    \mathrm{d}x.
    \end{align}
    In particular, this implies that
    \begin{align}\label{evbcb2li}
    1= & \varepsilon^{-2}_n\int_{\mathbb{R}^N}(|\nabla U|+|\nabla v_n|)^{p-2}|\nabla v_n|^2 dx \nonumber\\
    \leq & C(p)\left[\int_{\mathcal{R}_{n}}\left| \nabla  U \right|^{p-2}|\nabla \widehat{v}_n|^2 \mathrm{d}x
    +\varepsilon^{p-2}_n\int_{\mathcal{S}_{n}}|\nabla \widehat{v}_n|^p \mathrm{d}x\right] \nonumber\\
    \leq & C(N,p,\gamma_0)\left[(p^*_{\beta}-1)+\tau\right]
    \int_{\mathbb{R}^N}|x|^{-\beta}
    \frac{(U+C_1|v_n|)^{p^*_{\beta}}}{U^2+|v_n|^2}|\widehat{v}_n|^2
    \mathrm{d}x.
    \end{align}
    Furthermore, thanks to \eqref{propcetlcpi} in Corollary \ref{propcetlpi}, for $n$ large enough so that $\varepsilon_n$ small we have
    \begin{align}\label{evbcb2lib}
    \int_{\mathbb{R}^N}|x|^{-\beta}
    \frac{(U+C_1|v_n|)^{p^*_{\beta}}}{U^2+|v_n|^2}|\widehat{v}_n|^2
    \mathrm{d}x
    \leq & C(N,p,C_1)\int_{\mathbb{R}^N}|x|^{-\beta}
    (U+|v_n|)^{p^*_{\beta}-2}|\widehat{v}_n|^2\mathrm{d}x \nonumber\\
    \leq & C(N,p,\beta,C_1)\int_{\mathbb{R}^N}
    (|\nabla U|+|\nabla v_n|)^{p-2}|\nabla \widehat{v}_n|^2\mathrm{d}x \nonumber\\
    = & C(N,p,\beta,C_1).
    \end{align}
    Hence, combining \eqref{evbcb2l} with \eqref{evbcb2lib}, by the definition of $\mathcal{S}_{i,R}$ we have
    \begin{equation*}
    \begin{split}
    \varepsilon^{-2}_n\int_{\mathcal{S}_{n,R}}|\nabla U|^p \mathrm{d}x
    \leq \frac{\varepsilon^{-2}_n}{2^p}\int_{\mathcal{S}_{n,R}}|\nabla v_n|^p \mathrm{d}x
    = \frac{\varepsilon^{p-2}_n}{2^p}\int_{\mathcal{S}_{n,R}}|\nabla \widehat{v}_n|^p \mathrm{d}x
    \leq C(N,p,\beta,C_1),
    \end{split}
    \end{equation*}
    then since
    \[
    0<c(R)\leq |\nabla U|\leq C(R)\quad \mbox{inside}\quad B_R\backslash B_{\frac{1}{R}},\quad \forall R>1,
    \]
    for some constants $c(R)\leq C(R)$ depending only on $R$, we conclude that
    \begin{equation}\label{csirt0}
    |\mathcal{S}_{n,R}|\to 0\quad \mbox{as}\quad n\to \infty,\quad \forall R>1.
    \end{equation}
    Now, from \eqref{defevh} we have
    \[
    \int_{\mathbb{R}^N}(|\nabla U|+\varepsilon_n|\nabla \widehat{v}_n|)^{p-2}|\nabla \widehat{v}_n|^2 \mathrm{d}x\leq 1,
    \]
    then according to Lemma \ref{propcetl}, we deduce that $\widehat{v}_n\rightharpoonup \widehat{v}$ in $\mathcal{D}^{1,p}_0(\mathbb{R}^N)$ for some $\widehat{v}\in \mathcal{D}^{1,p}_{0}(\mathbb{R}^N)\cap L^2_{\beta,*}(\mathbb{R}^N)$, and
    \begin{equation}\label{csirtc}
    \int_{\mathbb{R}^N}|x|^{-\beta}
    \frac{(U+C_1|v_n|)^{p^*_{\beta}}}{U^2+|v_n|^2}|\widehat{v}_n|^2
    \mathrm{d}x
    \to \int_{\mathbb{R}^N}|x|^{-\beta}U^{p^*_{\beta}-2}
    |\widehat{v}|^2\mathrm{d}x,
    \end{equation}
    as $n\to \infty$, for any $C_1\geq 0$.
    Also, using \eqref{evbcb2l} and \eqref{evbcb2lib} again we have
    \begin{equation*}
    \int_{\mathcal{R}_{n,R}}\left| \nabla  U \right|^{p-2}|\nabla \widehat{v}_n|^2 \mathrm{d}x
    \leq C(N,p,\beta,C_1),
    \end{equation*}
    therefore \eqref{csirt0} and $\widehat{v}_n\rightharpoonup \widehat{v}$ in $\mathcal{D}^{1,p}_0(\mathbb{R}^N)$ imply that, up to a subsequence,
    \[
    \widehat{v}_n \chi_{\mathcal{R}_{n,R}}\rightharpoonup \widehat{v} \chi_{B_R\backslash B_{\frac{1}{R}}}\quad \mbox{in}\quad \mathcal{D}^{1,2}_{0,*}(\mathbb{R}^N),\quad \forall R>1.
    \]
    Here $\chi_E$ denotes that $\chi_E=1$ if $x\in E$ and  $\chi_E=0$ if $x\notin E$.
    In addition, letting $n\to \infty$ in \eqref{evbcb2li} and \eqref{evbcb2lib}, and using \eqref{csirtc}, we deduce that
    \begin{equation}\label{csirtl1}
    0<c(N,p,\beta,C_1,\gamma_0)
    \leq \|\widehat{v}\|_{L^2_{\beta,*}(\mathbb{R}^N)}
    \leq C(N,p,\beta,C_1).
    \end{equation}
    Let us write
    \[
    \widehat{v}_n=\widehat{v}+\varphi_n,\quad\mbox{with}\quad \varphi_n:= \widehat{v}_n-\widehat{v},
    \]
    we have
    \[
    \varphi_n \rightharpoonup 0 \quad \mbox{in}\quad   \mathcal{D}^{1,p}_{0}(\mathbb{R}^N)
    \quad \mbox{and}\quad \varphi_n \chi_{\mathcal{R}_{n}}\rightharpoonup 0
    \quad \mbox{locally in} \quad \mathcal{D}^{1,2}_{0,*}(\mathbb{R}^N\setminus\{\mathbf{0}\}).
    \]
    We now look at the left side of \eqref{evbcb2}.
    The strong convergence $v_n\to 0$ in $\mathcal{D}^{1,p}_{0}(\mathbb{R}^N)$ implies that, $|\omega_n|\to |\nabla U|$ a.e. in $\mathbb{R}^N$. Then, let us rewrite
    \begin{align*}
    \left(\frac{|\nabla (U+v_n)|-|\nabla U|}{\varepsilon_n}\right)^2
    = & \left(\left[\int^1_0\frac{\nabla U+ t\nabla v_n}{|\nabla U+ t\nabla v_n|}\mathrm{d}t\right]\cdot \nabla \widehat{v}_n\right)^2 \\
    = & \left(\left[\int^1_0\frac{\nabla U+ t\nabla v_n}{|\nabla U+ t\nabla v_n|}\mathrm{d}t\right]\cdot \nabla (\widehat{v}+\varphi_n)\right)^2.
    \end{align*}
    Hence, if we set
    \[
    f_{n,1}=\left[\int^1_0\frac{\nabla U+ t\nabla v_n}{|\nabla U+ t\nabla v_n|}\mathrm{d}t\right]\cdot \nabla \widehat{v},\quad
    f_{n,2}=\left[\int^1_0\frac{\nabla U+ t\nabla v_n}{|\nabla U+ t\nabla v_n|}\mathrm{d}t\right]\cdot \nabla \varphi_n,
    \]
    since $\frac{\nabla U+ t\nabla v_n}{|\nabla U+ t\nabla v_n|}\to \frac{\nabla U}{|\nabla U|}$ a.e., it follows from Lebesgue's dominated convergence theorem that
    \[
    f_{n,1}\to \frac{\nabla U}{|\nabla U|}\cdot\nabla \widehat{v}\quad \mbox{locally in}\quad L^2(\mathbb{R}^N\backslash\{\mathbf{0}\}),\quad f_{n,2}\chi_{\mathcal{R}_n}\rightharpoonup 0\quad \mbox{locally in}\quad L^2(\mathbb{R}^N\backslash\{\mathbf{0}\}).
    \]
    Thus, the left hand side of \eqref{evbcb2} from below as follows:
    \begin{align}\label{evbcbb2}
    & \int_{\mathcal{R}_{n,R}}\left[|\nabla U|^{p-2}|\nabla \widehat{v}_n|^2+(p-2)|\omega_n|^{p-2}\left(\frac{|\nabla (U+v_n)|-|\nabla U|}{\varepsilon_n}\right)^2\right]\mathrm{d}x \nonumber\\
    = & \int_{\mathcal{R}_{n,R}}\left[|\nabla U|^{p-2}\left(|\nabla \widehat{v}|^2+2 \nabla \varphi_n\cdot \nabla \widehat{v}\right)+(p-2)|\omega_n|^{p-2}
    \left(f_{n,1}^2+2f_{n,1}f_{n,2}\right)\right]\mathrm{d}x \nonumber\\
    & + \int_{\mathcal{R}_{n,R}}\left[|\nabla U|^{p-2} |\nabla \varphi_n|^2+(p-2)|\omega_n|^{p-2}f_{n,2}^2\right]\mathrm{d}x
    \nonumber\\
    \geq & \int_{\mathcal{R}_{n,R}}\left[|\nabla U|^{p-2}\left(|\nabla \widehat{v}|^2+2 \nabla \varphi_n\cdot \nabla \widehat{v}\right)+(p-2)|\omega_n|^{p-2}
    \left(f_{n,1}^2+2f_{n,1}f_{n,2}\right)\right]\mathrm{d}x,
    \end{align}
    where the last inequality follows from the nonnegativity of $\left[|\nabla U|^{p-2} |\nabla \varphi_n|^2+(p-2)|\omega_n|^{p-2}f_{n,2}^2\right]$ (thanks to \eqref{evbcb2i} and the fact that $f_{n,2}^2\leq |\nabla \varphi_n|^2$). Then, combining the convergence
    \begin{equation*}
    \begin{split}
    & \nabla \varphi_n \chi_{\mathcal{R}_n}\rightharpoonup 0,
    \quad f_{n,1}\to \frac{\nabla U}{|\nabla U|}\cdot\nabla \widehat{v},
    \quad f_{n,2}\chi_{\mathcal{R}_n}\rightharpoonup 0,
    \quad \mbox{locally in}\quad L^2(\mathbb{R}^N\backslash\{\mathbf{0}\}),\\
    & |\omega_n|\to |\nabla U|\quad \mbox{a.e.}, \quad |(B_R\backslash B_{\frac{1}{R}})\backslash\mathcal{R}_{n,R}|=|\mathcal{S}_{n,R}|\to 0,
    \end{split}
    \end{equation*}
    with the fact that
    \[
    |\omega_n|^{p-2}\leq C(p)|\nabla U|^{p-2},
    \]
    by Lebesgue's dominated convergence theorem, we deduce that
    \begin{align*}
    & \liminf_{n\to \infty}\int_{\mathcal{R}_{n,R}}\left[|\nabla U|^{p-2}\left(|\nabla \widehat{v}|^2+2 \nabla \varphi_n\cdot \nabla \widehat{v}\right)+(p-2)|\omega_n|^{p-2}
    \left(f_{n,1}^2+2f_{n,1}f_{n,2}\right)\right]\mathrm{d}x \\
    \to & \int_{B_R\backslash B_{\frac{1}{R}}}\left[|\nabla U|^{p-2}|\nabla \widehat{v} |^2+(p-2)|\nabla U|^{p-2}\left(\frac{\nabla U\cdot\nabla \widehat{v} }{| \nabla U|}\right)^2\right]\mathrm{d}x,
    \end{align*}
    thus from \eqref{evbcbb2} we obtain
    \begin{align*}
    & \liminf_{n\to \infty}\int_{\mathcal{R}_{n,R}}\left[|\nabla U|^{p-2}|\nabla \widehat{v}_n|^2+(p-2)|\omega_n|^{p-2}\left(\frac{|\nabla (U+v_n)|-|\nabla U|}{\varepsilon_n}\right)^2\right]\mathrm{d}x \\
    \geq & \int_{B_R\backslash B_{\frac{1}{R}}}\left[|\nabla U|^{p-2}|\nabla \widehat{v} |^2+(p-2)|\nabla U|^{p-2}\left(\frac{\nabla U\cdot\nabla \widehat{v} }{| \nabla U|}\right)^2\right]\mathrm{d}x.
    \end{align*}
    Recalling \eqref{evbcb2l} and \eqref{csirtc}, since $R>1$ is arbitrary and the integrand is nonnegative, this proves that
    \begin{align}\label{uinb2p2}
    & \int_{\mathbb{R}^N}\left[|\nabla U|^{p-2}|\nabla \widehat{v}|^2+(p-2)|\nabla U|^{p-2}\left(\frac{\nabla U\cdot\nabla \widehat{v}}{|\nabla U|}\right)^2\right]\mathrm{d}x \nonumber\\
    \leq  &  \left[(p^*_{\beta}-1)+\tau\right]
    \int_{\mathbb{R}^N}|x|^{-\beta}U^{p^*_{\beta}-2}|\widehat{v}|^2\mathrm{d}x,
    \end{align}
    The orthogonality of $v_n$ (and also of $\widehat{v}_n$) implies that $\widehat{v}$ also is orthogonal to $T_{U} \mathcal{M}_\beta$. Since $\widehat{v}\in L^2_{\beta,*}(\mathbb{R}^N)$, \eqref{csirtl1} and \eqref{uinb2p2} contradict Proposition \ref{propevl}, the proof is complete.

    $\bullet$ {\em The case $\frac{2N}{N+2-\beta}< p<2$} which implies $p^*_{\beta}> 2$. If the statement fails, there exists a sequence $0\not\equiv v_n\to 0$ in $\mathcal{D}^{1,p}_{0}(\mathbb{R}^N)$, with $v_n$ orthogonal to $T_{U} \mathcal{M}_\beta$, such that
    \begin{small}
    \begin{align}\label{evbc2g}
    & \int_{\mathbb{R}^N}\bigg[
    |\nabla U|^{p-2}|\nabla v_n|^2
    +(p-2)|\omega_n|^{p-2}(|\nabla (U+v_n)|-|\nabla U|)^2
    \nonumber\\ &\quad\quad +\gamma_0 \min\{|\nabla v_n|^p,|\nabla U|^{p-2}|\nabla v_n|^2\}
    \bigg]\mathrm{d}x
    \nonumber\\  < &  \left[(p^*_{\beta}-1)+\tau\right]
    \int_{\mathbb{R}^N}|x|^{-\beta}U^{p^*_{\beta}-2}|v_n|^2\mathrm{d}x,
    \end{align}
    \end{small}
    where $\omega_n$ corresponds to $v_n$ as in the statement. As in the case $1<p<\frac{2N}{N+2-\beta}$, we define
    \[
    \varepsilon_n:=\left(\int_{\mathbb{R}^N}(|\nabla U|+|\nabla v_n|)^{p-2}|\nabla v_n|^2 \mathrm{d}x\right)^{\frac{1}{2}},\quad \widehat{v}_n=\frac{v_n}{\varepsilon_n},
    \]
    and we also have $\varepsilon_n\to 0$ as $n\to \infty$.

    Then, we also split $B_R\backslash B_{\frac{1}{R}}=\mathcal{R}_{n,R} \cup \mathcal{S}_{n,R}$, \eqref{evbcb2l} and \eqref{evbcb2li} hold also in this case, with the only difference that the last term in both equations now becomes
    \[
    \left[(p^*_{\beta}-1)+ \tau\right]\int_{\mathbb{R}^N}
    |x|^{-\beta}U^{p^*_{\beta}-2}|v_n|^2\mathrm{d}x,
    \]
    which is much simpler.

    We now observe that, by using H\"{o}lder inequality, we have
    \begin{align*}
    \int_{\mathbb{R}^N}|\nabla \widehat{v}_n|^p \mathrm{d}x
    \leq & \left(\int_{\mathbb{R}^N}(|\nabla U|+|\nabla v_n|)^{p-2}|\nabla \widehat{v}_n|^2 \mathrm{d}x\right)^{\frac{p}{2}}
    \left(\int_{\mathbb{R}^N}(|\nabla U|+|\nabla v_n|)^{p}\mathrm{d}x\right)^{1-\frac{p}{2}} \\
    = & \left(\int_{\mathbb{R}^N}(|\nabla U|+|\nabla v_n|)^{p}\mathrm{d}x\right)^{1-\frac{p}{2}} \\
    \leq & C(p)\left[\left(\int_{\mathbb{R}^N}|\nabla U|^{p}\mathrm{d}x\right)^{1-\frac{p}{2}}
    +\varepsilon_n^{\frac{p(2-p)}{2}}\left(\int_{\mathbb{R}^N} |\nabla \widehat{v}_n|^{p}\mathrm{d}x\right)^{1-\frac{p}{2}}\right]
    \end{align*}
    from which it follows that
    \begin{equation}\label{vihl}
    \begin{split}
    \int_{\mathbb{R}^N}|\nabla \widehat{v}_n|^p \mathrm{d}x
    \leq C(N,p,\beta).
    \end{split}
    \end{equation}
    Thus, up to a subsequence, $\widehat{v}_n\rightharpoonup \widehat{v}$ in $\mathcal{D}^{1,p}_{0}(\mathbb{R}^N)$ and $\widehat{v}_n \to \widehat{v}$ locally in $L^2(\mathbb{R}^N)$. In addition, Then by H\"{o}lder inequality and Sobolev inequality, together with \eqref{vihl}, yield for any $\rho>0$, it holds that
    \begin{align*}
    \int_{\mathbb{R}^N\backslash B_\rho} |x|^{-\beta}U^{p^*_{\beta}-2}|\widehat{v}_n|^2 \mathrm{d}x
    \leq & \left(\int_{\mathbb{R}^N\backslash B_\rho}|x|^{-\beta} U^{p^*_{\beta}}\mathrm{d}x\right)^{1-\frac{2}{p^*_{\beta}}}
    \left(\int_{\mathbb{R}^N\backslash B_\rho}|x|^{-\beta} |\widehat{v}_n|^{p^*_{\beta}}\mathrm{d}x\right)^{\frac{2}{p^*_{\beta}}} \\
    \leq & \frac{1}{\mathcal{S}_\beta}\left(\int_{\mathbb{R}^N}|x|^{-\beta} U^{p^*_{\beta}}\mathrm{d}x\right)^{1-\frac{2}{p^*_{\beta}}}
    \left(\int_{\mathbb{R}^N}|\nabla \widehat{v}_n|^{p}\mathrm{d}x\right)^{\frac{2}{p}}.
    \end{align*}
    Combining \eqref{vihl} and the strong convergence $\widehat{v}_n \to \widehat{v}$ locally in $L^2(\mathbb{R}^N)$, we conclude that $\widehat{v}_n\to \widehat{v}$ in $L^2_{\beta,*}(\mathbb{R}^N)$.

    In particular, letting $n\to \infty$ in the analogue of \eqref{evbcb2li} we obtain
    \begin{equation}\label{csirt0l}
    0<c(N,p,\beta,C_1,\gamma_0)
    \leq \|\widehat{v}\|_{L^2_{\beta,*}(\mathbb{R}^N)}
    \leq C(N,p,\beta,C_1).
    \end{equation}
    Similarly, the analogue of \eqref{evbcb2l} implies that
    \begin{equation}\label{csirt0g}
    |\mathcal{S}_{n,R}|\to 0\quad \mbox{and}\quad \int_{\mathbb{R}^N}|\nabla U|^{p-2}|\nabla \widehat{v}_n|^2 \mathrm{d}x\leq C(N,p,\beta),\quad \forall R>1.
    \end{equation}
    So, it follows from the weak convergence $\widehat{v}_n\rightharpoonup \widehat{v}$ in $\mathcal{D}^{1,p}_{0}(\mathbb{R}^N)$ that, up to a subsequence,
    \[
    \widehat{v}_n \chi_{\mathcal{R}_{n,R}}\rightharpoonup \widehat{v}\chi_{B_R\backslash B_{\frac{1}{R}}}
    \quad \mbox{locally in} \quad L^2(\mathbb{R}^N),\quad \forall R>1.
    \]
    Thanks to this bound, we can split
    \[
    \widehat{v}_n=\widehat{v}+\varphi_n,\quad\mbox{with}\quad \varphi_n:= \widehat{v}_n-\widehat{v},
    \]
    and very same argument as in the case $1<p<\frac{2N}{N+2-\beta}$ allows us to deduce that
    \begin{align*}
    & \liminf_{n\to \infty}\int_{\mathcal{R}_{n,R}}\left[|\nabla U|^{p-2}|\nabla \widehat{v}_n|^2+(p-2)|\omega_n|^{p-2}\left(\frac{|\nabla (U+v_n)|-|\nabla U|}{\varepsilon_n}\right)^2\right]\mathrm{d}x \\
     \geq & \int_{B_R\backslash B_{\frac{1}{R}}}\left[|\nabla U|^{p-2}|\nabla \widehat{v} |^2+(p-2)|\nabla U|^{p-2}\left(\frac{\nabla U\cdot\nabla \widehat{v} }{| \nabla U|}\right)^2\right]\mathrm{d}x.
    \end{align*}
    Recalling \eqref{evbc2g}, since $R>1$ is arbitrary and the integrands above are nonnegative, this proves that \eqref{uinb2p2} holds, a contradiction to Proposition \ref{propevl} since $\widehat{v}$ is orthogonal to $T_{U} \mathcal{M}_\beta$ (being the strong $L^2_{\beta,*}(\mathbb{R}^N)$ limit of $\widehat{v}_n$).
    \end{proof}

\section{Stability of Hardy-Sobolev inequality.}\label{sectpromr}

    The main ingredient of the stability of Hardy-Sobolev inequality is contained in the two lemmas below, in which the behavior near the extremal functions set $\mathcal{M}_\beta$ is studied.

    In order to shorten formulas, for each $u_n\in \mathcal{D}^{1,p}_{0}(\mathbb{R}^N)$ we denote
    \begin{equation*} 
    \begin{split}
    \|u_n\|:
    =\left(\int_{\mathbb{R}^N}|\nabla u_n|^p \mathrm{d}x\right)^{\frac{1}{p}},
    \quad \|u_n\|_*: =\left(\int_{\mathbb{R}^N}|x|^{-\beta}|u_n|^{p^*_{\beta}} \mathrm{d}x\right)^{\frac{1}{p^*_{\beta}}},
    \end{split}
    \end{equation*}
    and
    \[
    d_n:=\inf_{c\in\mathbb{R}, \lambda>0}\|u_n-cU_\lambda\|.\]

    \begin{lemma}\label{lemma:rtnm2b}
    Suppose $2\leq p<N$ and $0<\beta<p$.
    There exists a small constant $\rho_1>0$ such that for any sequence $\{u_n\}\subset \mathcal{D}^{1,p}_{0}(\mathbb{R}^N)\backslash \mathcal{M}_\beta$ satisfying $\inf_n\|u_n\|>0$ and $d_n\to 0$, it holds that
    \begin{equation}\label{rtnmb}
    \liminf_{n\to\infty}
    \frac{\|u_n\|^p
    -\mathcal{S}_\beta\|u_n\|_*^p}
    {d_n^p}
    \geq \rho_1.
    \end{equation}
    \end{lemma}

    \begin{proof}
    We know that for each $u_n\in \mathcal{D}^{1,p}_{0}(\mathbb{R}^N)$, there exist $c_n\in\mathbb{R}$ and $\lambda_n>0$ such that $d_n=\|u_n-c_nU_{\lambda_n}\|$. In fact, since $2\leq p<N$, for each fixed $n$, from Lemma \ref{lemui1p}, we obtain that for any $0<\kappa<1$, there exists a constant $\mathcal{C}_1=\mathcal{C}_1(p,\kappa)>0$ such that
    \begin{align}\label{ikeda}
    \|u_n-cU_\lambda\|^p
    = & \int_{\mathbb{R}^N}|\nabla u_n-c\nabla U_\lambda|^p \mathrm{d}x\nonumber\\
    \geq & \int_{\mathbb{R}^N}|\nabla u_n|^p \mathrm{d}x
    -pc\int_{\mathbb{R}^N}|\nabla u_n|^{p-2}  \nabla u_n\cdot \nabla U_\lambda \mathrm{d}x  +\mathcal{C}_1|c|^{p}\int_{\mathbb{R}^N}|\nabla U_\lambda|^{p} \mathrm{d}x
    \nonumber\\ &
    +\frac{(1-\kappa)p}{2}c^2 \int_{\mathbb{R}^N}|\nabla u_n|^{p-2}  |\nabla U_\lambda|^2\mathrm{d}x \nonumber\\
    & +\frac{(1-\kappa)p(p-2)}{2}\int_{\mathbb{R}^N}|\omega(\nabla u_n, \nabla u_n-c \nabla U_{\lambda})|^{p-2}(|\nabla u_n|-|\nabla u_n-c \nabla U_{\lambda}|)^2  \mathrm{d}x \nonumber\\
    \geq & \|u_n\|^p+ \mathcal{C}_1|c|^{p}\|U\|^p-pc\int_{\mathbb{R}^N}|\nabla u_n|^{p-2}  \nabla u_n\cdot \nabla U_\lambda \mathrm{d}x\nonumber\\
    \geq & \|u_n\|^p+ \mathcal{C}_1|c|^{p}\|U\|^p-p|c|\|U\|\|u_n\|^{p-1},
    \end{align}
    where $\omega:\mathbb{R}^{2N}\to \mathbb{R}^N$ corresponds to $\nabla u_n$ and $\nabla u_n-c \nabla U_{\lambda}$ as in Lemma \ref{lemui1p}  for the case $p\geq 2$.
    Thus the minimizing sequence of $d_n$, say $\{c_{n,m},\lambda_{n,m}\}$, must satisfying $|c_{n,m}|\leq C$ which means $\{c_{n,m}\}$ is bounded.
    On the other hand,
    \begin{align*}
    \left|\int_{|\lambda x|\leq \rho}|\nabla u_n|^{p-2}  \nabla u_n\cdot \nabla U_\lambda \mathrm{d}x\right|
    \leq & \int_{|y|\leq \rho}|\nabla  (u_n)_{\frac{1}{\lambda}}(y)|^{p-1}|\nabla U(y)| \mathrm{d}y \\
    \leq & \|u_n\|^{p-1}\left(\int_{|y|\leq \rho}|\nabla U|^p \mathrm{d}y\right)^{\frac{1}{p}} \\
    = & o_\rho(1)
    \end{align*}
    as $\rho\to 0$ which is uniform for $\lambda>0$, where $(u_n)_{\frac{1}{\lambda}}(y)=\lambda^{-\frac{N-p}{p}}u_n(\lambda^{-1}y)$, and
    \begin{align*}
    \left|\int_{|\lambda x|\geq \rho}|\nabla u_n|^{p-2}  \nabla u_n\cdot \nabla U_\lambda \mathrm{d}x\right|
    \leq & \|U\|\left(\int_{|x|\geq \frac{\rho}{\lambda}}|\nabla u_n|^p \mathrm{d}y\right)^{\frac{1}{p}}
    =  o_\lambda(1)
    \end{align*}
    as $\lambda\to 0$ for any fixed $\rho>0$. By taking $\lambda\to 0$ and then $\rho\to 0$, we obtain
    \[\left|\int_{\mathbb{R}^N}|\nabla u_n|^{p-2}  \nabla u_n\cdot \nabla U_\lambda \mathrm{d}x\right| \to 0\quad \mbox{as}\quad \lambda\to 0.\]
    Moreover, by the explicit from of $U_\lambda$ we have
    \begin{equation*}
    \begin{split}
    \left|\int_{|\lambda x|\leq R}|\nabla u_n|^{p-2}  \nabla u_n\cdot \nabla U_\lambda \mathrm{d}x\right|
    \leq & \|U\|\left(\int_{| x|\leq \frac{R}{\lambda}}|\nabla u_n|^p \mathrm{d}x\right)^{\frac{1}{p}}
    = o_\lambda(1)
    \end{split}
    \end{equation*}
    as $\lambda\to +\infty$ for any fixed $R>0$, and
    \begin{equation*}
    \begin{split}
    \left|\int_{|\lambda x|\geq R}|\nabla u_n|^{p-2}  \nabla u_n\cdot \nabla U_\lambda \mathrm{d}x\right|
    \leq & \int_{|y|\geq R}|\nabla (u_n)_{\frac{1}{\lambda}}(y)|^{p-1}|\nabla U(y)| \mathrm{d}y \\
    \leq & \|u_n\|^{p-1}\left(\int_{|y|\geq R}|\nabla U|^p \mathrm{d}y\right)^{\frac{1}{p}}
    =  o_R(1)
    \end{split}
    \end{equation*}
    as $R\to +\infty$ which is uniform for $\lambda>0$. Thus, by taking first $\lambda\to +\infty$ and then $R\to +\infty$, we also obtain
    \[
    \left|\int_{\mathbb{R}^N}|\nabla u_n|^{p-2}  \nabla u_n\cdot \nabla U_\lambda \mathrm{d}x\right| \to 0\quad \mbox{as}\quad \lambda\to +\infty.
    \]
    It follows from (\ref{ikeda}) and $d_n\to 0$, $\inf_n\|u_n\|>0$ that the minimizing sequence $\{c_{n,m},\lambda_{n,m}\}$ must satisfying $1/C\leq \lambda_{n,m}\leq C$ for some $C>1$ independent of $m$,  which means $\{\lambda_{n,m}\}$ is bounded. Thus for each $u_n\in  \mathcal{D}^{1,p}_{0}(\mathbb{R}^N)\backslash \mathcal{M}_\beta$, $d_n$ can be attained by some $c_n\in\mathbb{R}$ and $\lambda_n>0$.

    Since $\mathcal{M}_\beta$ is two-dimensional manifold embedded in $\mathcal{D}^{1,p}_{0}(\mathbb{R}^N)$, that is
    \[
    (c,\lambda)\in\mathbb{R}\times\mathbb{R}_+\to cU_\lambda\in \mathcal{D}^{1,p}_{0}(\mathbb{R}^N),
    \]
    then from \eqref{czkj}, under suitable transform, the tangential space is
    \[
    T_{c_n U_{\lambda_n}}\mathcal{M}_\beta={\rm Span}\left\{U_{\lambda_n}, \ \frac{\partial U_\lambda}{\partial\lambda}\Big|_{\lambda=\lambda_n}\right\}.
    \]
    Then we have,
    \[
    \int_{\mathbb{R}^N}|x|^{-\beta}U_{\lambda_n}^{p^*_{\alpha,\beta}-2}(u_n-c_n U_{\lambda_n})\xi \mathrm{d}x=0,\quad \forall \xi \in T_{c_n U_{\lambda_n}}\mathcal{M}_\beta,
    \]
    particularly, taking $\xi=U_{\lambda_n}$ we obtain
    \begin{equation*}
    \int_{\mathbb{R}^N}|x|^{-\beta}U_{\lambda_n}^{p^*_{\beta}-1}(u_n-c_n U_{\lambda_n})\mathrm{d}x
    =\int_{\mathbb{R}^N}|\nabla U_{\lambda_n}|^{p-2}\nabla U_{\lambda_n}\cdot \nabla (u_n-c_n U_{\lambda_n}) \mathrm{d}x
    =0.
    \end{equation*}
    Let
    \begin{equation}\label{defunwn}
    u_n=c_n U_{\lambda_n}+d_n w_n,
    \end{equation}
     then $w_n$ is perpendicular to $T_{c_n U_{\lambda_n}}\mathcal{M}_\beta$, we have
    \begin{equation*}
    \|w_n\|=1\quad \mbox{and}\quad \int_{\mathbb{R}^N} |\nabla U_{\lambda_n}|^{p-2}\nabla U_{\lambda_n}\cdot \nabla w_n \mathrm{d}x=0.
    \end{equation*}
    From Lemma \ref{lemui1p}, for any $\kappa>0$, there exists a constant $\mathcal{C}_1=\mathcal{C}_1(p,\kappa)>0$ such that
    \begin{small}
    \begin{align}\label{epknug}
    \|u_n\|^p
    = & \int_{\mathbb{R}^N}|c_n \nabla U_{\lambda_n}+d_n \nabla w_n|^p \mathrm{d}x\nonumber\\
    \geq & |c_n|^{p}\int_{\mathbb{R}^N}|\nabla U_{\lambda_n}|^p \mathrm{d}x
    +p|c_n|^{p-2}c_nd_n \int_{\mathbb{R}^N}|\nabla U_{\lambda_n}|^{p-2}  \nabla U_{\lambda_n}\cdot \nabla w_n \mathrm{d}x  \nonumber\\
    & +\mathcal{C}_1d_n^{p}\int_{\mathbb{R}^N}|\nabla w_n|^{p} \mathrm{d}x
    +\frac{(1-\kappa)p}{2} |c_n|^{p-2}d_n^2\int_{\mathbb{R}^N}|\nabla U_{\lambda_n}|^{p-2}  |\nabla w_n|^2\mathrm{d}x \nonumber\\
    & +\frac{(1-\kappa)p(p-2)}{2}\int_{\mathbb{R}^N}|\omega(c_n \nabla U_{\lambda_n},\nabla u_n)|^{p-2}(|c_n \nabla U_{\lambda_n}|-|\nabla u_n|)^2  \mathrm{d}x \nonumber\\
    = & |c_n|^{p}\|U\|^p+ \mathcal{C}_1d_n^{p}
     +\frac{(1-\kappa)p}{2} |c_n|^{p-2}d_n^2\int_{\mathbb{R}^N}|\nabla U_{\lambda_n}|^{p-2}  |\nabla w_n|^2\mathrm{d}x \nonumber\\
    & +\frac{(1-\kappa)p(p-2)}{2}\int_{\mathbb{R}^N}|\omega(c_n \nabla U_{\lambda_n},\nabla u_n)|^{p-2}(|c_n \nabla U_{\lambda_n}|-|\nabla u_n|)^2  \mathrm{d}x,
    \end{align}
    \end{small}
    where $\omega:\mathbb{R}^{2N}\to \mathbb{R}^N$ corresponds to $c_n \nabla U_{\lambda_n}$ and $u_n$ as in Lemma \ref{lemui1p}, since
    \begin{equation*}
    \int_{\mathbb{R}^N}|x|^{-\beta}U_{\lambda_n}^{p^*_{\beta}-1}w_n \mathrm{d}x=\int_{\mathbb{R}^N}|\nabla U_{\lambda_n}|^{p-2}\nabla U_{\lambda_n}\cdot \nabla w_n \mathrm{d}x=0,
    \end{equation*}
    and
    \begin{equation*}
    \int_{\mathbb{R}^N}|x|^{-\beta}U_{\lambda_n}^{p^*_{\beta}} \mathrm{d}x=\int_{\mathbb{R}^N}|\nabla U_{\lambda_n}|^{p}\mathrm{d}x=\|U\|^p.
    \end{equation*}
    Then from Lemma \ref{lemui1p*l}, for any $\kappa>0$, there exists a constant $\mathcal{C}_2=\mathcal{C}_2(p^*_{\alpha,\beta},\kappa)>0$ such that
    \begin{align*}
    \|u_n\|^{p^*_{\beta}}_*= & \int_{\mathbb{R}^N}|x|^{-\beta}|c_n U_{\lambda_n}+d_nw_n|^{p^*_{\beta}}  \mathrm{d}x\\
    \leq & |c_n|^{p^*_{\beta}}
    \int_{\mathbb{R}^N}|x|^{-\beta}U_{\lambda_n}^{p^*_{\beta}} \mathrm{d}x
    +|c_n|^{p^*_{\beta}-2}c_n p^*_{\beta} d_n \int_{\mathbb{R}^N}|x|^{-\beta}U_{\lambda_n}^{p^*_{\beta}-1}w_n \mathrm{d}x  \\
    & +\left(\frac{p^*_{\beta}(p^*_{\beta}-1)}{2}+\kappa\right)
    |c_n|^{p^*_{\beta}-2} d_n^2
    \int_{\mathbb{R}^N}|x|^{-\beta}U_{\lambda_n}^{p^*_{\beta}-2}w_n^2 \mathrm{d}x \\
    & +\mathcal{C}_2d_n^{p^*_{\beta}}
    \int_{\mathbb{R}^N}|x|^{-\beta}|w_n|^{p^*_{\beta}} \mathrm{d}x \\
    = & |c_n|^{p^*_{\beta}}\|U\|^p
    +\left(\frac{p^*_{\beta}(p^*_{\beta}-1)}{2}+\kappa\right)
    |c_n|^{p^*_{\beta}-2} d_n^2
    \int_{\mathbb{R}^N}|x|^{\beta}
    U_{\lambda_n}^{p^*_{\alpha,\beta}-2}w_n^2 \mathrm{d}x
    + o(d_n^p),
    \end{align*}
    since $p<p^*_{\beta}$.
    Thus, by the concavity of $t\mapsto t^{\frac{p}{p^*_{\beta}}}$, we have
    \begin{align}\label{epkeyiyxbb}
    \|u_n\|^p_*
    \leq &  |c_n|^p\|U\|^{\frac{p^2}{p^*_{\beta}}}
    + \frac{p|c_n|^{p^*_{\beta}-2} d_n^2}{p^*_{\beta}}
    \left(\frac{p^*_{\beta}(p^*_{\beta}-1)}{2}+\kappa\right)
    \|U\|^{\frac{p^2}{p^*_{\beta}}-p}
    \int_{\mathbb{R}^N}|x|^{-\beta}U_{\lambda_n}^{p^*_{\beta}-2}w_n^2
    \mathrm{d}x
    \nonumber\\ & +o(d_n^p).
    \end{align}
    Therefore, as $d_n\to 0$, combining \eqref{epknug} with \eqref{epkeyiyxbb} , it follows from Lemma \ref{lemsgap} that, by choosing $\kappa>0$ small enough,
    \begin{small}
    \begin{align*}
    \|u_n\|^p
    -\mathcal{S}_\beta\|u_n\|_*^p
    \geq & |c_n|^{p}\|U\|^p+ \mathcal{C}_1d_n^{p}
    +\frac{(1-\kappa)p}{2} d_n^2\int_{\mathbb{R}^N}|\nabla c_n U_{\lambda_n}|^{p-2}  |\nabla w_n|^2\mathrm{d}x \\
    & +\frac{(1-\kappa)p(p-2)}{2}\int_{\mathbb{R}^N}|\omega(c_n \nabla U_{\lambda_n},\nabla u_n)|^{p-2}(|c_n \nabla U_{\lambda_n}|-|\nabla u_n|)^2 \mathrm{d}x \\
    & -\mathcal{S}_\beta\Bigg\{|c_n|^p\|U\|^{\frac{p^2}{p^*_{\beta}}} +o(d_n^p) \\
    & + \frac{p|c_n|^{p^*_{\beta}-2} d_n^2}{p^*_{\beta}}
    \left(\frac{p^*_{\beta}(p^*_{\beta}-1)}{2}+\kappa\right)
    \|U\|^{\frac{p^2}{p^*_{\beta}}-p}
    \int_{\mathbb{R}^N}|x|^{-\beta}U_{\lambda_n}^{p^*_{\beta}-2}w_n^2 \mathrm{d}x\Bigg\} \\
    \geq  & \mathcal{C}_1d_n^p -o(d_n^p)  \\
    & + \frac{(1-\kappa)p  |c_n|^{p^*_{\beta}-2} d_n^2}{2}\left[(p^*_{\beta}-1)+\tau\right]
    \int_{\mathbb{R}^N}|x|^{-\beta}U_{\lambda_n}^{p^*_{\beta}-2}w_n^2 \mathrm{d}x\\
    & - \frac{p|c_n|^{p^*_{\beta}-2} d_n^2}{p^*_{\beta}}
    \left(\frac{p^*_{\beta}(p^*_{\beta}-1)}{2}+\kappa\right)
    \mathcal{S}_\beta\|U\|^{\frac{p^2}{p^*_{\beta}}-p}
    \int_{\mathbb{R}^N}|x|^{-\beta}U_{\lambda_n}^{p^*_{\beta}-2}w_n^2 \mathrm{d}x \\
    \geq  & \mathcal{C}_1d_n^p -o(d_n^p),
    \end{align*}
    \end{small}
    since $\|U\|^p=\mathcal{S}_\beta^{\frac{p^*_{\beta}}{p^*_{\beta}-p}}$, then (\ref{rtnmb}) follows immediately.
    \end{proof}

    \begin{lemma}\label{lemma:rtnm2b2}
    Suppose $1<p<2\leq N$  and $0<\beta<p$.
    There exists a small constant $\rho_2>0$ such that for any sequence $\{u_n\}\subset \mathcal{D}^{1,p}_{0}(\mathbb{R}^N)\backslash \mathcal{M}_\beta$ satisfying $\inf_n\|u_n\|>0$ and $d_n\to 0$, it holds that
    \begin{equation}\label{rtnmb2}
    \liminf_{n\to\infty}
    \frac{\|u_n\|^p
    -\mathcal{S}_\beta\|u_n\|_*^p}
    {d_n^2}
    \geq \rho_2.
    \end{equation}
    \end{lemma}

    \begin{proof}
    We know that for each $u_n\in \mathcal{D}^{1,p}_{0}(\mathbb{R}^N)$, there exist $c_n\in\mathbb{R}$ and $\lambda_n>0$ such that $d_n=\|u_n-c_nU_{\lambda_n}\|$.
    In fact, since $1< p<2$, for each fixed $n$, from Lemma \ref{lemui1p}, we obtain that for any $0<\kappa<1$, there exists a constant $\mathcal{C}_1=\mathcal{C}_1(p,\kappa)>0$ such that
    \begin{small}
    \begin{align}\label{ikeda2}
    \|u_n-cU_\lambda\|^p
    = & \int_{\mathbb{R}^N}|\nabla u_n-c\nabla U_\lambda|^p \mathrm{d}x\nonumber\\
    \geq & \int_{\mathbb{R}^N}|\nabla u_n|^p dx
    -pc\int_{\mathbb{R}^N}|\nabla u_n|^{p-2}  \nabla u_n\cdot \nabla U_\lambda \mathrm{d}x  \nonumber\\ &
    +\mathcal{C}_1|c|^{2}\int_{\mathbb{R}^N}\min\{|c|^{p-2}|\nabla U_\lambda|^p,|\nabla u_n|^{p-2}|\nabla U_\lambda|\} \mathrm{d}x
    \nonumber\\ &
    +\frac{(1-\kappa)p}{2}c^2 \int_{\mathbb{R}^N}|\nabla u_n|^{p-2}  |\nabla U_\lambda|^2\mathrm{d}x \nonumber\\
    & +\frac{(1-\kappa)p(p-2)}{2}\int_{\mathbb{R}^N}
    |\widetilde{\omega}(\nabla u_n, \nabla u_n-c \nabla U_{\lambda})|^{p-2}(|\nabla u_n|-|\nabla u_n-c \nabla U_{\lambda}|)^2  \mathrm{d}x \nonumber\\
    \geq & \int_{\mathbb{R}^N}|\nabla u_n|^p dx
    -pc\int_{\mathbb{R}^N}|\nabla u_n|^{p-2}  \nabla u_n\cdot \nabla U_\lambda \mathrm{d}x  \nonumber\\ &
    +\mathcal{C}_1|c|^{2}\int_{\mathbb{R}^N}\min\{|c|^{p-2}|\nabla U_\lambda|^p,|\nabla u_n|^{p-2}|\nabla U_\lambda|\} \mathrm{d}x
    \nonumber\\
    \geq & \|u_n\|^p-p|c|\|U\|\|u_n\|^{p-1}\nonumber\\ &
    +\mathcal{C}_1|c|^{2}\int_{\mathbb{R}^N}\min\{|c|^{p-2}|\nabla U_\lambda|^p,|\nabla u_n|^{p-2}|\nabla U_\lambda|\} \mathrm{d}x,
    \end{align}
    \end{small}
    where $\widetilde{\omega}:\mathbb{R}^{2N}\to \mathbb{R}^N$ corresponds to $\nabla u_n$ and $\nabla u_n-c \nabla U_{\lambda}$ the same as $\omega$ in Lemma \ref{lemui1p} for the case $1<p<2$,  since from \eqref{evbcb2i} it holds that
    \begin{align*}
    0\leq & c^2 \int_{\mathbb{R}^N}|\nabla u_n|^{p-2}  |\nabla U_\lambda|^2\mathrm{d}x
    \nonumber\\ & + (p-2)\int_{\mathbb{R}^N}|\widetilde{\omega}(\nabla u_n, \nabla u_n-c \nabla U_{\lambda})|^{p-2}(|\nabla u_n|-|\nabla u_n-c \nabla U_{\lambda}|)^2  \mathrm{d}x.
    \end{align*}
    Therefore the minimizing sequence of $d_n$, say $\{c_{n,m},\lambda_{n,m}\}$, must satisfying $|c_{n,m}|\leq C$ for some $C>0$ independent of $m$,  which means $\{c_{n,m}\}$ is bounded.
    On the other hand, taking the same steps as those in Lemma \ref{lemma:rtnm2b}, we deduce that
    \[
    \left|\int_{\mathbb{R}^N}|\nabla u_n|^{p-2}  \nabla u_n\cdot \nabla U_\lambda \mathrm{d}x\right|
    \to 0\quad \mbox{as}\quad \lambda\to 0,
    \]
    and
    \[
    \left|\int_{\mathbb{R}^N}|\nabla u_n|^{p-2}  \nabla u_n\cdot \nabla U_\lambda \mathrm{d}x\right|
    \to 0\quad \mbox{as}\quad \lambda\to +\infty.
    \]
    It follows from (\ref{ikeda}) and $d_n\to 0$, $\inf_n\|u_n\|>0$ that the minimizing sequence $\{c_{n,m},\lambda_{n,m}\}$ must satisfying $1/C\leq |\lambda_{n,m}|\leq C$ for some $C>1$ independent of $m$,  which means $\{\lambda_{n,m}\}$ is bounded. Thus for each $u_n\in  \mathcal{D}^{1,p}_{0}(\mathbb{R}^N)\backslash \mathcal{M}_\beta$, $d_n$ can also be attained by some $c_n\in\mathbb{R}$ and $\lambda_n>0$.

    As stated in Lemma \ref{lemma:rtnm2b}, we have
    \[
    T_{c_n U_{\lambda_n}}\mathcal{M}_\beta={\rm Span}\left\{U_{\lambda_n}, \ \frac{\partial U_\lambda}{\partial\lambda}\Big|_{\lambda=\lambda_n}\right\},
    \]
    and
    \[
    \int_{\mathbb{R}^N}|x|^{-\beta}U_{\lambda_n}^{p^*_{\alpha,\beta}-2}(u_n-c_n U_{\lambda_n})\xi \mathrm{d}x=0,\quad \forall \xi \in T_{c_n U_{\lambda_n}}\mathcal{M}_\beta,
    \]
    particularly, taking $\xi=U_{\lambda_n}$ we obtain
    \begin{equation*}
    \int_{\mathbb{R}^N}|x|^{-\beta}U_{\lambda_n}^{p^*_{\beta}-1}(u_n-c_n U_{\lambda_n})\mathrm{d}x
    =\int_{\mathbb{R}^N}|\nabla U_{\lambda_n}|^{p-2}\nabla U_{\lambda_n}\cdot \nabla (u_n-c_n U_{\lambda_n}) \mathrm{d}x
    =0.
    \end{equation*}
    Let
    \begin{equation}\label{defunwn2}
    u_n=c_n U_{\lambda_n}+d_n w_n,
    \end{equation}
     then $w_n$ is perpendicular to $T_{c_n U_{\lambda_n}}\mathcal{M}_\beta$, we have
    \begin{equation*}
    \|w_n\|=1\quad \mbox{and}\quad \int_{\mathbb{R}^N} |\nabla U_{\lambda_n}|^{p-2}\nabla U_{\lambda_n}\cdot \nabla w_n \mathrm{d}x=0.
    \end{equation*}
    Since $1<p<2$, from Lemma \ref{lemui1p} we obtain that for any $\kappa>0$, there exists a constant $\mathcal{C}_2=\mathcal{C}_2(p,\kappa)>0$ such that
    \begin{align}\label{epknug2}
    \|u_n\|^p
    = & \int_{\mathbb{R}^N}|c_n \nabla U_{\lambda_n}+d_n \nabla w_n|^p \mathrm{d}x \nonumber\\
    \geq & |c_n|^{p}\int_{\mathbb{R}^N}|\nabla U_{\lambda_n}|^p \mathrm{d}x
    +p|c_n|^{p-2}c_nd_n \int_{\mathbb{R}^N}|\nabla U_{\lambda_n}|^{p-2}  \nabla U_{\lambda_n}\cdot \nabla w_n \mathrm{d}x  \nonumber\\
    & +\frac{(1-\kappa)p}{2} |c_n|^{p-2}d_n^2\int_{\mathbb{R}^N}|\nabla U_{\lambda_n}|^{p-2}  |\nabla w_n|^2\mathrm{d}x \nonumber\\
    & +\frac{(1-\kappa)p(p-2)}{2}\int_{\mathbb{R}^N}|\omega(c_n \nabla U_{\lambda_n},\nabla u_n)|^{p-2}(|c_n \nabla U_{\lambda_n}|-|\nabla u_n|)^2  \mathrm{d}x \nonumber\\
    & +\mathcal{C}_2d_n^{2} \int_{\mathbb{R}^N}\min\{d_n^{p-2}|\nabla w_n|^{p}, |c_n \nabla U_{\lambda_n}|^{p-2}|\nabla w_n|^{2}\} \mathrm{d}x  \nonumber\\
    = & |c_n|^{p}\|U\|^p
    +\frac{(1-\kappa)p}{2} |c_n|^{p-2}d_n^2\int_{\mathbb{R}^N}|\nabla U_{\lambda_n}|^{p-2}  |\nabla w_n|^2\mathrm{d}x \nonumber\\
    & +\frac{(1-\kappa)p(p-2)}{2}\int_{\mathbb{R}^N}|\omega(c_n \nabla U_{\lambda_n},\nabla u_n)|^{p-2}(|c_n \nabla U_{\lambda_n}|-|\nabla u_n|)^2  \mathrm{d}x\nonumber\\
    & +\mathcal{C}_2d_n^{2} \int_{\mathbb{R}^N}\min\{d_n^{p-2}|\nabla w_n|^{p}, |c_n \nabla U_{\lambda_n}|^{p-2}|\nabla w_n|^{2}\} \mathrm{d}x,
    \end{align}
    where $\omega:\mathbb{R}^{2N}\to \mathbb{R}^N$ corresponds to $c_n \nabla U_{\lambda_n}$ and $\nabla u_n$ as in Lemma \ref{lemui1p}. Then we consider the following two cases:

    $\bullet$ {\em The case $1<p\leq\frac{2N}{N+2-\beta}$} which implies $p^*_{\beta}\leq2$.

    From Lemma \ref{lemui1p*l}, for any $\kappa>0$ and $C_1>0$, there exists a constant $\mathcal{C}_1=\mathcal{C}_1(p^*_{\beta},\kappa,C_1)>0$ such that
    \begin{align*}
    \|u_n\|^{p^*_\beta}_*
    = & \int_{\mathbb{R}^N}|x|^{-\beta}|c_n U_{\lambda_n}+d_nw_n|^{p^*_{\beta}}  \mathrm{d}x\\
    \leq & |c_n|^{p^*_{\beta}}
    \int_{\mathbb{R}^N}|x|^{-\beta}U_{\lambda_n}^{p^*_{\beta}} \mathrm{d}x
    +p^*_{\alpha} |c_n|^{p^*_{\beta}-2}c_n d_n \int_{\mathbb{R}^N}|x|^{-\beta}U_{\lambda_n}^{p^*_{\beta}-1}w_n \mathrm{d}x  \\
    & +\left(\frac{p^*_{\beta}(p^*_{\beta}-1)}{2}+\kappa\right)d_n^2
    \int_{\mathbb{R}^N}|x|^{-\beta}\frac{(|c_n U_{\lambda_n}|+\mathcal{C}_1|d_nw_n|)^{p^*_{\beta}}}{|c_n U_{\lambda_n}|^2+|d_nw_n|^2}w_n^2 \mathrm{d}x \\
    = & |c_n|^{p^*_{\beta}}\|U\|^p
    +\left(\frac{p^*_{\beta}(p^*_{\beta}-1)}{2}+\kappa\right)d_n^2
    \int_{\mathbb{R}^N}|x|^{-\beta}\frac{(|c_n U_{\lambda_n}|+\mathcal{C}_1|d_nw_n|)^{p^*_{\beta}}}{|c_n U_{\lambda_n}|^2+|d_nw_n|^2}w_n^2 \mathrm{d}x,
    \end{align*}
    since
    \begin{equation*}
    \int_{\mathbb{R}^N}|x|^{-\beta}U_{\lambda_n}^{p^*_{\beta}-1}w_n \mathrm{d}x=\int_{\mathbb{R}^N}|\nabla U_{\lambda_n}|^{p-2}\nabla U_{\lambda_n}\cdot \nabla w_n \mathrm{d}x
    =0,
    \end{equation*}
    and
    \begin{equation*}
    \int_{\mathbb{R}^N}|x|^{-\beta}U_{\lambda_n}^{p^*_{\beta}} \mathrm{d}x
    =\int_{\mathbb{R}^N}|\nabla U_{\lambda_n}|^{p}\mathrm{d}x
    =\|U\|^p.
    \end{equation*}
    Thus, by the concavity of $t\mapsto t^{\frac{p}{p^*_{\beta}}}$, we have
    \begin{align}\label{epkeyiyxbb2}
    \|u_n\|^p_*
    \leq  &  |c_n|^p\|U\|^{\frac{p^2}{p^*_{\beta}}}
    \nonumber\\ & +\frac{p}{p^*_{\beta}}
    \left(\frac{p^*_{\beta}(p^*_{\beta}-1)}{2}+\kappa\right)d_n^2
    \|U\|^{\frac{p^2}{p^*_{\beta}}-p}
    \int_{\mathbb{R}^N}|x|^{-\beta}\frac{(|c_n U_{\lambda_n}|+\mathcal{C}_1|d_nw_n|)^{p^*_{\beta}}}{|c_n U_{\lambda_n}|^2+|d_nw_n|^2}w_n^2 \mathrm{d}x.
    \end{align}
    Therefore, as $d_n\to 0$, combining \eqref{epkeyiyxbb2} with \eqref{epknug2}, it follows from Lemma \ref{lemsgap} that, by choosing $\kappa>0$ small enough,
    \begin{small}
    \begin{align*}
    \|u_n\|^p
    -\mathcal{S}_\beta\|u_n\|_*^p
    \geq & |c_n|^{p}\|U\|^p
    +\frac{(1-\kappa)p}{2}d_n^2\int_{\mathbb{R}^N}|\nabla c_n U_{\lambda_n}|^{p-2}  |\nabla w_n|^2\mathrm{d}x \\
    & +\frac{(1-\kappa)p(p-2)}{2}\int_{\mathbb{R}^N}|\omega(c_n \nabla U_{\lambda_n},\nabla u_n)|^{p-2}(|c_n \nabla U_{\lambda_n}|-|\nabla u_n|)^2 \mathrm{d}x\\
    & +\mathcal{C}_2d_n^{2} \int_{\mathbb{R}^N}\min\{d_n^{p-2}|\nabla w_n|^{p}, |c_n \nabla U_{\lambda_n}|^{p-2}|\nabla w_n|^{2}\}
    \mathrm{d}x \\
    & -\mathcal{S}_\beta\Bigg\{|c_n|^p\|U\|^{\frac{p^2}{p^*_{\alpha,\beta}}}
    \\ & \quad \quad +\left(\frac{p(p^*_{\beta}-1)}{2}
    +\frac{p\kappa}{p^*_{\beta}}\right)d_n^2
    \|U\|^{\frac{p^2}{p^*_{\beta}}-p}
    \int_{\mathbb{R}^N}|x|^{-\beta}\frac{(|c_n U_{\lambda_n}|+\mathcal{C}_1|d_nw_n|)^{p^*_{\beta}}}{|c_n U_{\lambda_n}|^2+|d_nw_n|^2}w_n^2 \mathrm{d}x\Bigg\} \\
    \geq  & \frac{(1-\kappa)p}{2} d_n^2\int_{\mathbb{R}^N} |\nabla c_nU_{\lambda_n}|^{p-2}  |\nabla w_n|^2\mathrm{d}x
    \\ & +\frac{(1-\kappa)p(p-2)}{2}\int_{\mathbb{R}^N}|\omega(c_n \nabla U_{\lambda_n},\nabla u_n)|^{p-2}(|c_n \nabla U_{\lambda_n}|-|\nabla u_n|)^2  \mathrm{d}x\\
    & +\mathcal{C}_2d_n^{2} \int_{\mathbb{R}^N}\min\{d_n^{p-2}|\nabla w_n|^{p}, |c_n \nabla U_{\lambda_n}|^{p-2}|\nabla w_n|^{2}\} \mathrm{d}x
    \\ & -\left(\frac{p(p^*_{\beta}-1)}{2}
    +\frac{p\kappa}{p^*_{\beta}}\right)d_n^2
    \int_{\mathbb{R}^N}|x|^{-\beta}\frac{(|c_n U_{\lambda_n}|+\mathcal{C}_1|d_nw_n|)^{p^*_{\beta}}}{|c_n U_{\lambda_n}|^2+|d_nw_n|^2}w_n^2 \mathrm{d}x,
    \end{align*}
    \end{small}
    since $\|U\|^p=\mathcal{S}_\beta^{\frac{p^*_{\beta}}{p^*_{\beta}-p}}$. Lemma \ref{lemsgap2} allows us to reabsorb the last term above: more precisely, we have
    \begin{small}
    \begin{align*}
    & \|u_n\|^p
    -\mathcal{S}_\beta\|u_n\|_*^p \\
    \geq & pd_n^2\left(\frac{(1-\kappa)}{2} - \frac{(p^*_{\beta}-1)+\frac{2}{p^*_{\beta}}\kappa}
    {2(p^*_{\beta}-1)
    +2\tau }\right) \\
    & \quad \times
    \int_{\mathbb{R}^N}\left[|\nabla c_nU_{\lambda_n}|^{p-2}  |\nabla w_n|^2
    +(p-2)|\omega(c_n \nabla U_{\lambda_n},\nabla u_n)|^{p-2}\left(\frac{|c_n \nabla U_{\lambda_n}|-|\nabla u_n|}{d_n}\right)^2\right]\mathrm{d}x \\
    & + d_n^{2}\left(\mathcal{C}_2 - \gamma_0\frac{p\left[(p^*_{\beta}-1)
    +\frac{2}{p^*_{\beta}}\kappa\right]}
    {2(p^*_{\beta}-1)+2\tau}\right)
    \int_{\mathbb{R}^N}\min\{d_n^{p-2}|\nabla w_n|^{p}, |c_n \nabla U_{\lambda_n}|^{p-2}|\nabla w_n|^{2}\} \mathrm{d}x.
    \end{align*}
    \end{small}
    Now, let us recall the definition of $\omega$, as stated in Lemma \ref{lemui1p}, we have
    \[
    |\nabla c_nU_{\lambda_n}|^{p-2}  |\nabla w_n|^2
    +(p-2)|\omega(c_n \nabla U_{\lambda_n},\nabla u_n)|^{p-2}\left(\frac{|c_n \nabla U_{\lambda_n}|-|\nabla u_n|}{d_n}\right)^2\geq 0,
    \]
    then choosing $\kappa>0$ small enough such that
    \[
    \frac{(1-\kappa)}{2} - \frac{(p^*_{\alpha,\beta}-1)+\frac{2}{p^*_{\alpha,\beta}}\kappa}{2(p^*_{\alpha,\beta}-1)
    +2\tau }\geq 0,
    \]
    and then $\gamma_0>0$ small enough such that
    \[
    \frac{\mathcal{C}_2}{2}\geq \gamma_0\frac{p\left[(p^*_{\beta}-1)
    +\frac{2}{p^*_{\beta}}\kappa\right]}
    {2(p^*_{\beta}-1)+2\tau},
    \]
    we eventually arrive at
    \begin{align}\label{emn21}
    & \|u_n\|^p
    -\mathcal{S}_\beta\|u_n\|_*^p
    \geq \frac{\mathcal{C}_2}{2}d_n^{2}\int_{\mathbb{R}^N}
    \min\{d_n^{p-2}|\nabla w_n|^{p}, |c_n \nabla U_{\lambda_n}|^{p-2}|\nabla w_n|^{2}\} \mathrm{d}x.
    \end{align}
    Observe that, since $1<p<2$, it follows by H\"{o}lder inequality that
    \begin{align*}
    \left(\int_{\{d_n|\nabla w_n|\geq|c_n \nabla U_{\lambda_n}|\}}|\nabla w_n|^{p} \mathrm{d}x\right)^{\frac{2}{p}}
    \leq & \left(\int_{\{d_n|\nabla w_n|\geq|c_n \nabla U_{\lambda_n}|\}}|\nabla U_{\lambda_n}|^{p} \mathrm{d}x\right)^{\frac{2}{p}-1}  \\
    & \quad \times
    \int_{\{d_n|\nabla w_n|\geq|c_n \nabla U_{\lambda_n}|\}}|\nabla U_{\lambda_n}|^{p-2}|\nabla w_n|^{2} \mathrm{d}x  \\
    \leq & \mathcal{S}_\beta^{\frac{p^*_{\beta}(\frac{2}{p}-1)}
    {p^*_{\beta}-p}} \int_{\{d_n|\nabla w_n|\geq|c_n \nabla U_{\lambda_n}|\}}|\nabla U_{\lambda_n}|^{p-2}|\nabla w_n|^{2}
    \mathrm{d}x,
    \end{align*}
    then we obtain
    \begin{align}\label{emn21b}
    & \int_{\mathbb{R}^N}\min\{d_n^{p-2}|\nabla w_n|^{p}, |c_n \nabla U_{\lambda_n}|^{p-2}|\nabla w_n|^{2}\} \mathrm{d}x
    \nonumber\\
    = & d_n^{p-2}\int_{\{d_n|\nabla w_n|<|c_n \nabla U_{\lambda_n}|\}}|\nabla w_n|^{p} \mathrm{d}x
    +\int_{\{d_n|\nabla w_n|\geq |c_n \nabla U_{\lambda_n}|\}}|c_n \nabla U_{\lambda_n}|^{p-2}|\nabla w_n|^{2}\mathrm{d}x
    \nonumber\\
    \geq & d_n^{p-2}\int_{\{d_n|\nabla w_n|< |c_n \nabla U_{\lambda_n}|\}}|\nabla w_n|^{p} \mathrm{d}x
    + c\left(\int_{\{d_n|\nabla w_n|\geq|c_n \nabla U_{\lambda_n}|\}}|\nabla w_n|^{p} \mathrm{d}x\right)^{\frac{2}{p}}
    \nonumber\\
    \geq & c\left(\int_{\mathbb{R}^N}|\nabla w_n|^{p} \mathrm{d}x\right)^{\frac{2}{p}}=c,
    \end{align}
    for some constant $c>0$.
    The conclusion (\ref{rtnmb2}) follows immediately from \eqref{emn21} and \eqref{emn21b}.

    $\bullet$ {\em The case $\frac{2N}{N+2-\beta}< p<2$} which implies $p^*_{\beta}> 2$.

    The proof is very similar to the previous case, with very small changes. From Lemma \ref{lemui1p*l}, we have that for any $\kappa>0$, there exists a constant $\mathcal{C}_1=\mathcal{C}_1(p^*_{\beta},\kappa)>0$ such that
    \begin{align*}
    \|u_n\|_*^{p^*_\beta}
    = & \int_{\mathbb{R}^N}|x|^{-\beta}|u_n|^{p^*_{\beta}} \mathrm{d}x \\
    \leq & |c_n|^{p^*_{\beta}}
    \int_{\mathbb{R}^N}|x|^{-\beta}U_{\lambda_n}^{p^*_{\beta}} \mathrm{d}x
    +|c_n|^{p^*_{\beta}-2}c_n p^*_{\alpha} d_n \int_{\mathbb{R}^N}|x|^{-\beta}U_{\lambda_n}^{p^*_{\beta}-1}w_n \mathrm{d}x  \\
    & +\left(\frac{p^*_{\beta}(p^*_{\beta}-1)}{2}+\kappa\right)
    |c_n|^{p^*_{\beta}-2} d_n^2
    \int_{\mathbb{R}^N}|x|^{-\beta}U_{\lambda_n}^{p^*_{\beta}-2}w_n^2
    +\mathcal{C}_1d_n^{p^*_{\beta}}
    \int_{\mathbb{R}^N}|x|^{-\beta}|w_n|^{p^*_{\alpha,\beta}} \mathrm{d}x \\
    = & |c_n|^{p^*_{\beta}}\|U\|^p
    +\left(\frac{p^*_{\beta}(p^*_{\beta}-1)}{2}+\kappa\right)
    |c_n|^{p^*_{\beta}-2} d_n^2
    \int_{\mathbb{R}^N}|x|^{-\beta}
    U_{\lambda_n}^{p^*_{\beta}-2}w_n^2 \mathrm{d}x
    + o(d_n^2).
    \end{align*}
    Then by the concavity of $t\mapsto t^{\frac{p}{p^*_{\beta}}}$, we have
    \begin{align}\label{epkeyiyxbbg2}
    \|u_n\|_*^{p}
    \leq  &  |c_n|^p\|U\|^{\frac{p^2}{p^*_{\beta}}}
    + o(d_n^2)
    + \frac{p|c_n|^{p^*_{\beta}-2} d_n^2}{p^*_{\beta}}
    \left(\frac{p^*_{\beta}(p^*_{\beta}-1)}{2}+\kappa\right)
    \|U\|^{\frac{p^2}{p^*_{\beta}}-p}
    \int_{\mathbb{R}^N}|x|^{-\beta}U_{\lambda_n}^{p^*_{\beta}-2}w_n^2 \mathrm{d}x.
    \end{align}
    Hence, arguing as in the case $1<p<\frac{2N}{N+2-\beta}$,
    Therefore, as $d_n\to 0$, combining \eqref{epknug2} with \eqref{epkeyiyxbbg2}, it follows from Lemma \ref{lemsgap} that, by choosing $\kappa>0$ small enough,
    \begin{small}
    \begin{align*}
    \|u_n\|^p
    -\mathcal{S}_\beta\|u_n\|_*^p
    \geq & |c_n|^{p}\|U\|^p
    +\frac{(1-\kappa)p}{2} |c_n|^{p-2}d_n^2\int_{\mathbb{R}^N}|\nabla U_{\lambda_n}|^{p-2}  |\nabla w_n|^2\mathrm{d}x \\
    & +\frac{(1-\kappa)p(p-2)}{2}\int_{\mathbb{R}^N}|\omega(c_n \nabla U_{\lambda_n},\nabla u_n)|^{p-2}(|c_n \nabla U_{\lambda_n}|-|\nabla u_n|)^2  \mathrm{d}x\\
    & +\mathcal{C}_2d_n^{2} \int_{\mathbb{R}^N}\min\{d_n^{p-2}|\nabla w_n|^{p}, |c_n \nabla U_{\lambda_n}|^{p-2}|\nabla w_n|^{2}\} \mathrm{d}x \\
    & -\mathcal{S}_\beta\Bigg\{|c_n|^p\|U\|^{\frac{p^2}{p^*_{\beta}}}
    + o(d_n^2) \\
    & \quad \quad+ \frac{p|c_n|^{p^*_{\beta}-2} d_n^2}{p^*_{\beta}}
    \left(\frac{p^*_{\beta}(p^*_{\beta}-1)}{2}+\kappa\right)
    \|U\|^{\frac{p^2}{p^*_{\beta}}-p}
    \int_{\mathbb{R}^N}|x|^{-\beta}U_{\lambda_n}^{p^*_{\alpha,\beta}-2}
    w_n^2 \mathrm{d}x
    \Bigg\} \\
    \geq  & \frac{(1-\kappa)p}{2} |c_n|^{p-2}d_n^2\int_{\mathbb{R}^N}|\nabla U_{\lambda_n}|^{p-2}  |\nabla w_n|^2\mathrm{d}x \\
    & +\frac{(1-\kappa)p(p-2)}{2}\int_{\mathbb{R}^N}|\omega(c_n \nabla U_{\lambda_n},\nabla u_n)|^{p-2}(|c_n \nabla U_{\lambda_n}|-|\nabla u_n|)^2  \mathrm{d}x\\
    & +\mathcal{C}_2d_n^{2} \int_{\mathbb{R}^N}\min\{d_n^{p-2}|\nabla w_n|^{p}, |c_n \nabla U_{\lambda_n}|^{p-2}|\nabla w_n|^{2}\} \mathrm{d}x \\
    & -\left(\frac{p(p^*_{\beta}-1)}{2}
    +\frac{p\kappa}{p^*_{\beta}}\right)d_n^2
    \int_{\mathbb{R}^N}|x|^{-\beta}U_{\lambda_n}^{p^*_{\beta}-2}w_n^2 \mathrm{d}x
    - o(d_n^2) .
    \end{align*}
    \end{small}
    Lemma \ref{lemsgap2} allows us to reabsorb the last term above: more precisely, we have
    \begin{small}
    \begin{align*}
    & \|u_n\|^p
    -\mathcal{S}_\beta\|u_n\|_*^p
    \\ \geq & pd_n^2\left(\frac{(1-\kappa)}{2} - \frac{(p^*_{\alpha,\beta}-1)+\frac{2}{p^*_{\alpha,\beta}}\kappa}{2(p^*_{\alpha,\beta}-1)
    +2\tau }\right) \\
    & \quad \times
    \int_{\mathbb{R}^N}\left[|\nabla U_{\lambda_n}|^{p-2}  |\nabla w_n|^2
    +(p-2)|\omega(c_n \nabla U_{\lambda_n},\nabla u_n)|^{p-2}\left(\frac{|c_n \nabla U_{\lambda_n}|-|\nabla u_n|}{d_n}\right)^2\right]\mathrm{d}x \\
    & + d_n^{2}\left(\mathcal{C}_2 - \gamma_0\frac{p\left[(p^*_{\beta}-1)
    +\frac{2}{p^*_{\beta}}\kappa\right]}
    {2(p^*_{\beta}-1)+2\tau}\right)
    \int_{\mathbb{R}^N}\min\{d_n^{p-2}|\nabla w_n|^{p}, |c_n \nabla U_{\lambda_n}|^{p-2}|\nabla w_n|^{2}\} \mathrm{d}x
    \\ & - o(d_n^2) .
    \end{align*}
    \end{small}
    Now, let us recall the definition of $\omega$, as stated in Lemma \ref{lemui1p}, we have
    \[
    |\nabla c_nU_{\lambda_n}|^{p-2}  |\nabla w_n|^2
    +(p-2)|\omega(c_n \nabla U_{\lambda_n},\nabla u_n)|^{p-2}\left(\frac{|c_n \nabla U_{\lambda_n}|-|\nabla u_n|}{d_n}\right)^2\geq 0,
    \]
    then  choosing $\kappa>0$ small enough such that
    \[
    \frac{(1-\kappa)}{2} - \frac{(p^*_{\beta}-1)+\frac{2}{p^*_{\beta}}\kappa}{2(p^*_{\beta}-1)
    +2\tau}\geq 0,
    \]
    and then choosing $\gamma_0>0$ small enough such that
    \[
    \frac{\mathcal{C}_2}{2}\geq \gamma_0\frac{p\left[(p^*_{\beta}-1)+\frac{2}{p^*_{\beta}}\kappa\right]}
    {2(p^*_{\beta}-1)+2\tau}.
    \]
    From \eqref{emn21b}, we eventually arrive at
    \begin{small}
    \begin{align}\label{emn212}
    \|u_n\|^p
    -\mathcal{S}_\beta\|u_n\|_*^p
    \geq & \frac{\mathcal{C}_2}{2}d_n^{2}\int_{\mathbb{R}^N}|x|^{\alpha}\min\{d_n^{p-2}|\nabla w_n|^{p}, |c_n \nabla U_{\lambda_n}|^{p-2}|\nabla w_n|^{2}\} \mathrm{d}x
    - o(d_n^2) \nonumber \\
    \geq & cd_n^{2} ,
    \end{align}
    \end{small}
    for some constant $c>0$, thus the conclusion (\ref{rtnmb2}) follows immediately.
    \end{proof}

    Now, we are ready to prove our main result.
\subsection{\bf Proof of Theorem \ref{thmprtp}.}
  We argue by contradiction. In fact, if the theorem is false then there exists a sequence $\{u_n\}\subset \mathcal{D}^{1,p}_{0}(\mathbb{R}^N)\backslash \mathcal{M}_\beta$ such that
    \begin{equation*}
    \liminf_{n\to\infty}
    \frac{\|u_n\|^p
    -\mathcal{S}_\beta\|u_n\|_*^p}
    {d_n^\gamma}
    \to 0,\quad \mbox{as}\quad n\to \infty,
    \end{equation*}
    where $\gamma=\max\{2,p\}$.
    By homogeneity, we can assume that $\|u_n\|=1$, and after selecting a subsequence we can assume that $d_n\to \varpi\in[0,1]$ since $d_n=\inf_{c\in\mathbb{R}, \lambda>0}\|u_n-cU_{\lambda}\|\leq \|u_n\|$. If $\varpi=0$, then we deduce a contradiction by Lemmas  \ref{lemma:rtnm2b} and \ref{lemma:rtnm2b2}.

    The other possibility only is that $\varpi>0$, that is
    \[d_n=\inf_{c\in\mathbb{R}, \lambda>0}\|u_n-cU_{\lambda}\|\to \varpi>0\quad \mbox{as}\quad n\to \infty,\]
    then we must have
    \begin{equation}\label{wbsi}
    \|u_n\|^p
    -\mathcal{S}_\beta\|u_n\|_*^p\to 0,\quad \|u_n\|=1.
    \end{equation}
    By Lions' concentration and compactness principle (see \cite[Theorem 2.4]{Li85-2}),  going if necessary to a subsequence,
    we deduce that there exists a sequence of positive numbers  $\lambda_n>0$ such that
    \begin{equation*}
    \lambda_n^{\frac{N-p}{p}}u_n(\lambda_n x)\to U_*\quad \mbox{in}\quad \mathcal{D}^{1,p}_0(\mathbb{R}^N)\quad \mbox{as}\quad n\to \infty,
    \end{equation*}
    for some $U_*\in\mathcal{M}_\beta$, thus
    \begin{equation*}
    d_n\to 0 \quad \mbox{as}\quad n\to \infty,
    \end{equation*}
    which leads to a contradiction.

    Therefore, the proof of Theorem \ref{thmprtp} is complete.
    \qed

\noindent{\bfseries Acknowledgements}

The research has been supported by National Natural Science Foundation of China (No. 11971392).

\appendix

\section{Several crucial algebra inequalities}\label{sectpls}

\begin{lemma}\label{lemui1p}
    \cite[Lemmas 2.1]{FZ22} Let $x, y\in\mathbb{R}^N$. Then for any $\kappa>0$, there exists a constant $\mathcal{C}_1=\mathcal{C}_1(p,\kappa)>0$ such that
     the following inequalities hold.

    $\bullet$ For $1<p<2$,
    \begin{align}\label{uinb1p}
    |x+y|^p
    \geq & |x|^p+ p|x|^{p-2}x\cdot y+ \frac{1-\kappa}{2}\left(p|x|^{p-2}|y|^2+ p(p-2)|\omega|^{p-2}(|x|-|x+y|)^2 \right) \nonumber\\
    & +\mathcal{C}_1\min\{|y|^p,|x|^{p-2}|y|^2\},
    \end{align}
    where
    \begin{eqnarray*}
    \omega=\omega(x,x+y)=
    \left\{ \arraycolsep=1.5pt
       \begin{array}{ll}
        \left(\frac{|x+y|}{(2-p)|x+y|+(p-1)|x|}\right)^{\frac{1}{p-2}}x,\ \ &{\rm if}\ \ |x|<|x+y|,\\[3mm]
        x,\ \ &{\rm if}\ \  |x+y|\leq |x|.
        \end{array}
    \right.
    \end{eqnarray*}
    Furthermore, it is easy to verify that $|x|^{p-2}|y|^2+(p-2)|\omega|^{p-2}(|x|-|x+y|)^2\geq 0$.

    $\bullet$ For $p\geq 2$,
    \begin{align}\label{uinb2p}
    |x+y|^p
    \geq & |x|^p+ p|x|^{p-2}x\cdot y+ \frac{1-\kappa}{2}\left(p|x|^{p-2}|y|^2+ p(p-2)|\omega|^{p-2}(|x|-|x+y|)^2 \right) \nonumber\\
    & +\mathcal{C}_1 |y|^p ,
    \end{align}
    where
    \begin{eqnarray*}
    \omega=\omega(x,x+y)=
    \left\{ \arraycolsep=1.5pt
       \begin{array}{ll}
        x,\ \ &{\rm if}\ \ |x|<|x+y|,\\[3mm]
        \left(\frac{|x+y|}{|x|}\right)^{\frac{1}{p-2}}(x+y),\ \ &{\rm if}\ \  |x+y|\leq |x|.
        \end{array}
    \right.
    \end{eqnarray*}
    \end{lemma}

    \begin{lemma}\label{lemui1p*l}
    Let $a, b\in\mathbb{R}$. Then for any $\kappa>0$, there exists a constant $\mathcal{C}_2=\mathcal{C}_2(p^*_{\beta},\kappa)>0$ where $p^*_{\beta}=\frac{p(N-\beta)}{N-p}$ such that
     the following inequalities hold.

    $\bullet$ For $1<p\leq\frac{2N}{N+2-\beta}$,
    \begin{equation}\label{uinx2pl}
    |a+b|^{p^*_{\beta}}
    \leq |a|^{p^*_{\beta}}+ p^*_{\beta}|a|^{p^*_{\beta}-2}a b
    + \left(\frac{p^*_{\beta}(p^*_{\beta}-1)}{2}+\kappa\right)
    \frac{(|a|+\mathcal{C}_2|b|)^{{p^*_{\beta}}}|b|^2}
    {|a|^2+|b|^2}|b|^{2}.
    \end{equation}

    $\bullet$ For $\frac{2N}{N+2-\beta}< p<N$,
    \begin{equation}\label{uinx2pb}
    |a+b|^{p^*_{\beta}}
    \leq |a|^{p^*_{\beta}}+ p^*_{\beta}|a|^{p^*_{\beta}-2}a b
    + \left(\frac{p^*_{\beta}(p^*_{\beta}-1)}{2}+\kappa\right)
    |a|^{{p^*_{\beta}}-2}|b|^2 +\mathcal{C}_2|b|^{p^*_{\beta}}.
    \end{equation}
    \end{lemma}

    \begin{proof}
    We notice that $1<p\leq\frac{2N}{N+2-\beta}$ indicates $p^*_{\beta}\leq2$ and $\frac{2N}{N+2-\beta}< p<N$ indicates $p^*_{\beta}> 2$, then the above inequalities directly follow from \cite[Lemma 3.2]{FN19} and \cite[Lemma 2.4]{FZ22}.
    \end{proof}

\section{Some useful estimates}\label{appsue}

     The following lemmas mainly play roles in proving  Lemma \ref{propcetl}.
     Firstly, according the work of \cite{Sk14}, we establish the following Hardy-Poincar\'{e} type inequality:
    \begin{lemma}\label{lem:A-hpiw}
Let $1<p<N$, $0<\beta<p$  and $\xi\geq 1$. Then, for any compactly supported function $w\in \mathcal{D}^{1,p}_0(\mathbb{R}^N)$, one has
    \begin{align}\label{HPIw}
    \bar{C}_{N,p,\beta,\xi}\int_{\mathbb{R}^N}|w|^{p}
    |x|^{-\beta}\left[\left(1+|x|^{\frac{p-\beta}{p-1}}
    \right)^{p-1}\right]^{\xi-1}
    \mathrm{d}x
    \leq & \int_{\mathbb{R}^N}|\nabla w|^{p}
    \left[\left(1+|x|^{\frac{p-\beta}{p-1}}\right)^{p-1}\right]^{\xi}
    \mathrm{d}x,
    \end{align}
    for some constant $\bar{C}_{N,p,\beta,\xi}>0$.
    \end{lemma}

    \begin{proof}
    When $\xi=1$, \eqref{HPIw} reduces to classical Hardy inequality (see \cite[Theorem 4.1]{Sk14}), that is,
    \begin{align*}
    \int_{\mathbb{R}^N}|\nabla w|^{p}\left(1+|x|^{\frac{p-\beta}{p-1}}\right)^{p-1}
    \mathrm{d}x
    \geq \int_{\mathbb{R}^N}|\nabla w|^{p}|x|^{p-\beta}
    \mathrm{d}x
    \geq \left(\frac{N-\beta}{p}\right)^p
    \int_{\mathbb{R}^N}|w|^{p}
    |x|^{-\beta}
    \mathrm{d}x.
    \end{align*}
    So we only need to prove the case $\xi>1$.

    First we note that, by standard density argument, it suffices to prove \eqref{HPIw} for every compactly supported function $w\in C^1_{c,0}(\mathbb{R}^N)$.
    Indeed, let $w\in \mathcal{D}^{1,p}_0(\mathbb{R}^N)$ and
    \begin{eqnarray*}
    \phi(x)=
    \left\{ \arraycolsep=1.5pt
       \begin{array}{ll}
        1,\ \ &{\rm for}\ \ |x|<1,\\[2mm]
       -|x|+2,\ \ &{\rm for}\ \  1\leq |x|\leq 2,\\[2mm]
        0,\ \ &{\rm for}\ \ |x|> 2,
        \end{array}
    \right.
    \end{eqnarray*}
    and
    \begin{align*}
    \phi_R(x)=\phi\left(\frac{x}{R}\right),\quad w_R(x)=w(x)\phi_R(x).
    \end{align*}
    An easy verification shows that $w_R\to w$ in $\mathcal{D}^{1,p}_0(\mathbb{R}^N)$.  A standard convolution argument shows that every compactly supported function $w\in \mathcal{D}^{1,p}_0(\mathbb{R}^N)$ can be approximated in $ \mathcal{D}^{1,p}_0(\mathbb{R}^N)$ by compactly supported functions in $C^1_{c,0}(\mathbb{R}^N)$.

    Let us consider the function $u_\upsilon(x)=\left(1+|x|^{\frac{p-\beta}{p-1}}
    \right)^{-\upsilon}$ with $\upsilon>0$. Now the proof follows by steps.

    $\bullet$ {\em \color{blue} Step 1}. We recognize $u_\upsilon$ locally in $ \mathcal{D}^{1,p}_0(\mathbb{R}^N)$ and that it is a nonnegative solution to the following equation
    \begin{equation}\label{PpwhAe}
    -\mathrm{div}(|\nabla u_\upsilon|^{p-2}\nabla u_\upsilon)=
     |x|^{-\beta}d \left(1+|x|^{\frac{p-\beta}{p-1}}
    \right)^{\upsilon-\upsilon p-p} \left(1+k|x|^{\frac{p-\beta}{p-1}}
    \right)=:\Phi,
    \end{equation}
     a.e. in $\mathbb{R}^N$, where
    \begin{equation}\label{PpwhAc}
    d
    =\left(\frac{\upsilon(p-\beta)}{p-1}\right)^{p-1}(N-\beta),\quad \mbox{and}\quad
    k
    =1-\frac{(\upsilon+1)(p-\beta)}{N-\beta}.
    \end{equation}
    Moreover, it is easy to verify that $\Phi$ satisfies the following (same as  \cite[Definition 2.2]{Sk14}): let $\Omega$ be an open subset of $\mathbb{R}^N$, for every nonnegative compactly supported $v\in \mathcal{D}^{1,p}_0(\mathbb{R}^N)$ it holds that
    \[
    \int_{\Omega} \Phi v \mathrm{d}x>-\infty.
    \]

    $\bullet$ {\em \color{blue} Step 2}. Let us recall  a Hardy-type inequality \cite[Theorem 4.1]{Sk13}, we deduce that let $\varsigma$ and $\sigma$ be arbitrary numbers such that $\varsigma>0$ and $\varsigma>\sigma\geq \sigma_0$,  where
    \[
    \sigma_0:=-\inf\left\{\sigma\in\mathbb{R}: \Phi u_\upsilon+\sigma|\nabla u_\upsilon|^p\geq 0 \quad \mbox{a.e. in}\quad \{u_\upsilon>0\}\cap \Omega\right\}\in\mathbb{R},
    \]
    then for every Lipschitz function $w$ with compact support in $\Omega$, it holds that
    \begin{equation}\label{PpwhAcy}
    \int_{\Omega}|w|^p \mu_1(\mathrm{d}x)\leq \int_{\Omega}|\nabla w|^p \mu_2(\mathrm{d}x),
    \end{equation}
    where
    \begin{align}
    \mu_1(\mathrm{d}x)
    = & \left(\frac{\varsigma-\sigma}{p-1}\right)^{p-1}
    \left[\Phi u_\upsilon+\sigma |\nabla u_\upsilon|^p\right]u_\upsilon^{-\varsigma-1}\chi_{\{u_\upsilon>0\}}\mathrm{d}x,
    \label{mu1}
    \\ \mu_2(\mathrm{d}x)
    = & u_\upsilon^{p-\varsigma-1}\chi_{\{|\nabla u_\upsilon|\neq 0\}}\mathrm{d}x.
    \label{mu2}
    \end{align}
    Here $\chi_E$ denotes that $\chi_E=1$ if $x\in E$ and  $\chi_E=0$ if $x\notin E$.
    By direct calculation, we deduce  that
    \begin{align*}
    \sigma_0
    = & -\mathrm{ess}\inf\left(\frac{\Phi u_\upsilon}{|\nabla u_\upsilon|^p}\right)
    \\ = & -\inf\frac{\left(\frac{\upsilon(p-\beta)}{p-1}\right)^{p-1}(N-\beta)
    \left(1+|x|^{\frac{p-\beta}{p-1}}
    \right)^{\upsilon-\upsilon p-p} \left(1+\left(1-\frac{(\upsilon+1)(p-\beta)}{N-\beta}\right)|x|^{\frac{p-\beta}{p-1}}
    \right)|x|^{-\beta}}
    {\left(\frac{\upsilon(p-\beta)}{p-1}\right)^{p}\left(1+|x|^{\frac{p-\beta}{p-1}}
    \right)^{-(\upsilon+1)p}|x|^{\frac{(1-\beta)p}{p-1}}}
    \\ =  & -\inf \frac{(p-1)
    \left(1+|x|^{\frac{p-\beta}{p-1}}
    \right)^{\upsilon} \left(N-\beta+\left(N-\beta-(\upsilon+1)(p-\beta)\right)|x|^{\frac{p-\beta}{p-1}}
    \right)}
    {\upsilon(p-\beta)|x|^{\frac{p-\beta}{p-1}}}
    \\ =  & -\frac{(p-1)\left(N-\beta-(\upsilon+1)(p-\beta)\right)}{\upsilon(p-\beta)}.
    \end{align*}

    $\bullet$ {\em \color{blue} Step 3}. For given $\upsilon>-\xi$, define $\varsigma=(p-1)\left(\frac{\xi}{\upsilon}+1\right)$.  In order to apply inequality \eqref{PpwhAcy}, we require that $\varsigma>0$ and that $\sigma\in\mathbb{R}$ is such that $\varsigma>\sigma\geq \sigma_0$. This is equivalent to the condition $\xi>\max\{-\upsilon,\frac{p-N}{p-\beta}\}$, which obviously holds for all $\xi>1$, $\varsigma>0$.

    We are going to compute the measure $\mu_1(\mathrm{d}x)$ given by \eqref{mu1}. We note that $\xi=\upsilon\left(\frac{\varsigma}{p-1}-1\right)$ and $-p(\upsilon+1)+\upsilon(\varsigma+1)=(p-1)(\xi-1)-1$ and recall that $d$ and $k$ are given in \eqref{PpwhAc}. Applying these formulates to \eqref{mu1}, we obtain
    \begin{small}
    \begin{align}\label{mu1c}
    \mu_1(\mathrm{d}x)
    = & \left(\frac{\varsigma-\sigma}{p-1}\right)^{p-1}
    \left[\Phi u_\upsilon+\sigma |\nabla u_\upsilon|^p\right]u_\upsilon^{-\varsigma-1}\mathrm{d}x
    \nonumber\\
    = & \left(\frac{\varsigma-\sigma}{p-1}\right)^{p-1}
    \left[
    \frac{|x|^{-\beta}d \left(1+k|x|^{\frac{p-\beta}{p-1}}\right)}
    {\left(1+|x|^{\frac{p-\beta}{p-1}}\right)^{p(\upsilon+1)}}
    + \frac{\left(\frac{\upsilon(p-\beta)}{p-1}\right)^{p}\sigma |x|^{\frac{p(1-\beta)}{p-1}}}
    {\left(1+|x|^{\frac{p-\beta}{p-1}}\right)^{p(\upsilon+1)}}
    \right]
    \left(1+|x|^{\frac{p-\beta}{p-1}}\right)^{\upsilon(\varsigma+1)}
    \mathrm{d}x
    \nonumber\\
    = & |x|^{-\beta}\left(\frac{(\varsigma-\sigma)\upsilon (p-\beta)}{(p-1)^2}\right)^{p-1}
    \left\{N-\beta+\left[N-\beta-(\upsilon+1)(p-\beta)+\frac{\sigma \upsilon (p-\beta)}{p-1}
    \right]
    |x|^{\frac{p-\beta}{p-1}}\right\}
    \nonumber\\
    & \quad \times \left(1+|x|^{\frac{p-\beta}{p-1}}\right)^{-1}
    \left[\left(1+|x|^{\frac{p-\beta}{p-1}}\right)^{p-1}\right]^{\xi-1}
    \mathrm{d}x,
    \end{align}
    \end{small}
    with after substitution of $\varsigma=\frac{(p-1)(\upsilon+\xi)}{\upsilon}$, we obtain from \eqref{mu2} that
    \begin{align}
    \mu_2(\mathrm{d}x)
    = & u_\upsilon^{p-\varsigma-1}\chi_{\{|\nabla u_\upsilon|\neq 0\}}\mathrm{d}x
    = \left[\left(1+|x|^{\frac{p-\beta}{p-1}}\right)^{-\upsilon}\right]^{p-\varsigma-1}
    \mathrm{d}x
    = \left[\left(1+|x|^{\frac{p-\beta}{p-1}}\right)^{p-1}\right]^{\xi}\mathrm{d}x.
    \label{mu2c}
    \end{align}
    We choose $\sigma:=\frac{(p-1)(\upsilon+1)}{\upsilon}$ and realize that
    \[
    \varsigma =\frac{(p-1)(\upsilon+\xi)}{\upsilon}
    >\sigma>\sigma_0
    =\frac{(p-1)\left[(\upsilon+1)-(N-\beta)/(p-\beta)\right]}{\upsilon},
     \]
    since $\xi>1$ and $\max\{1,\beta\}<p<N$, then combining inequality \eqref{PpwhAcy} with \eqref{mu1c} and \eqref{mu2c}, we get our conclusion \eqref{HPIw} with
    \[
    \bar{C}_{N,p,\beta,\xi}
    =(N-\beta)\left(\frac{(\xi-1)(p-\beta)}{p-1}\right)^{p-1},\quad \mbox{if}\quad \xi>1.
    \]
    \end{proof}

\begin{lemma}\label{lem:A-ni}
Let $1<p\leq \frac{2N}{N+2-\beta}$ and $0<\beta<p$. Given $\varepsilon_0>0$, there exists $\eta=\eta(\varepsilon_0)>0$ small enough so that the following inequality holds for any nonnegative numbers $\varepsilon$, $r$, $a$,$b$ satisfying $\varepsilon\in (0,1)$ and $\varepsilon a \leq \zeta \left(1+r^{\frac{p-\beta}{p-1}}\right)^{-\frac{N-p}{p-\beta}}$:
\begin{align}
    & \left(1+r^{\frac{p-\beta}{p-1}}\right)^{-\frac{N-p}{p-\beta}(p^*_\beta-2)+p-1}
    \left[
    a^2\zeta^p r^{\frac{p(1-\beta)}{p-1}}\left(1+r^{\frac{p-\beta}{p-1}}\right)^{-p}
    +a^2 \varepsilon^p b^p \left(1+r^{\frac{p-\beta}{p-1}}\right)^{\frac{(N-p)p}{p-\beta}}
    +a^{2-p}b^p
    \right]
    \nonumber \\
    \leq & \varepsilon_0  r^{-\beta} \left(1+r^{\frac{p-\beta}{p-1}}\right)^{-\frac{N-p}{p-\beta}(p^*_\beta-2)}a^2
    +C (1+r)^{-\frac{p-\beta}{p-1}}
    \left(\left(1+r^{\frac{p-\beta}{p-1}}\right)^{-\frac{N-\beta}{p-\beta}}
    r^{\frac{1-\beta}{p-1}}+\varepsilon b\right)^{p-2}b^2
    \label{nif}
    \\
    \leq & \varepsilon_0    r^{-\beta} \left(1+r^{\frac{p-\beta}{p-1}}\right)^{-\frac{N-p}{p-\beta}(p^*_\beta-2)}a^2
    +C
    \left(\left(1+r^{\frac{p-\beta}{p-1}}\right)^{-\frac{N-\beta}{p-\beta}}
    r^{\frac{1-\beta}{p-1}}+\varepsilon b\right)^{p-2}b^2.
    \label{nis}
    \end{align}
\end{lemma}

\begin{proof}
We follow the arguments as those in \cite[Lemma B.1]{FZ22}. Note that \eqref{nis} immediately follows from \eqref{nif}, so it suffices to prove \eqref{nif}. When $r=0$, \eqref{nif} holds obviously since $0<\beta<p$. Then we distinguish several cases.

$\bullet$ {\em \color{blue} Case 1: $0< r\leq 1$}. In this case, up to changing the values of $\varepsilon_0$ and $\zeta$ by a universal constant, \eqref{nif} is equivalent to
\begin{align}\label{ni1e}
    a^2\zeta^p r^{\frac{p(1-\beta)}{p-1}}
    +a^2 \varepsilon^p b^p
    +a^{2-p}b^p
    \leq \varepsilon_0    r^{-\beta} a^2
    +C
    \left(r^{\frac{1-\beta}{p-1}}+\varepsilon b\right)^{p-2}b^2.
    \end{align}
    Note that:

    \noindent{-if } $\varepsilon b\leq \left(\frac{\varepsilon_0}{3}\right)^{\frac{1}{p}}r^{\frac{1-\beta}{p-1}}$
    then $a^2\varepsilon^p b^p\leq \frac{\varepsilon_0}{3}r^{-\beta} a^2$ since $\beta<p$;

    \noindent{-if } $\varepsilon b> \left(\frac{\varepsilon_0}{3}\right)^{\frac{1}{p}}r^{\frac{1-\beta}{p-1}}$
    then, since $\varepsilon a \leq \zeta \left(1+r^{\frac{p-\beta}{p-1}}\right)^{-\frac{N-p}{p-\beta}}\leq 2\zeta$,
    \[
    a^2 \varepsilon^p b^p \leq 4\zeta^2\varepsilon^{p-2} b^p
    \leq C
    \left(r^{\frac{1-\beta}{p-1}}+\varepsilon b\right)^{p-2}b^2.
    \]
    Similarly:

    \noindent{-if } $b\leq \left(\frac{\varepsilon_0}{3}\right)^{\frac{1}{p}}a r^{-\frac{\beta}{p}}$
    then $a^{2-p} b^p\leq \frac{\varepsilon_0}{3}r^{-\beta} a^2$;

    \noindent{-if } $\left(\frac{\varepsilon_0}{3}\right)^{\frac{1}{p}}a r^{-\frac{\beta}{p}}
    <b<\varepsilon^{-1}r^{\frac{1-\beta}{p-1}}$ then
    \[
    a^{2-p} b^p
    \leq \left(\frac{\varepsilon_0}{3}\right)^{-\frac{2-p}{p}} r^{\frac{\beta(2-p)}{p}} b^2
    \leq C r^{\frac{(1-\beta)(p-2)}{p-1}}b^2
    \leq C
    \left(r^{\frac{1-\beta}{p-1}}+\varepsilon b\right)^{p-2}b^2;
    \]

    \noindent{-if } $b\geq \varepsilon^{-1}r^{\frac{1-\beta}{p-1}}$
    then, since $\varepsilon a \leq \zeta \left(1+r^{\frac{p-\beta}{p-1}}\right)^{-\frac{N-p}{p-\beta}}\leq 2\zeta$,
    \[
    a^{2-p} b^p
    \leq 2^{2-p}\zeta^{2-p}\varepsilon^{p-2} b^p
    \leq 4^{2-p}\zeta^{2-p}
    \left(r^{\frac{1-\beta}{p-1}}+\varepsilon b\right)^{p-2}b^2.
    \]
    Thus, choosing $\zeta^p\leq \frac{\varepsilon_0}{3}$, \eqref{ni1e} holds in all cases.

    $\bullet$ {\em \color{blue} Case 2: $ r>1$}. In this case, up to changing the values of $\varepsilon_0$ and $\zeta$ by a universal constant, \eqref{nif} is equivalent to
    \begin{align}\label{nise}
    & r^{\frac{(p-N)(p^*_\beta-2)}{p-1}-\beta}
    a^2\zeta^p
    +r^{\frac{(p-N)(p^*_\beta-2-p)}{p-1}+p-\beta}a^2 \varepsilon^p b^p
    +r^{\frac{(p-N)(p^*_\beta-2)}{p-1}+p-\beta}a^{2-p}b^p
    \nonumber \\
    \leq & \varepsilon_0  r^{\frac{(p-N)(p^*_\beta-2)}{p-1}-\beta}a^2
    +C r^{-\frac{p-\beta}{p-1}}
    \left(r^{\frac{1-N}{p-1}}+\varepsilon b\right)^{p-2}b^2.
    \end{align}
    Again:

    \noindent{-if } $\varepsilon b\leq \left(\frac{\varepsilon_0}{3}\right)^{\frac{1}{p}}r^{\frac{1-N}{p-1}}$
    then
    \[
    r^{\frac{(p-N)(p^*_\beta-2-p)}{p-1}+p-\beta}a^2 \varepsilon^p b^p
    \leq \frac{\varepsilon_0}{3}  r^{\frac{(p-N)(p^*_\beta-2)}{p-1}-\beta}a^2;
    \]

    \noindent{-if } $\varepsilon b> \left(\frac{\varepsilon_0}{3}\right)^{\frac{1}{p}}r^{\frac{1-N}{p-1}}$,
    then we apply the inequality $\varepsilon a \leq \zeta \left(1+r^{\frac{p-\beta}{p-1}}\right)^{-\frac{N-p}{p-\beta}}\leq 2\zeta r^{\frac{p-N}{p-1}}$ to conclude that
    \[
    r^{\frac{(p-N)(p^*_\beta-2-p)}{p-1}+p-\beta}a^2 \varepsilon^p b^p
    \leq 4  r^{-\frac{p-\beta}{p-1}} \zeta^2 (\varepsilon b)^{p-2} b^2
    \leq C r^{-\frac{p-\beta}{p-1}}
    \left(r^{\frac{1-N}{p-1}}+\varepsilon b\right)^{p-2}b^2.
    \]
    On the other hand:

    \noindent{-if } $b\leq \left(\frac{\varepsilon_0}{3}\right)^{\frac{1}{p}}a r^{-1}$
    then
    \[
    r^{\frac{(p-N)(p^*_\beta-2)}{p-1}+p-\beta}a^{2-p}b^p
    \leq \frac{\varepsilon_0}{3}  r^{\frac{(p-N)(p^*_\beta-2)}{p-1}-\beta}a^2;
    \]
    \noindent{-if } $\left(\frac{\varepsilon_0}{3}\right)^{\frac{1}{p}}a r^{-1}
    <b<\varepsilon^{-1}r^{\frac{1-N}{p-1}}$ then
    \begin{align*}
    r^{\frac{(p-N)(p^*_\beta-2)}{p-1}+p-\beta}a^{2-p}b^p
     \leq & Cr^{\frac{(p-N)(p^*_\beta-2)}{p-1}+2-\beta}b^2
     \\ = & C r^{-\frac{p-\beta}{p-1}}r^{\frac{(1-N)(p-2)}{p-1}}b^2
    \leq C
    r^{-\frac{p-\beta}{p-1}}
    \left(r^{\frac{1-N}{p-1}}+\varepsilon b\right)^{p-2}b^2;
    \end{align*}

    \noindent{-if } $b\geq \varepsilon^{-1}r^{\frac{1-N}{p-1}}$
    then, since $\varepsilon a \leq \zeta \left(1+r^{\frac{p-\beta}{p-1}}\right)^{-\frac{N-p}{p-\beta}}\leq 2\zeta r^{\frac{p-N}{p-1}}$,
    \[
    r^{\frac{(p-N)(p^*_\beta-2)}{p-1}+p-\beta}a^{2-p}b^p
    \leq 2^{2-p}\zeta^{2-p}r^{-\frac{p-\beta}{p-1}}
    (\varepsilon b)^{p-2} b^2
    \leq C
    r^{-\frac{p-\beta}{p-1}}
    \left(r^{\frac{1-N}{p-1}}+\varepsilon b\right)^{p-2}b^2.
    \]
    This proves \eqref{nise} whenever $\zeta^p\leq \frac{\varepsilon_0}{3}$, concluding the proof of \eqref{nif}.
    \end{proof}

\subsection*{Acknowledgments}
The authors have been supported by National Natural Science Foundation of China 11971392.


\begin{thebibliography}{99}

\bibitem{Au76}
Aubin, T.: {\em Probl\`{e}mes isop\'{e}rim\'{e}triques et espaces de Sobolev}. J. Differ. Geom. {\bf 11}, 573--598 (1976)

 
\bibitem{ACP05}
Abdellaoui, B., Colorado, E., Peral, I.: {\em Some improved Caffarelli-Kohn-Nirenberg inequalities}. Calc. Var. Partial Differential Equations {\bf 23}(3), 327--345 (2005)

 
\bibitem{BE91}
Bianchi, G., Egnell, H.: {\em A note on the Sobolev inequality}. J. Funct. Anal. {\bf 100}(1), 18--24 (1991)

\bibitem{BrL85}
Brezis H., Lieb, E.: {\em Sobolev inequalities with remainder terms}. J. Funct. Anal. {\bf 62}, 73--86 (1985)

\bibitem{BWW03}
Bartsch, T., Weth T., Willem, M.: {\em A Sobolev inequality with remainder term and critical equations on domains with topology for the polyharmonic operator}. Calc. Var. Partial Dif. {\bf 18}, 253--268 (2003)

 
\bibitem{CFW13}
Chen, S., Frank, R.L., Weth, T.: {\em Remainder terms in the fractional Sobolev inequality}. Indiana Univ. Math. J. {\bf 62}(4), 1381--1397 (2013)

\bibitem{CFMP09}
Cianchi, A., Fusco, N., Maggi, F., Pratelli, A.: {\em The sharp Sobolev inequality in quantitative form}. J. Eur. Math. Soc. (JEMS) {\bf 11}(5), 1105--1139 (2009)


\bibitem{CKN84}
Caffarelli, L., Kohn R., Nirenberg, L.: {\em First order interpolation inequalities with weights}. Compos. Math. {\bf 53}, 259--275 (1984)

\bibitem{CGS89}
Caffarelli, L., Gidas, B., Spruck, J.: {\em  Asymptotic symmetry and local behavior of semilinear elliptic equations with critical Sobolev growth}. Comm. Pure Appl. Math. {\bf 42}(3), 271--297 (1989)

\bibitem{CMR22}
Catino, G., Monticelli, D., Roncoroni, A.: {\em On the critical $p$-Laplace equation}. Preprint. \url{
https://doi.org/10.48550/arXiv.2204.06940} [math.AP]

 
\bibitem{DEFS22}
Dolbeault, J., Esteban, M. J., Figalli, A., Frank, R., Loss, M.: {\em Stability for the Sobolev inequality with explicit constants}. Preprint. \url{https://doi.org/10.48550/arXiv.2209.08651}

\bibitem{DMMS14}
Damascelli, L., Merch\'{a}n, S., Montoro, L.,  Sciunzi, B.: {\em  Radial symmetry and applications for a problem involving the $-\Delta_p (\cdot)$  operator and critical nonlinearity in $\mathbb{R}^N$}. Adv. Math. {\bf 256}, 313--335 (2014)

\bibitem{DR01}
Damascelli, L., Ramaswamy, M.: {\em  Symmetry of $C^1$-solutions of $p$-Laplace equations in $\mathbb{R}^N$}. Adv. Nonlinear Stud. {\bf 1}, 40--64 (2001)

\bibitem{DT23}
Deng, S., Tian, X.: {\em Some weighted fourth-order Hardy-H\'{e}non equations}. J. Funct. Anal. {\bf 284}(1), Paper No. 109745 (2023)

 
\bibitem{FN19}
Figalli, A., Neumayer, R.: {\em Gradient stability for the Sobolev inequality: the case $p\geq 2$}. J. Eur. Math. Soc. (JEMS) {\bf 21}(2), 319--354 (2019)

\bibitem{FZ22}
Figalli, A., Zhang Y.: {\em Sharp gradient stability for the Sobolev inequality}. Duke Math. J. {\bf 171}(12), 2407--2459 (2022)

\bibitem{GY00}
Ghoussoub, N., Yuan, C.: {\em Multiple solutions for quasi-linear PDEs involving the critical Sobolev and Hardy exponents}. Trans. Amer. Math. Soc. {\bf 352}(12), 5703--5743 (2000)

\bibitem{GV88}
Guedda, M., V\'{e}ron, L.: {\em Local and global properties of solutions of quasilinear elliptic equations}. J. Differ. Equ. {\bf 76}, 159--189 (1988)

\bibitem{Ko22}
K\"{o}nig, T.: {\em On the sharp constant in the Bianchi-Egnell stability inequality}. Preprint. \url{https://doi.org/10.48550/arXiv.2210.08482}

\bibitem{Ko22-2}
K\"{o}nig, T.: {\em Stability for the Sobolev inequality: existence of a minimizer}. Preprint. \url{https://doi.org/10.48550/arXiv.2211.14185}

\bibitem{Li85-2}
Lions, P.L.: {\em The concentration-compactness principle in the calculus of variations. The limit case. \uppercase\expandafter{\romannumeral 2}}.  Rev. Mat. Iberam. {\bf 1}(2), 45--121 (1985)

\bibitem{LW00}
Lu, G., Wei, J.: {\em  On a Sobolev inequality with remainder terms}. Proc. Amer. Math. Soc. {\bf 128}(1), 75--84 (2000)

\bibitem{Ma85}
Maz'ja, V. G.: {\em Sobolev spaces}. Translated from the Russian by T. O. Shaposhnikova. Springer Series in Soviet Mathematics. Springer-Verlag, Berlin, 1985.

 
\bibitem{Ne20}
Neumayer, R.: {\em A note on strong-form stability for the Sobolev inequality}. Calc. Var. Partial Differential Equations {\bf 59}(1), Paper No. 25 (2020)

\bibitem{Ou22}
Ou, Q: {\em On the classification of entire solutions to the critical $p$-Laplace equation}. Preprint. \url{
https://doi.org/10.48550/arXiv.2210.05141}

\bibitem{PV21}
Pistoia, A., Vaira, G.: {\em Nondegeneracy of the bubble for the critical p-Laplace equation}. Proc. Roy. Soc. Edinburgh Sect. A {\bf 151}(1), 151--168 (2021)

\bibitem{RSW02}
R\u{a}dulescu, V., Smets, D., Willem, M.: {\em Hardy-Sobolev inequalities with remainder terms}. Topol. Methods Nonlinear Anal. {\bf 20}(1), 145--149 (2002)

\bibitem{Sc16}
Sciunzi, B.: {\em Classification of positive $\mathcal{D}^{1,p}(\mathbb{R}^N)$-solutions to the critical $p$-Laplace equation in $\mathbb{R}^N$}. Adv. Math. {\bf 291}, 12--23 (2016)

\bibitem{Sk13}
Skrzypczak, I.: {\em  Hardy-type inequalities derived from p-harmonic problems}. Nonlinear Anal. {\bf 93}, 30--50 (2013)

\bibitem{Sk14}
Skrzypczak, I.: {\em Hardy-Poincar\'{e} type inequalities derived from p-harmonic problems}. Calculus of variations and PDEs, 225--238, Banach Center Publ., 101, Polish Acad. Sci. Inst. Math., Warsaw (2014)

\bibitem{ST18}
Sano, M., Takahashi, F.: {\em Some improvements for a class of the Caffarelli-Kohn-Nirenberg inequalities}. Differ. Integral Equ. {\bf 31}(1-2), 57--74 (2018)

\bibitem{SmWi03}
Smets, D., Willem, M.: {\em Partial symmetry and asymptotic behavior for some elliptic variational problems}. Calc. Var. Partial Differential Equations {\bf 18}(1), 57--75 (2003)

\bibitem{Ta76}
Talenti, G.: {\em Best constant in Sobolev inequality}. Ann. Mat. Pura Appl. {\bf 110}, 353--372 (1976)

\bibitem{Ve16}
V\'{e}tois, J.: {\em A priori estimates and application to the symmetry of solutions for critical $p$-Laplace
equations}. J. Diff. Equ. {\bf 260}(1), 149--161 (2016)

 
\bibitem{WaWi03}
Wang, Z.-Q., Willem, M.: {\em Caffarelli-Kohn-Nirenberg inequalities with remainder terms}. J. Funct. Anal. {\bf 203}(2), 550--568 (2003)

\bibitem{WW22}
Wei, J., Wu, Y.: {\em On the stability of the Caffarelli-Kohn-Nirenberg inequality}. Math. Ann. {\bf 384}, no. 3-4, 1509--1546 (2022)



\end{thebibliography}
    \end{document}